\documentclass[reqno]{amsart}

\newtheorem{theorem}{Theorem}[section]
\newtheorem{lemma}[theorem]{Lemma}
\newtheorem{proposition}[theorem]{Proposition}

\newtheorem{corollary}[theorem]{Corollary}
\newtheorem*{theorem*}{Theorem}

\theoremstyle{definition}
\newtheorem{definition}[theorem]{Definition}
\newtheorem{notation}[theorem]{Notation}

\theoremstyle{remark}
\newtheorem{remark}[theorem]{Remark}

\usepackage{amsthm}
\theoremstyle{plain} 
\newcommand{\thistheoremname}{}
\newtheorem*{genericthm}{\thistheoremname}

\numberwithin{equation}{section}

\usepackage{thmtools}
\usepackage{thm-restate}
\usepackage{graphics}
\usepackage{mathtools}
\usepackage{bbm}
\usepackage[utf8]{inputenc} 
\usepackage[T1]{fontenc}
\usepackage{xcolor}
\usepackage{float}
\usepackage{amssymb}
\usepackage{srcltx}
\usepackage{enumerate}
\usepackage{esint}
\usepackage[shortlabels]{enumitem}
\usepackage{textcomp}
\usepackage{gensymb}

\usepackage{mathrsfs}

\usepackage[colorlinks,plainpages=true,pdfpagelabels,hypertexnames=true,colorlinks=true,pdfstartview=FitV,linkcolor=blue,citecolor=red,urlcolor=black, linktocpage=true]{hyperref}
\usepackage{cleveref}

\usepackage{tikz-cd}

\newtheorem{claim}[theorem]{Claim}

\everymath{\displaystyle}

\newcommand{\abs}[1]{\left\lvert#1\right\rvert}

\newcommand{\defeq}{\vcentcolon=}
\DeclarePairedDelimiterX{\inp}[2]{\langle}{\rangle}{#1, #2}

\newcommand\boldbox[1]{\mbox{\boldmath{\boxed{\ensuremath #1}} \unboldmath}}

\DeclareMathOperator{\st}{{\bf st}}

\DeclareMathOperator{\ns}{{\bf Ns}}

\renewcommand{\restriction}{\mathord{\upharpoonright}}

\DeclareMathOperator{\tprs}{\tau_{\mathfrak{P}_r(S)}}

\DeclareMathOperator{\ps}{\mathfrak{P}(S)}

\DeclareMathOperator{\pt}{\mathfrak{P}(T)}

\DeclareMathOperator{\prt}{\mathfrak{P}_r(T)}
\DeclareMathOperator{\prs}{\mathfrak{P}_r(S)}
\DeclareMathOperator{\prsk}{\mathfrak{P}_r(S^k)}
\DeclareMathOperator{\prsinf}{\mathfrak{P}_r(S^{\infty})}

\DeclareMathOperator{\bas}{\mathcal{B}_a(S)}
\DeclareMathOperator{\bask}{\mathcal{B}_a(S^k)}
\DeclareMathOperator{\pbas}{\mathfrak{P}_{\text{Ba}}(S)}

\DeclareMathOperator{\pbasinf}{\mathfrak{P}_{\text{Ba}}(S^{\infty})}

\DeclareMathOperator{\sigmabprs}{\sigma(\mathbb{B}(\prs))}
\DeclareMathOperator{\bprs}{\mathbb{B}(\prs)}

\DeclareMathOperator{\finunbprs}{\left\{\cup_{A \in \mathscr{F}} A: \mathscr{F} \in \mathcal{P}_{\text{fin}}(\mathbb{B}(\prs))\right\}}

\newcommand{\starint}{{\prescript{\ast}{}\int}}
\newcommand{\co}{\colon}

\usepackage{graphicx}
\usepackage{amsmath}
\usepackage{tikz}
\usepackage{pgfplots}
\graphicspath{ {images/} }

\usepackage{tikz}
\usepackage{tikz-3dplot}
\usepackage{pgfplots}

\pgfplotsset{%
    compat=1.8,
    compat/show suggested version=false,
}

\usetikzlibrary{calc,fadings,decorations.pathreplacing}
\usepackage[many]{tcolorbox}
\usepackage{wrapfig}

\newcommand{\mbp}{{\mathbb P}}

\usepackage{environ}

\makeatletter
\long\def\@secondofthree#1#2#3{#2}
\long\def\@thirdoffour#1#2#3#4{#3}

\NewEnviron{restatethis}[2]{%
  \begin{#1}%
      \BODY
      \protected@write\@auxout{}{%
        \string\@restatetheorem{#1}{#2}{\csname the#1\endcsname}{\detokenize\expandafter{\BODY}}%
      }%
  \end{#1}%
}

\@ifpackageloaded{thmtools}{%
    \def\restatethm@getthmcountercsname#1{\def\thethmcsname{#1}}%
}{%
    \@ifpackageloaded{ntheorem}{%
        \def\restatethm@getthmcountercsname#1{%
            \def\thethmcsname{\expandafter\expandafter\expandafter\restatethm@ntheorem@getthmcountercsname@helper\csname mkheader@#1\endcsname}}%
        \def\restatethm@ntheorem@getthmcountercsname@helper#1\@thm#2#3#4{#3}
    }{%
        \@ifpackageloaded{amsthm}{%
            \def\restatethm@getthmcountercsname#1{\edef\thethmcsname{\expandafter\expandafter\expandafter\@thirdoffour\csname#1\endcsname}}%
        }{%
            \def\restatethm@getthmcountercsname#1{\edef\thethmcsname{\expandafter\expandafter\expandafter\@secondofthree\csname#1\endcsname}}%
        }%
    }%
}

\newcommand{\@restatetheorem}[4]{%
  \expandafter\gdef\csname restatethis@#2\endcsname{%
    \begingroup
    \restatethm@getthmcountercsname{#1}
    \expandafter\def\csname the\thethmcsname\endcsname{#3}%
    \begin{#1}#4\end{#1}%
    \endgroup
  }%
}
\newcommand{\restate}[1]{\csname restatethis@#1\endcsname} 
\makeatother

\usepackage{scalerel,stackengine}
\stackMath
\newcommand\reallywidehat[1]{%
\savestack{\tmpbox}{\stretchto{%
  \scaleto{%
    \scalerel*[\widthof{\ensuremath{#1}}]{\kern-.6pt\bigwedge\kern-.6pt}%
    {\rule[-\textheight/2]{1ex}{\textheight}}
  }{\textheight}%
}{0.5ex}}%
\stackon[1pt]{#1}{\tmpbox}%
}
\newtheorem{innercustomthm}{Theorem}
\newenvironment{customthm}[1]
  {\renewcommand\theinnercustomthm{#1}\innercustomthm}
  {\endinnercustomthm}

\makeatletter
\@namedef{subjclassname@2020}{%
  \textup{2020} Mathematics Subject Classification}
\makeatother

\begin{document}

\title[Generalizing the de Finetti--Hewitt--Savage theorem]{Generalizing the de Finetti--Hewitt--Savage theorem}
\author{Irfan Alam}
\address{Department of Computer and Mathematical Sciences, University of Toronto Scarborough, 1050 Military Trail, Toronto, ON, CANADA M1C 1A4}
\email{irfanalamisi@gmail.com}
\urladdr{\href{https://sites.google.com/view/irfan-alam/}{https://sites.google.com/view/irfan-alam/}}


\date{\today}

\begin{abstract}
A sequence of random variables is called \textit{exchangeable} if its joint distribution is invariant under permutations of indices. de Finetti's theorem is a fundamental result on such random variables that shaped the foundations of Bayesian statistics over the last century. The original formulation of de Finetti's theorem roughly says that any exchangeable sequence of $\{0,1\}$-valued random variables can be thought of as a mixture of independent and identically distributed sequences in a certain precise mathematical sense. Interpreting this statement from a convex analytic perspective, Hewitt and Savage were able to obtain the same conclusion for exchangeable sequences of random variables taking values in more general state spaces under some topological conditions.

This manuscript is an expanded version of a shorter paper in which de Finetti--Hewitt--Savage Theorem was recently generalized. More precisely, using tools from nonstandard analysis we prove that an exchangeable sequence of Radon-distributed random variables taking values in any Hausdorff state space must be representable as a mixture of sequences of independent and identically distributed random variables. 

Our presentation of this work follows the style of \textit{lecture notes} intended for broad graduate-level mathematical audiences --- the main body of the manuscript starts with a historically grounded introduction to the problem, foreshadowing our techniques that are developed via a series of appendices. These techniques are used to provide self-contained proofs of our main results in a short section following the introduction. As an immediate corollary of our main result, it is consistent with ZFC to assume that de Finetti--Hewitt--Savage theorem is true for all completely metrizable state spaces. Furthermore, we also obtain a de Finetti-style theorem for exchangeable sequences of random variables with a tight, but possibly non-Radon, marginal distribution. 

We have provided a self-contained philosophically motivated introduction to nonstandard analysis in the first appendix, thus rendering first courses in measure theoretic probability and point-set topology as the only prerequisites for the work. This introduction aims to develop some new ideologies about the subject that might be of interest to mathematicians, philosophers, and mathematics educators alike. One highlight of the rest of the appendices is a new generalization of Prokhorov's theorem in the setting of the space of all probability measures on arbitrary Hausdorff spaces. 
\end{abstract}

\maketitle

\tableofcontents

\vspace{-1.2869cm}
\section{Introduction}


Let us fix a probability space $(\Omega, \mathcal{F}, \mathbb{P})$, and consider a sequence of random variables $(X_n)_{n \in \mathbb{N}}$ taking values in an arbitrary set $S$ (called \textit{state space}) equipped with a sigma algebra $\mathcal{S}$. This sequence is called \textit{exchangeable} if its finite-dimensional marginal distributions are invariant under permutations of indices. 

In the 1920s, Bruno de Finetti \cite{definetti1} proved the following characterization of exchangeable sequences of Bernoulli random variables, which is the case when $S = \{0,1\}$. 

\begin{theorem}[de Finetti]\label{de finetti theorem}
Let $(X_n)_{n \in \mathbb{N}}$ be an exchangeable sequence of Bernoulli random variables. Then there exists a unique Borel probability measure $\nu$ on the interval $[0,1]$ such that the following holds:
\begin{align}\label{de Finetti equation}
    \mathbb{P}(X_1 = e_1, \ldots, X_k = e_k) = \int_{[0,1]} p^{\sum_{j = 1}^k e_j}(1 - p)^{k - \sum_{j = 1}^k e_j} d\nu(p)
\end{align}
for any $k \in \mathbb{N}$ and $e_1, \ldots, e_k \in \{0,1\}$.
\end{theorem}

Notice that in equation \eqref{de Finetti equation}, the variable of integration, $p$, can be identified with the measure induced on $S = \{0,1\}$ by a coin toss for which the chance of success (with success identified with the state $1$) is $p$. Clearly, all probability measures on the discrete set $\{0,1\}$ are of this form.  Thus, $\nu$ in \eqref{de Finetti equation} can be thought of as a measure on the set of all probability measures on $S$. The integrand in \eqref{de Finetti equation} then represents the probability of obtaining $\textstyle{\sum_{j = 1}^k e_j}$ successes in $k$ independent coin tosses, while the integral represents the expected value of this probability.

With this perspective, it becomes possible to look for generalizations of Theorem \ref{de finetti theorem} in situations when the exchangeable random variables need not be binary valued. Through independent works of de Finetti \cite{definetti2} and Dynkin \cite{Dynkin-exchangeable}, it was known by the early 1950s that the state space $S$ could be taken to be the set $\mathbb{R}$ of real numbers, equipped with its Borel sigma algebra $\mathcal{B}(\mathbb{R})$\footnote{Here, we are considering $\mathbb{R}$ as a topological space under the usual topology. More generally, given a topological space $T$, we shall use the notation $\mathcal{B}(T)$ for the \textit{Borel sigma algebra over 
$T$}, which is the smallest sigma algebra containing all open subsets of $T$.}. 

In 1955, Hewitt and Savage \cite{Hewitt-Savage-1955} published a further generalization, showing that de Finetti's theorem was still true if the state space was an arbitrary compact Hausdorff space equipped with its Baire sigma algebra. In general, they called a measurable space $(S, \mathcal{S})$ \textit{presentable} if exchangeable sequences of $S$-valued random variables satisfied the conclusion of de Finetti's theorem. More concretely: 

\begin{definition}
   For a measurable space $(S, \mathcal{S})$, let $\mathfrak{P}(S)$ be the set of all probability measures on it. A measurable space $(S, \mathcal{S})$ is called \textit{presentable} if for each exchangeable sequence $(X_n)$ of $S$-valued random variables, there exists a unique probability measure $\mathscr{P}$ on $\ps$ such that the following holds for all $k \in \mathbb{N}$:
\begin{align}
    \mathbb{P}(X_1 \in B_1, \ldots, X_k \in B_k) = \int\limits_{\ps} \mu(B_1)\cdot \ldots \cdot \mu(B_k) d\mathscr{P}(\mu) \text{ for all } B_1, \ldots, B_k \in \mathcal{S}, \label{naive generalization}
\end{align}
where the set $\ps$ is equipped with a sigma algebra with respect to which the above integrals make sense.
\end{definition}

Since we want to integrate functions of the type $\mu \mapsto \mu(B)$ on $\ps$, where $B \in \mathcal{S}$, the smallest sigma algebra ensuring the measurability of all such functions is appropriate for this discussion. That minimal sigma algebra, which we denote by $\mathcal{C}(\ps)$, is generated by the so-called \textit{cylinder sets}. These are sets of the type $$\{\mu \in \ps: \mu(B_1) \in A_1, \ldots, \mu(B_k) \in A_k\},$$ where $k \in \mathbb{N}$; $B_1, \ldots, B_k \in \mathcal{S}$; and $A_1, \ldots, A_k \in \mathcal{B}(\mathbb{R})$.

By the early 1960s, Varadarajan \cite[p. 219]{Varadarajan-Borel} had observed that the result of Hewitt and Savage implied also that all \textit{analytic} state spaces are presentable. Here an \textit{analytic space} refers to a measurable space that is isomorphic to $(T, \mathcal{B}(T))$, where $T$ is a subset of a Polish space that can be realized as a continuous image of a Borel subset of a (possibly different) Polish space. Thus, in particular, all Polish spaces equipped with their Borel sigma algebras are presentable, further generalizing the previous works of de Finetti and Dynkin on the presentability of $(\mathbb{R}, \mathcal{B}(\mathbb{R}))$. 

The above observation of Varadarajan is the state of the art for modern treatments of de Finetti's theorem for Borel sigma algebras on topological state spaces. For example, Diaconis and Freedman \cite[Theorem 14, p. 750]{Diaconis-Freedman-Finite} reproved the result of Hewitt and Savage using their approximate de Finetti's theorem for finite exchangeable sequences in any state space (wherein they exploited the topological structure on the state space to be able to take the limit to go from their more general approximate de Finetti's theorem on finite exchangeable sequences to the theorem of Hewitt and Savage). They then concluded (see \cite[p. 751]{Diaconis-Freedman-Finite}) that de Finetti's theorem holds for state spaces that are isomorphic to Borel subsets of a Polish space. Since any Borel subset of a Polish space is also analytic, this observation is a special case of Varadarajan's. In his monograph, Kallenberg \cite[Theorem 1.1]{Kallenberg-symmetries} has a proof of de Finetti's theorem for any state space that is isomorphic to a Borel subset of the closed interval $[0,1]$, a formulation that is contained in the above. 

Recall that a Polish space is a complete separable metric space. By the late 1970s, Dubins and Freedman \cite{Dubins-Freedman} had shown that separability by itself was not sufficient for presentability, as they were able to construct a separable non-complete metric space $S$ and an exchangeable sequence of $S$-valued random variables that failed to satisfy the conclusion of de Finetti's theorem. This negatively settled a question asked by Hewitt and Savage, who were wondering whether all measurable spaces are presentable. 

Since all generalizations of de Finetti's theorem up to this point could be deduced from the work of Hewitt and Savage, a generalization of de Finetti's theorem to new state spaces is sometimes also called a de Finetti--Hewitt--Savage theorem. The counterexample of Dubins and Freedman seemed to put a halt to the program of generalizing de Finetti--Hewitt--Savage theorem through finding more general presentable state spaces. 

However, even if a particular state space is not presentable, it may so happen that many useful sequences of exchangeable random variables taking values in that space still satisfy a de Finetti-type theorem. This was observed by Paul Ressel in the early 1980s, who worked with exchangeable sequences of random variables whose (infinite-dimensional) joint distributions were Radon\footnote{For a Hausdorff space $T$, a Borel probability measure $\mu$ is called \textit{Radon} if for each Borel set $B \in \mathcal{B}(T)$, the following holds:
\begin{align*}
    \mu(B) &= \sup \{\mu(K): K \subseteq B \text{ and } K \text{ is compact}\}.
\end{align*}}. To more precisely state Ressel's contribution, we first make the following definitions.

\begin{definition}\label{weak topology Ressel}
Let $\ps$ and $\prs$ respectively denote the sets of all Borel probability measures and Radon probability measures on a Hausdorff space $S$. The \textit{weak topology} on either of these sets is the smallest topology under which the maps $\mu \mapsto \mathbb{E}_{\mu}(f)$ are continuous for each real-valued bounded continuous function $f \co S \rightarrow \mathbb{R}$.  
\end{definition}

\begin{definition}
Let a sequence of random variables $(X_n)_{n \in \mathbb{N}}$ taking values in a Hausdorff space $S$ be called \textit{jointly Radon distributed} if the pushforward measure induced by the sequence on $(S^{\infty}, \mathcal{B}(S^{\infty}))$ (the product of countably many copies of $S$, equipped with its Borel sigma algebra) is Radon. 
\end{definition}

\begin{definition}\label{Radon presentable definition}
Let a jointly Radon distributed sequence of exchangeable random variables $(X_n)_{n \in \mathbb{N}}$ be called \textit{Radon presentable} if there is a unique Radon measure $\mathscr{P}$ on the space $\prs$ of all Radon measures on $S$ (equipped with the Borel sigma algebra induced by its weak topology) such that the following holds:
\begin{align}
    \mathbb{P}(X_1 \in B_1, \ldots, X_k \in B_k) = \int\limits_{\prs} \mu(B_1)\cdot \ldots \cdot \mu(B_k) d\mathscr{P}(\mu) \nonumber \\
    \text{ for all } k \in \mathbb{N} \text{ and } B_1, \ldots, B_k \in \mathcal{B}(S). \label{Ressel's equation}
\end{align}
\end{definition}

Ressel \cite{Ressel-harmonic} proved in 1985 that all completely regular Hausdorff spaces\footnote{A Hausdorff space $T$ is called \textit{completely regular Hausdorff} if given a closed set $F \subseteq T$ and $x \in T \backslash F$, there is a continuous function $f\co T \rightarrow [0,1]$ such that $f(x) = 0$ and $f(y) = 1$ for all $y \in F$.} are Radon presentable, thus generalizing the de Finetti--Hewitt--Savage theorem in a new manner.\footnote{See Appendix \ref{appendix} for more details on how this generalization by Ressel \textit{includes} the generalization by Hewitt and Savage.} 

Prior to the statement of his theorem, Ressel remarked the following (see \cite[p. 906]{Ressel-harmonic}):
\begin{quote}
    ``It might be true that all Hausdorff spaces have this property.''
\end{quote}

This conjecture of Ressel was confirmed by Winkler \cite{Winkler} in 1990. 
 In the 2000s, Fremlin showed in his treatise \cite{Fremlin-4} that a stronger statement is actually true. Replacing the requirement of being jointly Radon distributed with the weaker requirement of being jointly quasi-Radon distributed (this notion is defined in Fremlin \cite[411H, p. 5]{Fremlin-4}) and marginally Radon distributed, Fremlin \cite[459H, p. 166]{Fremlin-4} showed that all such exchangeable sequences also satisfy \eqref{Ressel's equation}.
 
 The main result of the present paper generalizes this further to situations in which no assumptions on the joint distribution of the sequence of exchangeable random variables are needed, only Radonness of the marginal distribution is sufficient. Viewing $\prs$ as a subspace of $\ps$, we express our main result as follows.

 \begin{customthm}{2.5}
     Let $S$ be a Hausdorff space, with $\mathcal{B}(S)$ denoting its Borel sigma algebra. Let $\ps$ be the space of all Borel probability measures on $S$ and $\mathcal{B}(\ps)$ be the Borel sigma algebra on $\ps$ with respect to the $A$-topology\footnote{This topology, named after A.D. Alexandroff, is studied in Appendix \ref{AppA}. Briefly, it is the coarsest topology on $\ps$ with respect to which the map $\mu \mapsto \mathbb{E}_{\mu}(f)$ is upper-semicontinuous, whenever $f \colon S \to \mathbb{R}$ is a bounded upper-semicontinuous function.} on $\ps$. 

Let $(\Omega, \mathcal{F}, \mathbb{P})$ be a probability space. Let $X_1, X_2, \ldots$ be a sequence of exchangeable $S$-valued random variables such that the common distribution of the $X_i$ is Radon on $S$. Then there exists a probability measure $\mathscr{Q}$ on $(\ps, \mathcal{B}(\ps))$ such that the following holds for all $k \in \mathbb{N}$:
\begin{align*}
    \mathbb{P}(X_1 \in B_1, \ldots, X_k \in B_k) = \int\limits_{\ps} \mu(B_1)\cdot \ldots \cdot \mu(B_k) d\mathscr{Q}(\mu) \nonumber \\
    \text{ for all } B_1, \ldots, B_k \in \mathcal{B}(S). 
\end{align*}

Furthermore, all measures $\mathscr{Q}$ satisfying \eqref{classical generalized de Finetti Borel sets} must have the same restriction to the cylinder sigma algebra $\mathcal{C}(\mathfrak{P}(S))$.
 \end{customthm}
 
Using the idea that all tight measures\footnote{Recall that a Borel measure $\mu$ on a Hausdorff space $T$ is called \textit{tight} if \begin{align}\label{tightness 2}
    \mu(T) = \sup\{\mu(K): K \text{ is a compact subset of } T\}.
\end{align}} are ``close'' to Radon measures in a precise topological sense (see, for instance, Corollary \ref{what to do with tight measures}), we are also able to obtain a de Finetti-style result for any sequence of exchangeable random variables with a tight distribution---see Theorem \ref{deFinettiStyle}.

The existence of a non-Radon measure over a complete metric space is known to be equivalent to the existence of a real-valued measurable cardinal (see, for instance, Pantsulaia \cite[Remark 8, p. 340]{Pantsulaia}). In particular, it is thus consistent with the axioms of ZFC to assume that de Finetti's theorem holds whenever the state space is completely metrizable. This interpretation of our work has an interesting feature that it opens up the possibility to connect de Finetti's ideas, which famously ``brought about the rebirth of Bayesian statistics'' (see Cifarelli and Regazzini \cite[p. 253]{CifarelliRegazzini}), with philosophical aspects in the foundations of mathematics concerning the existence or non-existence of various large cardinals (see, for instance, Maddy \cite{MaddyBelieving}). 

Our work utilizes the framework of \textit{nonstandard analysis}, and is in some sense a refinement of the recent work Alam \cite{Alam-deFinetti}, in which a nonstandard proof of de Finetti's theorem for exchangeable sequences of Bernoulli random variables was obtained. 

More precisely, our proofs rely on technical results from two different topics: topological measure theory by itself, and its interaction with Loeb measure theory from nonstandard analysis, especially in the context of hyperfinite empirical distributions induced by sequences of (identically) Radon-distributed random variables. All requisite as well as supplementary technical details are compiled in the appendices (including Appendix \ref{CompactIntroduction} that provides an introduction to nonstandard methods for newcomers to the subject), allowing us to focus on the proofs of our main results in the next section. For readers interested in building more intuition before jumping to the proofs, the rest of the introduction discusses a heuristic strategy motivated from statistical practice, revealing how such a heuristic strategy naturally leads us to attempt nonstandard methods in this context.

\subsection{Building intuition for this work}
Let $\mathcal{S}$ be a sigma algebra on a state space $S$. Suppose we devise an experiment to sample values from an identically distributed sequence $X_1, \ldots, X_n$ (where $n \in \mathbb{N}$ can theoretically be as large as we please) of random variables from some underlying probability space $(\Omega, \mathcal{F}, \mathbb{P})$ to $(S, \mathcal{S})$. 

In real-world sampling of data, there can be situations in which it is not reasonable to assume that the sampled values are independent, even though it might be reasonable to believe that the joint distribution of the sampled values is invariant under permutations of the values (that is, the sample is \textit{exchangeable}). Depending on the application, one might be interested in the joint distribution of two (or more) of the $X_i$, which is difficult to establish without an assumption of independence.  It turns out that only under an assumption of \textit{exchangeability}, it is not very difficult to prove the following, which, as we describe in more detail in the next paragraph, serves as a foundation for the heuristic of estimating joint probabilities through statistical sampling.\footnote{Theorem \ref{internal version} is a nonstandard strengthening of this statement, with this standard statement having a proof along the same lines---replace the step where we use the hyperfiniteness of $N$ in that proof by an argument about taking limits.}
\begin{align}\label{heuristic idea}
    \mathbb{P}(X_1 \in B_1, \ldots, X_k \in B_k) = \lim_{n \rightarrow \infty} \mathbb{E}(\mu_{\cdot, n}(B_1) \cdot \ldots \cdot \mu_{\cdot, n}(B_k))
\end{align}
$\text{for all } k \in \mathbb{N} \text{ and } B_1, \ldots, B_k \in \mathcal{S}$, where 
\begin{align}\label{Introduction mu}
\mu_{\omega, n}(B) = \frac{\#\{i \in [n]\co X_i(\omega) \in B\}}{n} \text{ for all } \omega \in \Omega \text{ and } B \in \mathcal{S}.    
\end{align}

Here $[n]$ denotes the initial segment $\{1, \ldots, n\}$ of $n \in \mathbb{N}$. In (statistical) practice, for any $k \in \mathbb{N}$ and $B_1, \ldots, B_k \in \mathcal{S}$, we do multiple independent iterations of the experiment. For $j \in \mathbb{N}$, we calculate the product $\mu_{\cdot, n}^{(j)}(B_1) \cdot \ldots \cdot \mu_{\cdot, n}^{(j)}(B_k)$ of the ``empirical sample means'' in the $j^{\text{th}}$ iteration of the experiment. The strong law of large numbers (which we can use because of the assumption that the experiments generating samples of $(X_1, \ldots, X_n)$ are independent) thus implies the following:
\begin{align}\label{SLLN 1}
    \lim_{m \rightarrow \infty} \frac{\sum_{j \in [m]}\mu_{\cdot, n}^{(j)}(B_1) \cdot \ldots \cdot \mu_{\cdot, n}^{(j)}(B_k)}{m} = \mathbb{E}\left(\mu_{\cdot, n}(B_1) \cdot \ldots \cdot \mu_{\cdot, n}(B_k)\right) \text{ almost surely}.
\end{align}
By \eqref{SLLN 1} and \eqref{heuristic idea}, we thus obtain the following $\text{for all } k \in \mathbb{N} \text{ and } B_1, \ldots, B_k \in \mathcal{S}$:
\begin{align}\label{heuristic idea 2}
    \mathbb{P}(X_1 \in B_1, \ldots, X_k \in B_k) = \lim_{n \rightarrow \infty} \lim_{m \rightarrow \infty} \frac{\sum_{j \in [m]}\mu_{\cdot, n}^{(j)}(B_1) \cdot \ldots \cdot \mu_{\cdot, n}^{(j)}(B_k)}{m}.
\end{align}

Thus, only under an assumption of exchangeability of the values sampled in each experiment, as long as we have a method to repeat the experiment independently, we have the following heuristic algorithm to statistically approximate the joint probability $\mathbb{P}(X_1 \in B_1, \ldots, X_k \in B_k)$ for any $B_1, \ldots, B_k \in \mathcal{S}$:
\begin{enumerate}[(i)]
    \item In each iteration of the experiment, sample a large number (this corresponds to $n$ in \eqref{heuristic idea 2}) of values. 
    \item Conduct a large number (this corresponds to $m$ in \eqref{heuristic idea 2}) of such independent experiments. 
    \item The average of the empirical sample means $\mu_{\cdot, n}^{(j)}(B_1) \cdot \ldots \cdot \mu_{\cdot, n}^{(j)}(B_k)$ (as $j$ varies in $[m]$) is then an approximation to $\mathbb{P}(X_1 \in B_1, \ldots, X_k \in B_k)$. 
\end{enumerate}

Going back to \eqref{heuristic idea}, suppose for the moment that we have fixed some sigma algebra on $\ps$ (we will come back to the issue of which sigma algebra to fix) such that the following natural conditions are met:
\begin{enumerate}[(i)]
    \item For each $n \in \mathbb{N}$, the map $\omega \mapsto \mu_{\omega, n}$ is a $\ps$-valued random variable on $\Omega$.
    \item For each $B \in \mathcal{S}$, the map $\mu \mapsto \mu(B)$ is a real-valued random variable on $\ps$.
\end{enumerate}

For each $n \in \mathbb{N}$, this would define a pushforward probability measure $\nu_n$ on $\ps$ that is supported on $\{\mu_{\omega, n} \co \omega \in \Omega\} \subseteq \ps$, such that  
\begin{align}\nonumber
\int_{\ps} \mu(B_1) \ldots \mu(B_k) d\nu_n(\mu) = \int_{\Omega} \mu_{\omega, n}(B_1) \ldots \mu_{\omega, n}(B_k) d\mathbb{P}(\omega)  \\ \text{for all } B_1, \ldots, B_k \in \mathcal{S}. \label{pushforward}
\end{align}

Comparing \eqref{heuristic idea}, \eqref{pushforward}, and \eqref{naive generalization}, it is clear that de Finetti's theorem would hold for an \textit{infinite} sequence of exchangeable $S$-valued random variables $(X_n)_{n \in \mathbb{N}}$ whenever we could guarantee there to be a measure $\nu$ on $\ps$ such that the following is true:
\begin{align}
\lim_{n \rightarrow \infty}\int_{\ps} \mu(B_1) \ldots \mu(B_k) d\nu_n(\mu) = \int_{\ps} \mu(B_1) \ldots \mu(B_k) d\nu(\mu) \nonumber \\ \text{for all } B_1, \ldots, B_k \in \mathcal{S}. \label{heuristic idea 3}
\end{align}

Intuitively, equation \eqref{heuristic idea 3} is a statement of convergence (in some sense) of $\nu_n$ to $\nu$. A naive candidate for $\nu$ could come from \eqref{pushforward} if the following are true:
\begin{enumerate}
    \item\label{hope1} There exists an almost sure set $\Omega' \subseteq \Omega$ such that for each $B \in \mathcal{S}$, the limit $\lim_{n \rightarrow \infty} \mu_{\omega, n}(B)$ exists for all $\omega \in \Omega'$. Up to null sets in $\Omega$, this would thus define a map $\omega \mapsto \mu_{\omega}$ from $\Omega$ to the space of all real-valued functions on $\mathcal{S}$, where $\mu_{\omega}(B) = \lim_{n \rightarrow \infty} \mu_{\omega, n}(B)$. 
    \item\label{hope2} The function $\mu_{\omega} : \mathcal{S} \rightarrow [0,1]$ is actually a probability measure on $(S, \mathcal{S})$. 
\end{enumerate}

Indeed if these two conditions are true, then one may define $\nu$ to be the pushforward on $\ps$ of the map $\omega \mapsto \mu_{\omega}$. A weaker version of \eqref{hope1} is often interpreted as a generalization of the strong law of large numbers for exchangeable random variables---see, for instance, Kingman \cite[Equation (2.2), p. 185]{Kingman-survey}, which can be easily modified to work in the setting of an arbitrary $(S, \mathcal{S})$ to conclude that $\lim_{n \rightarrow \infty} \mu_{\omega, n}(B)$ exists for all $\omega$ in an almost sure set that depends on $B$. Of course, an issue with this idea is that if we have too many (that is, uncountably many) different choices for $B \in \mathcal{S}$, then there is no guarantee that an almost sure set would exist that works for all $B\in \mathcal{S}$ simultaneously. The condition \eqref{hope2} is even more delicate, as showing countable additivity of $\mu_{\omega}$ would require some control on the rates at which the sequences $(\mu_{\omega, n}(B))_{n \in \mathbb{N}}$ converge for different $B \in \mathcal{S}$. 

Thus we seem to have reached a dead end in this heuristic strategy in the absence of having more information about the specific structure of our spaces and measures. 

What helps get us past this apparent dead end is our use of hyperfinite numbers from nonstandard analysis as tools to model large sample sizes. Fix a hyperfinite $N > \mathbb{N}$ and study the map $\omega \mapsto \mu_{\omega, N}$ from ${^*}\Omega$ to ${^*}\ps$. This map induces an internal probability measure (through the pushforward) on the space ${^*}\ps$ of all internal probability measures on ${^*}S$. That is, this pushforward internal measure $Q_N$ (say) lives in the space ${^*}\mathfrak{P}(\ps)$. In view of \eqref{heuristic idea 3} (and the nonstandard characterization of limits), we want to have a standard probability measure $\mathscr{Q}$ on $\mathfrak{P}(\ps)$ that is close to $Q_N$ in the sense that the integral of the function $\mu \mapsto \mu({^*}B_1) \cdot \ldots \cdot \mu({^*}B_k)$ with respect to $Q_N$ is infinitesimally close to its integral with respect to ${^*}\mathscr{Q}$ for any $k \in \mathbb{N}$ and $B_1, \ldots, B_k \in \mathcal{S}$. 

If $S$ has a topological structure, then there are natural ways to topologize $\ps$. The notion of topology most useful to our situation is the so-called $A$-topology, named after A.D. Alexandroff, which is studied in more detail in Appendix \ref{AppA}. Briefly, it is the coarsest topology on $\ps$ with respect to which the map $\mu \mapsto \mathbb{E}_{\mu}(f)$ is upper-semicontinuous, whenever $f \colon S \to \mathbb{R}$ is a bounded upper-semicontinuous function.

Thus $\mathfrak{P}(\ps)$ also comes equipped with its $A$-topology, which lends itself to the straightforward nonstandard strategy, namely to look for conditions with respect to which the internal measure $Q_N$ from the previous paragraph is nearstandard to some $\mathscr{Q} \in \mathfrak{P}(\ps)$---while hoping that the topologies we are considering are rich enough to ensure that this nearstandardness is sufficient for our needs. A technical obstruction that appears in this strategy is that $Q_N$ actually does not belong to ${^*}\mathfrak{P}(\ps)$, since the Borel sigma algebra on $\ps$ induced by its $A$-topology is too large for the map $\omega \mapsto \mu_{\omega, n}$ (for some $n \in \mathbb{N}$ to necessarily be $\mathcal{F}-\mathcal{B}(\ps)$ measurable. So, we work with a slightly smaller sigma algebra on $\ps$ that still does the job for us. Another technical nuance is related to the fact that $\ps$ is not necessarily Hausdorff, while the space $\prs$ of Radon probability measures over a Hausdorff space $S$ is always Hausdorff (under the $A$-topology). Since the random measures $\mu_{\cdot, n}$ (where $n \in \mathbb{N}$) are supported on finitely many points, they are always Radon, and hence it is helpful to view the sample space of the internal measure $Q_N$ as ${^*}\prs$ instead of ${^*}\ps$. We study these hyperfinite empirical measures in Appendix \ref{AppB}.

The main tool in finding a standard measure ``close to'' $Q_N$ is Theorem \ref{Albeverio theorem} (originally from Albeverio et al. \cite[Proposition 3.4.6, p. 89]{Albeverio}). This technique is called ``pushing down Loeb measures'' and is well-known in the nonstandard literature (see, for example, Albeverio et al. \cite[Chapter 3.4]{Albeverio} or Ross \cite[Section 3]{Ross_NATO}). It is often used to construct a standard measure that is close in some sense to an internal (nonstandard) measure. The way we develop the theory of $A$-topology allows us to interpret this classical technique of pushing down Loeb measures as actually taking a standard part in a legitimate nonstandard space (of internal measures). See, for example, Theorem \ref{Landers and Rogge main result}, Remark \ref{pushing down remark}, and Theorem \ref{Radon Landers and Rogge main result}. Similar results were obtained in the context of the topology of weak convergence by Anderson \cite[Proposition 8.4(ii), p. 684]{Anderson-transactions}, and by Anderson--Rashid \cite[Lemma 2, p. 329]{Anderson--Rashid} (see also Loeb \cite{Loeb-1979}). Using Theorem \ref{Albeverio theorem} as described above requires us to first show the existence of large compact sets in $\prs$ in some sense, which is shown to be the case in Theorem \ref{Prokhorov for P_N} using a version of Prokhorov's theorem in this setting (see Theorem \ref{Radon Prokhorov}). 

In some sense, we prove a highly general de Finetti's theorem using the same underlying basic idea that works for the simplest versions of de Finetti's theorem (that being the idea of approximating using empirical sample means), the technical machinery from topological measure theory and nonstandard analysis notwithstanding. 

For a more thorough introduction to exchangeability, see Aldous \cite{Aldous-book}, Kingman \cite{Kingman-survey}, and Kallenberg \cite{Kallenberg-symmetries}. Besides a recent paper of the author on a nonstandard proof of de Finetti's theorem for Bernoulli random variables (see Alam \cite{Alam-deFinetti}), there is some precedence in the use of nonstandard analysis in this field, as Hoover \cite{Hoover-row-column-exchangeabiliy, Hoover-IAS} studied the notions of exchangeability for multi-dimensional arrays using nonstandard methods in the guise of ultraproducts. In view of this work, Aldous \cite[p. 179]{Aldous-book} had also expressed the hope of nonstandard analysis being useful in other topics in exchangeability. Another example is Dacunha-Castelle \cite{Dacunha-Castelle} who also used ultraproducts to study exchangeability in Banach spaces. 

\section{Our generalization of de Finetti--Hewitt--Savage theorem}
Here is a brief synopsis of the notation and conventions that we use.\footnote{For readers completely new to nonstandard analysis, Appendix \ref{CompactIntroduction} provides an introduction aimed at general audiences, and should be consulted before reading any further. This introduction in Appendix \ref{CompactIntroduction} is spilled over through the first two sections of Appendix \ref{AppA} that discuss nonstandard methods in the context of topology and measure theory. Readers familiar with basic nonstandard analysis, say, at the level of the first two chapters of \cite{working-mathematician}, definitely have more than enough background to be able to fruitfully read ahead right away, although they might still benefit from skimming through the beginning of Appendix \ref{AppA} in order to get more familiar with our conventions.} We follow the superstructure approach to nonstandard analysis. In particular, we fix a sufficiently saturated nonstandard extension of a \textit{superstructure} containing all standard mathematical objects under study. The nonstandard extension (or \textit{nonstandard interpretation}) of a set $A$ (respectively a function $f$) is denoted by ${^*}A$ (respectively ${^*}f$). Roughly, any first-order property we can express using symbols denoting objects in the standard universe is true if and only if the corresponding statement in the nonstandard universe holds true for the nonstandard interpretations of those symbols---this is called the \textit{transfer principle}. The nonstandard extension ${^*}\mathbb{R}$ of $\mathbb{R}$ has the same first order properties as $\mathbb{R}$ but contains many more elements, including \textit{infinitesimals} and \textit{hyperfinite} numbers. 

For any set $A$, the notation $\mathcal{P}(A)$ denotes the powerset of $A$. The members of ${^*}\mathcal{P}(A)$ are called the \textit{internal subsets} of ${^*}{A}$. For a set $T$ with a \textit{topology} $\tau \subseteq \mathcal{P}(T)$ on it, an element $y \in {^*}T$ is \textit{nearstandard} to $x \in T$ if $y$ belongs to ${^*}U$ for all open neighborhoods $U$ of $x$ (i.e., for all $U$ satisfying $x \in U \in \tau$). If $T$ is Hausdorff, then an element in ${^*}T$ can be nearstandard to at most one element of $T$, thus leading to the concept of the \textit{standard part} map $\st \colon \ns({^*}T) \to T$, where $\ns({^*}T)$ is the set of those elements in ${^*}T$ that are nearstandard. In the case of the nonstandard extension of the set $\mathbb{R}$ of real numbers, for two nonstandard numbers $x, y \in {^*}\mathbb{R}$, we will write $x \approx y$ to denote that $x - y$ is an infinitesimal. The set of finite nonstandard real numbers will be denoted by ${^*}\mathbb{R}_{\text{fin}}$ and the standard part map $\st\co {^*}\mathbb{R}_{\text{fin}} \rightarrow \mathbb{R}$ takes a finite nonstandard real to its closest real number.

A \textit{measurable space} is a set equipped with a sigma-algebra, while a probability space is a measurable space that is further equipped with a \textit{probability measure}. If $\Gamma$ is an internal set and $\mathcal{A}$ is an internal algebra on $\Gamma$, then given any finitely additive internal function $\mathbb{P} \colon \mathcal{A} \to {^*}[0,1]$, its standard part $\st(\mathbb{P}) \colon \mathcal{A} \to [0,1]$ is a finitely-additive measure on an algebra (it is actually a finitely-additive probability measure on an algebra, provided $\mathbb{P}(\Gamma) \approx 1$), which can always be extended to its corresponding \textit{Loeb measure} $L\mathbb{P}$ which is a \textit{complete} countably additive measure on a sigma algebra $L(\mathcal{A})$ containing $\mathcal{A}$. 

We begin our analysis on exchangeability with a combinatorial observation that, once we consider Loeb measures, a suitable form of de Finetti--Hewitt--Savage theorem can be shown to be generally true for \textit{any} hyperfinite collection of internal exchangeable random variables taking values in \textit{any} internal measurable space. More precisely, we have the following result. 

\begin{theorem}\label{internal version}
Let $\Gamma$ be an internal set, $\mathcal{A}$ be an internal algebra on $\Gamma$, and $\mathbf{P}$ be an internal finitely-additive map $\mathbf{P} \colon \mathcal{A} \to {^*}[0,1]$ such that $\mathbf{P}(\Gamma) \approx 1$. Let $\mathbb{S}$ be a (possibly different) internal set equipped with an internal algebra $\mathfrak{S}$. 

Let $N > \mathbb{N}$---that is, $N \in {^*}\mathbb{N} \backslash \mathbb{N}$. Let us denote by $[N]$ the initial segment of $N$ in ${^*}\mathbb{N}$, and let $\{\mathfrak{X}_i : i \in [N]\}$ be a hyperfinite collection of exchangeable $\mathbb{S}$-valued internal random variables defined on $\Gamma$---that is, each $\mathfrak{X}_i \colon \Gamma \to \mathbb{S}$ in this internal collection is an internal map such that the pre-image ${\mathfrak{X}_i}^{-1}(\mathscr{S}) \in \mathcal{A}$ for all $\mathscr{S} \in \mathfrak{S}$, while for any finitely many sets $\mathscr{S}_1, \ldots, \mathscr{S}_k \in \mathfrak{S}$ and any permutation $\sigma \in S_N$ (the internal symmetric group on $[N]$) we have $\mathbf{P}(\mathfrak{X}_1 \in \mathscr{S}_1, \ldots, \mathfrak{X}_k \in \mathscr{S}_k) = \mathbf{P}(\mathfrak{X}_{\sigma(1)} \in \mathscr{S}_1, \ldots, \mathfrak{X}_{\sigma(k)} \in \mathscr{S}_k)$.

For each $\gamma \in \Gamma$, we define the $N^{\text{th}}$ empirical mean at $\gamma$ to be the internal function $\mu_{\gamma, N} \colon \mathfrak{S} \to {^*}[0,1]$ satisfying the following formula:
   \begin{align}\label{generalEmpiricalDefinition}
       \mu_{\gamma, N}(\mathscr{S}) \defeq \frac{\# \{i \in [N]: \mathfrak{X}_i(\gamma) \in \mathscr{S} \}}{N} \text{ for all } \mathscr{S} \in \mathfrak{S}. 
   \end{align}
   
 Then $\mu_{\gamma, N}$ is an internal finitely additive probability on $(\mathbb{S}, \mathfrak{S})$ for all $\gamma \in \Gamma$. Furthermore, for each $\mathscr{S} \in \mathfrak{S}$, the map $\gamma \mapsto \mu_{\gamma, N}(\mathscr{S})$ is a ${^*}\mathbb{R}$-valued internal random variable such that the following is true:

 \begin{align}\label{generalLoebDeFinetti}
     L\mathbf{P}(\mathfrak{X}_1 \in \mathscr{S}_1, \ldots, \mathfrak{X}_k \in \mathscr{S}_k) = \int_{\Gamma} L\mu_{\gamma, N}(\mathscr{S}_1)\cdots L\mu_{\gamma, N}(\mathscr{S}_k) dL\mathbf{P}(\gamma) \nonumber \\
    \text{ for all } k \in \mathbb{N} \text{ and } \mathscr{S}_1, \ldots, \mathscr{S}_k \in \mathfrak{S}, 
 \end{align}
 where $(\Gamma, L(\mathcal{A}), L\mathbf{P})$ and $(\mathbb{S}, L_{\mu_{\gamma, N}}(\mathfrak{S}), L\mu_{\gamma, N})_{\gamma \in \Gamma}$ are the Loeb spaces induced by $(\Gamma, \mathcal{A}, \mathbf{P})$ and $(\mathbb{S}, \mathfrak{S}, \mu_{\gamma, N})_{\gamma \in \Gamma}$ respectively. 
\end{theorem}

A simple direct proof of Theorem \ref{internal version} is provided in Appendix \ref{AppC}, while a different proof using conditional probabilities (which is closer in spirit to Bayesian ways of thinking) is provided in Appendix \ref{appendix 2}.\footnote{While Theorem \ref{internal version} is a bit more general, the specific instantiation of it that we shall need for our generalization of classical de Finetti--Hewitt--Savage theorem can also be seen as a direct consequence of \textit{transferring} Diaconis--Freedman's finite, approximate version of de Finetti's theorem \cite[Theorem (13), pp. 749-750]{Diaconis-Freedman-Finite} into the hyperfinite setting, and then taking standard parts. Thus, readers familiar with the work of Diaconis--Freedman \cite{Diaconis-Freedman-Finite} may skip our proofs of Theorem \ref{internal version} without compromising much their understanding of the rest of the paper.} In the present section, our concern is to establish how Theorem \ref{internal version} leads to a proof of a de Finetti--Hewitt--Savage theorem for exchangeable sequences of Radon-distributed random variables taking values in any Hausdorff space (i.e., Theorem \ref{classical most general de Finetti}). This is achieved by combining Theorem \ref{internal version} with the nonstandard topological measure theory developed in the appendices. We aim to provide the relevant details of this proof next. 

Throughout the rest of this section, we fix a probability space $(\Omega, \mathcal{F}, \mbp)$ and a sequence $(X_n)_{n \in \mathbb{N}}$ of exchangeable random variables taking values in a topological space $S$. We also fix $N > \mathbb{N}$. Letting $(\Gamma, \mathcal{A}, \mathbf{P}) = ({^*}\Omega, {^*}\mathcal{F}, {^*}\mbp)$ and $(\mathbb{S}, \mathfrak{S}) = ({^*}S, {^*}\mathcal{B}(S))$, where $\mathcal{B}(S)$ denotes the Borel sigma algebra induced by the topology on $S$, an application of Theorem \ref{internal version} immediately yields the following: 
 \begin{align}
     \mathbb{P}(X_1 \in B_1, \ldots, X_k \in B_k) = \int_{{^*}\Omega} L\mu_{\omega, N}({^*}B_1)\cdots L\mu_{\omega, N}({^*}B_k) dL{^*}\mathbb{P}(\omega) \nonumber \\
    \text{ for all } k \in \mathbb{N} \text{ and } B_1, \ldots, B_k \in \mathcal{B}(S). \label{almost there}
 \end{align}

On the other hand, Corollary \ref{most useful corollary} shows that there must exist a Radon probability measure $\mathscr{P}$ on the space the space $\prs$, equipped with its $A$-topology, such that the right side of \eqref{almost there} equals the expected value of the random variable $\mu \mapsto \mu(B_1)\cdots\mu(B_k)$ with respect to $\mathscr{P}$. The idea behind Corollary \ref{most useful corollary} is that the measure $\mathscr{P}$ is obtained from the pushforward of the internal measure-valued internal random variable $\omega \mapsto \mu_{\omega, N}$ for \textit{any} choice of $N > \mathbb{N}$---by further pushing forward the Loeb measure of $\mu_{\omega, N}$ via the standard part on the nearstandard elements of ${^*}S$. In particular, we obtain the following (with the uniqueness due to Theorem \ref{uniqueness of Radon presentable}). 

\begin{restatethis}{theorem}{Finetti}\label{most general de Finetti} 
 Let $S$ be a Hausdorff topological space, with $\mathcal{B}(S)$ denoting its Borel sigma algebra. Let $\prs$ be the space of all Radon probability measures on $S$ and $\mathcal{B}(\prs)$ be the Borel sigma algebra on $\prs$ with respect to the $A$-topology on $\prs$.

Let $(\Omega, \mathcal{F}, \mathbb{P})$ be a probability space. Let $X_1, X_2, \ldots$ be a sequence of exchangeable $S$-valued random variables such that the common distribution of the $X_i$ is Radon on $S$. Then there exists a unique Radon probability measure $\mathscr{P}$ on  $(\prs, \mathcal{B}(\prs))$ such that the following holds for all $k \in \mathbb{N}$:
\begin{align}
    \mathbb{P}(X_1 \in B_1, \ldots, X_k \in B_k) = \int\limits_{\prs} \mu(B_1)\cdot \ldots \cdot \mu(B_k) d\mathscr{P}(\mu) \nonumber \\
    \text{ for all } B_1, \ldots, B_k \in \mathcal{B}(S). \label{generalized de Finetti Borel sets}
\end{align}
\end{restatethis}

Although Theorem \ref{most general de Finetti} is already a generalization of de Finetti's theorem, its conclusion is slightly different from classical statements of de Finetti's theorem that postulate the existence of a probability measure on the space of all probability measures on $S$. This can be very easily remedied by considering the unique Radon mixing measure $\mathscr{P}$ on $(\prs, \mathcal{B}(\prs))$ obtained in Theorem \ref{most general de Finetti} and constructing a corresponding mixing measure $\mathscr{Q} \co \mathcal{B}(\ps) \rightarrow [0,1]$ out of it as follows:
\begin{align}\label{definition of Q}
    \mathscr{Q}(\mathfrak{B}) \defeq \mathscr{P}(\mathfrak{B} \cap \prs) \text{ for all } \mathfrak{B} \in \mathcal{B}(\ps).
\end{align}
By Corollary \ref{non-Radon de Finetti uniqueness}, we can also postulate the uniqueness of a mixing measure on $\ps$ if we restrict to the cylinder sigma algebra $\mathcal{C}(\ps)$. Thus, we have obtained the following theorem. 

\begin{theorem}\label{classical most general de Finetti}
 Let $S$ be a Hausdorff space, with $\mathcal{B}(S)$ denoting its Borel sigma algebra. Let $\ps$ be the space of all Borel probability measures on $S$ and $\mathcal{B}(\ps)$ be the Borel sigma algebra on $\ps$ with respect to the $A$-topology on $\ps$. 

Let $(\Omega, \mathcal{F}, \mathbb{P})$ be a probability space. Let $X_1, X_2, \ldots$ be a sequence of exchangeable $S$-valued random variables such that the common distribution of the $X_i$ is Radon on $S$. Then there exists a probability measure $\mathscr{Q}$ on $(\ps, \mathcal{B}(\ps))$ such that the following holds for all $k \in \mathbb{N}$:
\begin{align}
    \mathbb{P}(X_1 \in B_1, \ldots, X_k \in B_k) = \int\limits_{\ps} \mu(B_1)\cdot \ldots \cdot \mu(B_k) d\mathscr{Q}(\mu) \nonumber \\
    \text{ for all } B_1, \ldots, B_k \in \mathcal{B}(S). \label{classical generalized de Finetti Borel sets}
\end{align}

Furthermore, all measures $\mathscr{Q}$ satisfying \eqref{classical generalized de Finetti Borel sets} must have the same restriction to the cylinder sigma algebra $\mathcal{C}(\mathfrak{P}(S))$.
\end{theorem}

\begin{remark}\label{Dubins-freedman remark}
Dubins and Freedman \cite{Dubins-Freedman} had constructed an exchangeable sequence of random variables taking values in a separable metric space for which the conclusion of de Finetti's theorem does not hold. An indirect consequence of the above theorem is that any random variable in such an example must have a non-Radon distribution. 
\end{remark}

Note that a \textit{Radon space} refers to a topological space all of whose Borel probability measures are Radon. Thus we have the following immediate corollary. 

\begin{corollary}\label{most general corollary}
Let $S$ be a Radon space. Let $(\Omega, \mathcal{F}, \mathbb{P})$ be a probability space. Let $X_1, X_2, \ldots$ be a sequence of exchangeable $S$-valued random variables. Then there exists a probability measure $\mathscr{Q}$ on the space $(\ps, \mathcal{B}(\ps))$ such that \eqref{classical generalized de Finetti Borel sets} holds. Furthermore, all measures $\mathscr{Q}$ satisfying \eqref{classical generalized de Finetti Borel sets} must be the same when restricted to the cylinder sigma algebra $\mathcal{C}(\mathfrak{P}(S))$. 
\end{corollary}

By Remark \ref{Dubins-freedman remark}, we know that separability alone of a metric space is not a sufficient condition for it to be a Radon space. At the same time, it is well known that completeness together with separability of a metric space is sufficient for it to be a Radon space. Hence it is natural to ask how important the separability condition is for the truth of the fact that all Polish spaces are Radon spaces. The following result from Pantsulaia \cite[Remark 8, p. 340]{Pantsulaia} shows that separability is not needed in case there do not exist real-valued measurable cardinals. 

\begin{theorem}[Pantsulaia]\label{Radonwic} \label{Pantsulaia's theorem}
    The following are equivalent:
\begin{enumerate}[(a)]
    \item All complete metric spaces are Radon spaces.
    \item There does not exist a real-valued measurable cardinal.
\end{enumerate}
\end{theorem}

A theorem of Ulam \cite{ulam1930}, which Ulam also attributes to Tarski (cf. Jech \cite[Historical Notes, p. 137]{Jech}), shows that any real-valued measurable cardinal is weakly inaccessible\footnote{See Fremlin \cite[Theorem 1D]{Fremlin-rvmc} or Jech \cite[Corollary 10.15]{Jech} for a proof of this statement.}. It is consistent with ZFC to assume that there are no weakly inaccessible cardinals (see, for instance, Cohen \cite[p. 80]{Cohen}). In other words, if ZFC is consistent, then we cannot prove using ZFC that a weakly inaccessible cardinal exists. Therefore, by Theorem \ref{Radonwic}, it is consistent with ZFC to assume that all complete metric spaces are Radon spaces. 

Note that the results proven using nonstandard methods in this manuscript are consequences of ZFC. Indeed the superstructure formulation of nonstandard analysis that we have used is constructed within ZFC, see for instance the discussion in Chang and Keisler \cite[Section 4.4]{Model_Theory} for more details. Therefore, by Corollary \ref{most general corollary}, we obtain the following result.

\begin{theorem}\label{de Finetti consistent}
    It is consistent with the axioms of ZFC that de Finetti's theorem holds for any sequence of exchangeable random variables taking values in a completely metrizable state space. Furthermore, the existence of a complete metric state space for which de Finetti's theorem fails implies the existence of real-valued inaccessible cardinals.   
\end{theorem}

The above argument rests on the fact that it is relatively consistent with ZFC that weakly inaccessible cardinals do not exist. This fact does not imply that such cardinals ``cannot'' exist, only that their existence is unprovable in ZFC. Although most modern mathematics is done in the framework of ZFC, there is a perennial debate in the foundations of mathematics about which other axioms should also be assumed. The existence (or non-existence) of weakly inaccessible cardinals is a popular candidate for such a new axiom to bolster ZFC with (see, for instance, Maddy \cite{MaddyBelieving}). Although we do not pursue such a philosophical investigation in the present work, Theorem \ref{de Finetti consistent} allows us to possibly interpret de Finetti's theorem as providing perspectives for the foundations of mathematics, and not just for the foundations of Bayesian statistics that de Finetti's name is usually tied with. 

As explained in Remark \ref{tight and outer regular on compact remark}, the way we have used the assumption of the common distribution of the $X_i$ being Radon in this work is by using the fact that all Radon measures are tight and outer regular on compact subsets --- (which are properties that characterize Radon measures). A natural situation in which all Borel probability measures on a Hausdorff space $S$ are clearly outer regular on compact subsets is when $S$ is a Hausdorff $G_{\delta}$ space---that is, when all closed subsets of $S$ are expressible as countable intersections of open sets (as any finite Borel measure on such a space is actually outer regular on all closed subsets, and in particular on all compact subsets). 

In the point-set topology literature, $G_{\delta}$ spaces typically arise in discussions on \textit{perfectly normal spaces}.\footnote{A \textit{perfectly normal space} is a \textit{normal} $G_{\delta}$ space. Here, a space $T$ is called \textit{normal} if any two disjoint closed subsets of $T$ can be separated by open sets---that is, given closed sets $F_1, F_2 \subseteq T$ such that $F_1 \cap F_2 = \emptyset$, there exist disjoint open sets $G_1$ and $G_2$ such that $F_1 \subseteq G_1$ and $F_2 \subseteq G_2$.} Following are some commonly studied examples of spaces that are perfectly normal (as described in Gartside \cite[p. 274]{Gartside-chapter}, these are actually examples of \textit{stratifiable spaces}, which are automatically perfectly normal):
\begin{enumerate}[(i)]
    \item\label{pn1}  All CW complexes are perfectly normal. See Lundell and Weingram \cite[Proposition 4.3, p. 55]{Lundell-Weingram}.
    
    \item\label{pn2} All La\u{s}nev spaces (that is, all continuous closed images of metric spaces, where a continuous map $g\co T \rightarrow T'$ is called \textit{closed} if $g(F)$ is closed in $T'$ whenever $F$ is closed in $T$) are perfectly normal. This, in particular, includes all metric spaces. See Slaughter \cite{Slaughter} for more details.
    
    \item\label{pn3}  If $T$ is a compact-covering image of a Polish space (here, a continuous map $f\co T \rightarrow T'$ is called a \textit{compact-covering} if every compact subset of $T'$ is the image of a compact subset of $T$), then the space $C_k(T)$ of continuous real-valued functions on $T$ (equipped with the compact-open topology) is perfectly normal. In particular, this implies that $C_k(T)$ is perfectly normal whenever $T$ is a Polish space. See Gartside and Reznichenko [Theorem 34, p. 111]\cite{Gartside-Reznichenko}.
\end{enumerate}

The above discussion shows that we could have stated Theorem \ref{classical most general de Finetti} for any exchangeable sequence of tightly distributed random variables taking values in a Hausdorff state space that is either a CW complex, a La\u{s}nev space, or a space of continuous real-valued functions on a Polish space (with the compact-open topology). This, however, would not be a more general statement than that of Theorem \ref{classical most general de Finetti}, as it is easy to see that any tight finite measure on a Hausdorff $G_{\delta}$ space is automatically Radon. It is still instructive to keep in mind these settings where one only needs to verify tightness of the common distribution in order for de Finetti--Hewitt--Savage theorem to hold. 

We conclude our investigation on the de Finetti--Hewitt--Savage theorem with another de Finetti-style result for exchangeable sequences of random variables with a tight, but not necessarily Radon, distribution. The proof of this result uses the fact that any tight probability measure $\mu$ is ``topologically inseparable'' from a particular Radon measure associated with the nonstandard extension of $\mu$. (See Corollary \ref{what to do with tight measures}, whose proof relies on Corollary \ref{not Hausdorff corollary 0}.) 

\begin{theorem}\label{deFinettiStyle} Let $(\Omega, \mathcal{F}, \mathbb{P})$ be a probability space. Let $(X_n)_{n \in \mathbb{N}}$ be a sequence of $S$-valued exchangeable random variables, where $S$ is a Hausdorff topological space, equipped with its Borel sigma algebra. Suppose further that the distribution of $X_1$ is tight. Then there exists a probability measure $\psi$ on the Borel sigma algebra over the set $\ps$ of all Borel probability measures on $S$, such that for all bounded continuous $F \colon S^\infty \to \mathbb{R}$ we have: 
\begin{align}
   \mathbb{E}\left(F(X_1, X_2, \ldots)\right) = \int_{\ps} \left(\int_{S^\infty} F d\mu^{\otimes \infty} \right) d\psi(\mu), 
    \label{deFinettiTight}
\end{align}
where $S^\infty$ is the Cartesian product of countably many copies of $S$, equipped with product topology, and $\mu^{\otimes \infty}$ denotes the product measure induced on $(S^\infty, \mathcal{B}(S^\infty))$ by countably many copies of $\mu$.  
\end{theorem}
\begin{proof}
    Let $\nu$ be the pushforward probability measure induced on the product space $(S^{\infty}, \mathcal{B}(S^{\infty}))$ by the $S^{\infty}$-valued random variable $(X_1, X_2, \ldots)$. It is not very difficult to verify that a probability measure on $(S^{\infty}, \mathcal{B}(S^{\infty}))$ is tight if and only if its pushforward under the projection maps $\pi_i \colon S^{\infty} \to S$ (where $\pi_i(x_1, x_2, \ldots) = x_i$ for all $i \in \mathbb{N}$) are all tight. In particular, $\nu$ is tight, and therefore, by Corollary \ref{what to do with tight measures} there exists a Radon measure $\nu'$ on $(S^{\infty}, \mathcal{B}(S^{\infty}))$ such that the following holds:
    \begin{align}\label{tight to Radon}
        \int_{S^{\infty}} F d\nu = \int_{S^{\infty}} F d\nu' \text{ for all bounded continuous functions } F \colon S^{\infty} \to \mathbb{R}. 
    \end{align}

By Theorem \ref{classical most general de Finetti}, there exists a probability measure $\psi$ on $(\ps, \mathcal{B}(\ps))$ such that the following holds for all $k \in \mathbb{N}$ and $B_1, \ldots, B_k \in \mathcal{B}(S)$:
\begin{align}
    \nu'(B_1 \times \ldots \times B_k) = \int\limits_{\ps} \mu(B_1)\cdot \ldots \cdot \mu(B_k) d\psi(\mu) =  \int\limits_{\ps} \mu^{\otimes k}(B_1 \times \ldots \times B_k) d\psi(\mu). \label{new classical generalized de Finetti Borel sets}
\end{align}

Let $\mathbf{B}(S^{\infty}, \mathbb{R})$ denote the collection of all bounded measurable functions from $S^{\infty}$ to $\mathbb{R}$. Consider the subcollection $\mathcal{C}$ defined as follows:
\begin{align*}
    \mathcal{C} \defeq \{\alpha \mathbbm{1}_{\pi_{(k)}^{-1}(B_1 \times \ldots \times B_k)}: k \in \mathbb{N}, \alpha \in \mathbb{R}, \text{ and } B_1, \ldots, B_k \in \mathcal{B}(S)\} \subseteq \mathbf{B}(S^{\infty}, \mathbb{R}),
\end{align*}
where for each $k \in \mathbb{N}$, the map $\pi_{(k)} \colon S^{\infty} \to S^k$ is the projection onto the first $k$ coordinates. Note that $\mathcal{C}$ is closed under multiplication.

Rewriting \eqref{new classical generalized de Finetti Borel sets} in terms of integrals of indicator functions, it follows that the following collection contains $\mathcal{C}$:
\begin{align*}
    \mathbb{H} \defeq \left\{F \in \mathbf{B}(S^{\infty}, \mathbb{R}):\int_{S^{\infty}} F d\nu' = \int_{\ps}\left(\int_{S^\infty} F d\mu^{\otimes \infty} \right) d\psi(\mu)\right\}.
\end{align*}

The set $\mathbb{H}$ clearly satisfies the assumptions of Theorem \ref{Dellacherie's theorem}, and hence, by that theorem, it contains all bounded functions measurable with respect to $\sigma(\mathcal{C})$, which denotes the smallest sigma algebra with respect to which all functions in $\mathcal{C}$ are measurable. However, $\sigma(\mathcal{C})$ is just $\mathcal{B}(S^{\infty})$. Thus, we have proved the following for all bounded measurable functions $F \colon S^{\infty} \to \mathbb{R}$:
\begin{align}\label{almost there 2}
    \int_{S^{\infty}} F d\nu' = \int_{\ps}\left(\int_{S^\infty} F d\mu^{\otimes \infty} \right) d\psi(\mu).
\end{align}
The proof is now completed by comparing \eqref{tight to Radon} and \eqref{almost there 2}.
\end{proof}

\begin{remark}
   In the above proof, the full strength of Theorem \ref{classical most general de Finetti} was not used. In fact, Winkler \cite{Winkler}, who proved in 1990 a de Finetti--Hewitt--Savage theorem for \textit{jointly} Radon-distributed sequences of exchangeable random variables, could have also obtained Theorem \ref{deFinettiStyle} if they had access to Corollary \ref{what to do with tight measures}. 
\end{remark}

\section{Summary and concluding Remarks}
Starting from a result on an exchangeable sequence of $\{0,1\}$-valued random variables, de Finetti's theorem has had generalizations in several directions. Dubins and Freedman \cite{Dubins-Freedman} had shown that some form of topological condition on the state space is necessary. Our Theorem \ref{classical most general de Finetti} shows that we actually do not need any topological conditions on the state space besides Hausdorffness as long as we focus on exchangeable sequences of Radon distributed random variables. 

The original preprint \cite{Alam-deFinetti-Hewitt-Savage}, on which the present work is based, listed several directions of possible future research. Two such directions have already been fruitfully investigated by other authors, as we describe next:
\begin{enumerate}
\item Towsner \cite{towsner2023aldoushoover} has proved, also using nonstandard methods, an Aldous--Hoover--Kallenberg theorem for exchangeable arrays sampled from a Radon distribution.

    \item  Potaptchik, Roy, and Schrittesser \cite{PotaptchikRoySchrittesser} have proved that any exchangeable sequence of random elements with a Radon (common) distribution must be conditionally independent with respect to the so-called exchangeable sigma algebra. 
\end{enumerate} 

In the present work, we have also obtained two additional results that were not present in the original preprint uploaded to the arXiv in 2020:
\begin{enumerate}[(i)]
    \item Since each tight measure is topologically close to a Radon measure in a sense made precise in Appendix \ref{AppA}, we obtained a de Finetti-style theorem for a sequence of exchangeable random variables with a tight, but not necessarily Radon, distribution (See Theorem \ref{deFinettiStyle}). 
    \item Using known results from the set theory literature, our Theorem \ref{classical most general de Finetti} immediately implies that it would be consistent with the axioms of ZFC to assume that de Finetti's theorem is true for sequences of random variables taking values in any complete metric space. 
\end{enumerate}

This leads us to conclude by pointing out a possible direction for future research which seems to require new tools to approach. Since the existence of real-valued measurable cardinals is a necessary condition for de Finetti--Hewitt--Savage theorem to possibly fail for a completely metrizable state space, it is worthwhile to investigate whether ``de Finetti's theorem for completely metrizable state spaces'' is provable in ZFC. It might very well turn out that the quoted statement is undecidable in ZFC; and indeed we suspect that might be the case, but our current tools are insufficient to fruitfully carry out this investigation. 

\section*{Acknowledgements and history of this manuscript}
This manuscript has evolved over the last five years, the first of those five years being my final year as a PhD student at Louisiana State University. In the initial stages of the work, I greatly benefitted by the support and feedback from my co-advisors Karl Mahlburg and Ambar Sengupta. I am grateful to David Ross for pointing to some references on nonstandard topological measure theory, and to Robert Anderson for helpful comments on the first public draft of the manuscript. 

The philosophical introduction to nonstandard analysis added as an appendix in the 2023 revision of the manuscript came into being because of a practically uncountable number of conversations that I had had with various individuals about ``how to think nonstandardly''. In particular, I am grateful to all six students in my Spring 2023 class MATH 5710/PHIL 6722 (Topics in Logic) at University of Pennsylvania, which was on this topic. I also appreciate Frederik Herzberg at Saarland University, and Soumyashant Nayak at Indian Statistical Institute (Bangalore) for their warm hospitality in Summer 2023. Traveling to and interacting with their departments under the guise of giving talks gave shape to the eventual write-up before it was written. 

The current 2025 revision of the manuscript, drafted at Vector Institute in Toronto whose support I would like to acknowledge, follows (and incorporates into the exposition material from) a shorter 2024 paper \cite{AlamPLMS} that presented the mathematical core of this work in which a key step in the previously announced proof was also corrected---namely what is Theorem \ref{Prokhorov for P_N} in the current revision (which appeared as Theorem 3.11 in the previous 2023 revision). The current revision also corrects the statements of some results in Appendix \ref{AppB}---for instance, in Lemma \ref{measurable} (compare with Lemma B.1 in \cite{AlamPLMS}), we needed to work with a slightly smaller sigma algebra on $\prs$ than the Borel sigma algebra induced by the $A$-topology when discussing the measurability of the empirical measures. This affects only some technical details in the statements in that appendix, most results still being applicable with this modification. 

\begin{appendix}
\section{A philosophically motivated introduction to basic nonstandard analysis}\label{CompactIntroduction}
\subsection{Why is nonstandard analysis relevant for the philosophy of mathematical practice and ethics of mathematics education?}
The long name of this first section of the appendix is a reference to the title of an article of Fenstad \cite{Fenstad1985} from 1985: ``Is nonstandard analysis relevant for the philosophy of mathematics?'' 

Nonstandard analysis is an area of mathematics that provides an alternative way to interpret the abstract object called the \textit{number line} (, the same number line we are used to studying in our classical training in mathematics, where we interpret it as representing the set of so-called real numbers). Under the nonstandard approach to mathematics, in any attempt to (mentally or abstractly) interpret the whole number line as a continuum of numbers, there must also be numbers that are positive but smaller than any positive number we can conceivably mark on the number line. (Such numbers would be called positive \textit{infinitesimals}.) In a rough intuitive sense, the numbers we could physically access as being marked on the number line would correspond to the usual set of real numbers $\mathbb{R}$, while all the numbers (accessible or not) that are identified with points on this abstract line would form a bigger set called ${^*}\mathbb{R}$--- the nonstandard real numbers. 

Because of how our mathematical foundations developed historically and culturally, most people who learn about nonstandard analysis would typically be first trained in their education on the standard real numbers, and they would usually require a further non-trivial amount of training in mathematical logic in order to understand how to think of a \textit{logical extension or completion} of the standard real numbers, which is what the nonstandard real numbers really are in some sense. (Tao \cite{Taocompletion} is a great reference for the intuition behind this perspective.) Since mathematical logic is another subject that does not feature in a significant number of mathematicians' common training, that prerequisite often practically proves to be a stumbling block for many otherwise capable mathematicians trying to learn about nonstandard analysis.  

Thus, writing a research paper (such as this one) on an application of nonstandard analysis in an area of mathematics where the practitioners are not expected to have gone through a training in mathematical logic has several pedagogical challenges. An author could always cite one of the many great introductions to nonstandard analysis available in the library (many of which are cited here: \cite{Albeverio, NATO, Cutland, Goldblatt, combinatorial}). However, there are practical drawbacks of that approach---such an author's work would end up being understandable to more or less only two categories of mathematicians:
\begin{enumerate}
    \item Ones who lie in the intersection of people who already understand nonstandard analysis at a reading level (if not working level) and who are also interested in the specific application in standard mathematics being presented.\footnote{Here, we define \textit{standard mathematics} as the mathematics created using foundations under which ``the number line'' is interpreted to contain no infinitesimals.}
    \item Ones who are sufficiently interested in the topic of the application, and who also lie in the small collection of mathematicians who have the time, energy, and resources to learn an entirely different way of thinking about things as fundamental as the numbers that they have all learnt under a certain (standard) perspective throughout their prior mathematics education and training.
\end{enumerate}

Mathematics is a socio-cultural endeavor, and being able to communicate how one thinks mathematically is a social currency in this enterprise. If there exist mathematicians in the first category above who happen to be more philosophically comfortable with a nonstandard foundation for the continuum after spending some time pondering about their personal philosophy of mathematics, then they would seldom be able to communicate their ways of thinking to most other mathematicians who might usually be interested in applications to standard mathematics, but who might not be intellectually privileged enough to land in the first category above. The mathematicians in this latter class would have only heard about the existence of nonstandard analysis but would seldom have the time, energy, and resources to pursue it after having already made a career in doing mathematics \textit{standardly}. A mathematician is often socio-culturally, as well as economically, tied to a system of doing mathematics where they are expected to keep publishing in the limited aspects of mathematics that they know about in order to practically sustain a career; therefore having the time, energy, and resources to study nonstandard analysis really is a privilege in an academic system where the infinitesimals are otherwise only presented as non-rigorous (even ``not real'') mathematical objects in most mathematicians' personal education and conditioning from a young age. 

While having the opportunities to learn both the standard and nonstandard foundations of numbers is indeed an intellectual privilege, it has a tendency to become the opposite of privilege socio-culturally to a mathematician who recognizes after the fact that they are now only comfortable in thinking nonstandardly about certain key aspects of the essential standard mathematical objects they are interested in. Based on the author's personal lived experiences (that this article is not getting into), the author suspects that there can be nuances in personal psychological traits that can make such a situation quite possible, but performing research at the interface of psychology, philosophy, mathematics, and education, while quite fascinating, is beyond the scope of the current manuscript. 

Unfortunately, as Ely \cite{ely2010nonstandard} demonstrates in a remarkable case study of one of his calculus students, a variation of the above situation might also be possible among students who might be thinking nonstandardly without recognizing the significance behind their thoughts since their educators would be ill-equipped to validate such mathematical thought patterns, given that they themselves have been trained only in the standardized perspective on Calculus. Such students may practically end up giving up on mathematics as the invalidation of their intuitions in their educational career would make them internalize a belief that they do not understand it. Such students might have had the ability to do mathematics if the ways of thinking about numbers in our educational system were more inclusive of their \textit{existing} mathematical intuitions. 

Sarah, the student in Ely's above-cited paper, had developed her own intuition for the number line that was based on a consistent appeal to infinitesimals. She said that she learned ``not from any of her classrooms but that it was her own way of making sense of things.'' She knew that she was ``wrong'' and maintained that she did not ``know the concepts'', primarily because her classroom instruction never vindicated her natural intuitions, which could have been done easily if she was exposed to a course in calculus based in nonstandard foundations (perhaps based on Keisler's landmark book \cite{keisler1976foundations}\footnote{The most recent print of Keisler's book is the 2013 edition \cite{keisler2013elementary}, but Keisler has been continuously updating a freely available online version of the book as well; see \cite{KeislerOnline}.}; see also a recent article of Ely \cite{ely2021teaching} that describes some pedagogical approaches).\footnote{We also provide a possible way to vindicate her intuition briefly later in this appendix, on p. 43.} Her mathematical intuitions about infinitesimals should have empowered Sarah to overcome her mathematical difficulties if our mathematics education system was built in a more intellectually inclusive way --- yet she struggled in mathematics because our system is not built in such an inclusive way. It was only a coincidence that Ely found Sarah and her self-realized nonstandard conceptions as attempts to ``making sense of things''. There could be lots of Sarahs around the world who might have developed similar conceptions while never recognizing that they do not have to be ``wrong'' if one is able to have the right (nonstandard) perspective. 
 
Besides a socio-cultural enterprise, mathematics is also an art. An abstract artist often has their preferred interpretations of the key objects in their work. Different interpretations can be equally valid, and it would be creatively ableist to not encourage students of the artform to discover their own preferred ways of interpreting the fundamental abstract object we call the number line. In the $17^{\text{th}}$ century, Leibniz was artistic enough to come up with a consistent theory of infinitesimals as ``mental fictions'' (we will come back to this later when we discuss Leibniz in a bit more detail), but it took about three centuries before someone (Abraham Robinson) vindicated his intuitions rigorously (meanwhile, the standard interpretation had already become mainstream by then). It is an objectively sad state of affairs that imaginative and artistic students such as Sarah are not vindicated now even after the fact that we now have the resources to vindicate them unlike in the time period before Robinson's seminal work \cite{robinson1961, Robinson}. Instead, such students end up thinking that they are not able to understand mathematics. 

Spoonfeeding one preferred interpretation, despite the presence of an equally valid alternative interpretation, like we do in our current mathematics education, thus makes the socio-cultural enterprise of mathematics underinclusive to equally capable artists who might give up on the art long before their equally valid interpretations could have a chance to be vindicated. In fact, many of them, like Sarah, might not even recognize their mathematical abilities, thus making the question of what mathematical ability means very nuanced. Perhaps we can make a comparison with left-handedness, which for a long time was considered a wrong way to live in our society despite it being possible to live as well with that ``condition'' as someone who is right-handed. Our society and culture in the past were underinclusive to alternative interpretations of how to use our own body parts based on what was personally intuitive to us, despite there being nothing wrong with those interpretations. A similar situation arises at a more abstract level when we standardize the number line in our education.   

Aside from the above issue of intellectual and artistic underinclusiveness, having only one interpretation of the abstract number line in our education might be detrimental to mathematics itself as a practical science. Indeed, those with a utilitarian mindset toward mathematics might find it intriguing to ponder on the possibility that perhaps we might be holding theoretical physics (and hence one of our most fundamental ways of understanding the physical universe) back in case the nonstandard foundation of the continuum is better equipped to modeling our physical theories at very large and very small scales. Fenstad's article \cite{Fenstad1985}, as well as the book \cite{Albeverio} he co-authored with Albeverio, Hoegh-Krohn, and Lindstr\o{}m are great references for readers curious about thinking more along this direction. 

The way we personally think about the (standard) real numbers as mathematics students or practitioners is something that takes shape through several years of unconscious building of intuition from a young age. We are indirectly told at a relatively early age in our textbooks that ``here is a number line---every point on it describes a number that has some sort of measurable magnitude''. The nonstandard interpretation of the number line challenges this belief about the concerned abstract object (that is, the number line) that our educational system tries to promote. Yet, it would be difficult for most people to be open about the alternative nonstandard interpretation even if they do learn about it later on after having spent many years trying to become a mathematician in the standard framework, since they would have already been socio-culturally conditioned to interpret the concerned abstract object one way by then. Thus the topic of nonstandard analysis is intimately connected to an issue of systemic creative suppression, which would perhaps be better discussed in forums devoted to ethics of art and art education, which, while beyond the scope of the current manuscript, would be a good aspect to keep in mind when thinking about the philosophical implications of nonstandard analysis. 

Under the backdrop of the preceding discussion concerning the philosophy of mathematical practice and ethics of mathematics education, a key motivation behind this appendix is quite selfish of the author. The goal is to present just enough nonstandard mathematical thinking at an intuitive and philosophical level that a mathematician who has never seen nonstandard analysis, but who is interested in the standard results of the main body of the paper, can hopefully follow the main body of the paper on their own after reading this appendix. It is a selfish goal because it underlies a human desire to be understood. While most presentations of nonstandard analysis in mathematics research papers focus on the mathematical construction(s) and/or properties of ${^*}\mathbb{R}$ and related (nonstandard) objects, we have (so far) focused on the philosophy of the subject because of the idea that understanding nonstandard analysis effectively seems to require one to fight against years of unconscious socio-cultural conditioning within their mathematics careers---and therefore this author believes that philosophizing about what it is that we are studying is helpful in not succumbing to the intellectually harmful effects of that conditioning. 

As the appendix continues, our focus will soon start morphing from philosophical considerations into the more technical mathematical sophistication needed to understand nonstandard analysis. The emphasis in the mathematics we shall present in this appendix will be on building intuition, and therefore we shall be content with understanding how to think about and use nonstandard analysis, as opposed to worrying about its precise model-theoretic foundations which, otherwise, can be both distracting, and not needed, to the mathematicians interested in the subject for either its artistic, philosophical, or utilitarian implications. 

This is comparable to how a mathematically talented high school student can (and often does) learn a lot about real analysis without learning how the real numbers are precisely defined --- it is enough for them to recognize how what they are learning is describing the reality around them in order to build a good intuition for the axioms of real numbers that they take for granted (oftentimes without precisely knowing what these axioms are, or, indeed, what  \textit{`axioms'} are) in their studies at that level. Such students who become enamored with mathematics at a young age are often pulled in by a sense of beauty that these intuitive mathematical arguments and theories exude. Seeing a precise construction of real numbers before seeing what they intuitively are and how we use them in our mathematical thinking can have a tendency of leaving behind confused students who do not follow what it is that they are constructing. An aim of this exposition, therefore, is to not construct anything and yet build an intuition for why nonstandard analysis can also be a natural way to interpret reality, thus hopefully allowing the patient reader to appreciate nonstandard arguments with the same zeal and vigor of a mathematically interested high school student who gets a thrill out of ``understanding'' real numbers without ever seeing a definition of the real numbers. 


\subsection{Imagine being Leibniz (How to begin thinking about infinitesimals?)}
The topic of nonstandard analysis is often met with apprehension by many practicing mathematicians, as is evident from perhaps the most common questions that a researcher using nonstandard methods faces after presenting their work:

\begin{quote}
    ``That's quite good, but can it be done using only standard analysis? In general, can you prove something using nonstandard analysis that cannot be proven using standard analysis? ''\footnote{For the sake of completion, we mention here the work of Keisler and Henson \cite{Henson--Keisler}, who showed that there do exist results in standard mathematics that can be proven using nonstandard methods (or something equivalent), but not without them. However, in practice, it might be difficult to find an explicit example of such a mathematical result.}
\end{quote}

As explained earlier, a lot of this apprehension can be thought to arise out of our cultural  predisposition to model the geometric (number) line with what we call the \textit{real} numbers. (Fenstad's article \cite{Fenstad1985} is a great further read for this cultural perspective.) We are culturally used to imagining that an abstract straight line can be marked in a way that it can be thought of as \textit{the} set $\mathbb{R}$ --- each element of $\mathbb{R}$ occupying its existence on this abstract line; one point being to the right of another if and only if the corresponding number is greater than the other. 

The intuition that a straight line can be thought of as a continuum of numbers goes back at least to the time of Descartes, if not much earlier. However, ``the continuum of numbers'' is only another abstraction to model the abstract straight line, and one needs to make sense of what we mean by it. The (cultural) practice of imagining that this continuum of numbers \textit{precisely} consists of what we call real numbers is not very old. Indeed, real numbers, as we know them, were not precisely identified until after a lot of work in Calculus had already been done using infinitesimal methods. 

The history and philosophy of numbers is a fascinating topic, but in the interest of not detracting from our goals too much, we will only point out a few ideas that can help us build intuition for the nonstandard numbers that we will soon start working with mathematically. 

At some point in our history, the Greek mathematicians used to think that the continuum of numbers is composed of only the rational numbers. One reasoning behind such a thinking could be that the rational numbers are what we could ever use in our practical measurements in the physical world. (See Arthan \cite{arthan2004eudoxus} for this viewpoint.)

When it was discovered that a right triangle with legs of unit length has a hypotenuse whose length could be shown to be a non-rational number, that was an important moment in our (mathematical) history as it made us feel humbled by the abstraction of the continuum. Earlier we used to think that we \textit{know} it because we could think of numbers in terms of ratios of more tractable whole numbers, but now we recognized that we did not \textit{really} know it, as we could not specify what all the numbers in this continuum are anymore. The famous quote by Bertrand Russell seems to be quite relevant in this conversation: 
\begin{quote}
    ``Mathematics may be defined as the subject where we never know what we are talking about, nor whether what we are saying is true.'' 
\end{quote}

Indeed, we had a general idea of what we were trying to model by the continuum of numbers, but we did not really know what we were talking about if we modeled it by the rational numbers under an assumption that all `numbers' must be of that type. This assumption could almost be thought of as a belief or faith. Under this faith, we thought what we might have been saying about all numbers must be true, but we had no way of knowing it---other than due to an overconfidence in our invented faith. Yet, that faith was important for the development of mathematics and hence of human civilization itself. Readers interested in this discussion might find the essay ``Mathematics and Faith'' \cite{nelsonMathFaith} by Edward Nelson to be thought-provoking. Inspired by Robinson's nonstandard analysis, Nelson notably gave an alternative foundation of mathematics called ``Internal Set Theory'' \cite{NelsonInroduction}, in which the infinitesimals were part of the continuum to begin with. Probability Theory under these foundations was developed in Nelson's remarkable book ``Radically Elementary Probability Theory'' \cite{NelsonREPT}. 

Hrb\'a\u{c}ek \cite{HrbacekFoundations} also developed similar foundations independently of Nelson around the same time. Several other approaches to nonstandard analysis have since been developed. In this paper, we have followed the so-called superstructure formulation (first constructed by Robinson and Zakon \cite{RobinsonZakon}) of Robinson's original nonstandard analysis \cite{robinson1961, Robinson}, which remains the relatively more common approach among practicing mathematicians. An article by Benci, Forti, and di Nasso \cite{eightfold} discusses this and seven other approaches to nonstandard analysis. 

The fact that the superstructure approach that we use in this work is also the one that is practiced most commonly in contemporary applications of nonstandard analysis to standard mathematics is perhaps another cultural artifact---it proceeds within the same framework we typically do standard analysis in, which means there is less to change for the practicing mathematician already culturally trained in standard mathematics. For example, we start out with the set of standard real numbers $\mathbb{R}$, and we create (using tools available within ZFC) a richer structure containing $\mathbb{R}$, which we call a non-standard extension of the real numbers. The nonstandard extensions of other standard structures are obtained in this framework by including those structures as part of the so-called \textit{standard universe} that we extend in this framework. 

Let us take a step back and recognize some of the works prior to Robinson's 1961 breakthrough that led him in the direction of nonstandard analysis. In 1934, Skolem \cite{skolem1934} had constructed nonstandard models of Peano arithmetic. Fourteen years later, Hewitt \cite{HewittRings} constructed the analog of real numbers under such a nonstandard model of arithmetic, which can be thought of as the first construction of a nonstandard extension of the real numbers. Hewitt's construction relied on a technical object called a \textit{non-principal ultrafilter}, which was made possible by Tarski's work \cite{tarski1930} from 1930 that showed that the then-recent set-theoretic foundations of ZFC implied the existence of that object. Hewitt's work anticipated the more general ultrapower construction from about a decade later in the works of {\L o\'{s}} \cite{Los-1955}, and of Frayne, Scott, and Tarski \cite{Frayne-1962, Frayne-1958}. It is this ultrapower construction that most modern treatments of nonstandard analysis use when they present constructions of nonstandard extensions. Originally, Robinson and Zakon \cite{RobinsonZakon} also used such an ultrapower foundation for their superstructure framework for Robinson's nonstandard analysis \cite{robinson1961, Robinson} that we are using in this paper. Notably, such a framework was already anticipated in the lecture notes of Luxemburg \cite{luxemburgCalTech}, who also gave an ultrapower foundation to Robinson's nonstandard analysis one year after Robinson's original announcement of his theory. 

The above was a very short bibliography of  mathematical developments from a very short period in our recent history that led to Robinson's nonstandard analysis. The actual history of infinitesimals is a rather colorful subject that predates these modern works, as mathematicians had been using infinitesimals for many centuries without foundations that were as rigorous, and we shall not attempt to document all those uses here. However, there is one mathematician, Leibniz, whose usage of infinitesimals to discover modern calculus can be thought of as a precursor to Robinon's nonstandard analysis in quite remarkable ways. Therefore before we discuss the mathematical details of Robinson's nonstandard analysis, we shall try to give intuition for it by recalling some of Leibniz's ideas that we build toward next.

As discussed earlier, creating an effective intuition for how to think of nonstandard numbers as modeling the abstract idea of a continuum might require us to try to remove our cultural conditioning of modeling the continuum by what we call real numbers. To that end, let us put ourselves in the shoes of a young student who has not yet been taught about real numbers, but who understands the vague idea of the abstract concept of \textit{numbers}. If you wish, you could mentally travel in time and imagine Leibniz who fits the description, or you could imagine the student Sarah in Ely's mathematics education paper \cite{ely2010nonstandard} cited earlier. 

This student of ours knows that there are things that we call numbers, and she understands examples of some such numbers. She had begun her education by learning about counting numbers (that is, the elements of $\mathbb{N}$). At some point, she started understanding ratios of those counting numbers (namely the positive rational numbers); and yet a bit later in her education, she learnt that we could also talk about \textit{negatives} of the numbers she had seen before. She has seen a number line in her textbook which, she is told, can be used to abstractly represent \textit{all} numbers. But what really are all the numbers? Practically, she can only ever represent a finite number of numbers before finding it difficult to mark more points. For instance, in the following figure, she has marked the rational numbers $\textstyle{\frac{1}{2}, \frac{1}{4},}$ and $\textstyle{\frac{1}{8}}$, along with a few integers:

\begin{center}
 \begin{tikzpicture}
        \begin{axis}[
            axis x line=middle,
            axis y line=none,
            height=50pt,
            width=\axisdefaultwidth,
              axis line style={<->},
            xmin=-2.5,
            xmax=2.5,
        ]
            \addplot coordinates {
                (0.5,0) (0.25,0) (0.125,0)
            };
        \end{axis}
    \end{tikzpicture}
\end{center}

To mark more numbers that are closer to zero, she might use a modern device (or perhaps her imagination) to zoom into an image of her existing number line, and thus have more space to work with. At any point of time, she can only zoom by a certain amount---if she ``zooms $2$x'', for instance, then she gets twice as much space to work with. For any natural number $n$, even if she ``zooms $nx$'', while she might be practically able to ``see'' (and hence mark) more numbers, she would still only be able to mark a finite number of numbers. 

She does know that there are infinitely many numbers, and that there are numbers arbitrarily close to any fixed number. How could she convince herself, even abstractly, that it is sound to assume that \textit{all} numbers that exist can be located on this line \textit{simultaneously} when she knows that practically it is not possible to mark them all as such? 

As discussed earlier, for this student to imagine such a proposition as true is really an exercise in abstract art. Without a description of what we are doing, we merely have a geometric object (namely a straight line) drawn on a piece of paper (or on the screen of an electronic device), which is abstract before interpretation. We have another ``mental abstract object'', the set of all real numbers which, in the standard interpretation, we bijectively identify with the (infinitely many) points that the straight line is imagined to have been made out of. All of this is a purely mental procedure that allows us to interpret the geometric line as a collection of numbers.   

The above was the standard way of interpretation, but a feature of art is that it need not have only one (``correct'' or \textit{consistent}) interpretation, and therefore our student could have reached an alternative interpretation of what a continuum of numbers looks like. To better understand such an alternative interpretation, recall that after having already marked a few points on the number line, to mark a much larger (but finite) number of points simultaneously practically requires us to zoom in at a larger (but finite) level. In this way of thinking, an imagination of \textit{all} numbers that are markable being marked \textit{simultaneously} requires us to be able to imagine a level of zoom larger than all natural numbers---a zoom of infinite magnitude, which, once achieved in some abstract sense, can yield logically consistent ways of zooming in more or less, providing us (at least mentally) a whole spectrum of infinite quantities (``zooms'') of various magnitudes. 

While this description is not rigorous or precise, it could be thought of as being a fairly workable intuition for thinking about infinitesimals. Indeed if one imagines performing an ``infinite amount of zoom while focusing their eyes on a particular number'' on the number line, then the numbers that would be visible to them would be infinitely close to the original particular number. (In other words, these numbers would have infinitesimal distances from the original particular number.) Keep zooming even more, and one would ``see'' even more numbers that are yet closer to the original number. Despite the inaccessible nature of these numbers (in terms of not being able to be ``marked'' by those who can only access finite amounts of zoom---that is, all of us human beings) under this view, the numbers at these scales are still like the other numbers occupying their existence on this line, in a continuum from left to right. One can of course never perform an infinite zoom in reality as we see or (standardly) interpret it, but one can similarly also never simultaneously mark all numbers on the number line in such a standard interpretation of reality. For us to imagine that it can be done in an abstract sense is a type of mental fiction, which is exactly how Leibniz viewed his infinitesimals:

\begin{quote}
    Philosophically speaking, I no more admit magnitudes infinitely small than infinitely great … I take both for mental fictions, as more convenient ways of speaking, and adapted to calculation, just like imaginary roots are in algebra.
(Leibniz to Des Bosses, 11 March 1706; in Gerhardt \cite{gerhardt1960}, II, p. 305).
\end{quote}

The author found the above translation of Leibniz in a seminal article on Leibniz's infinitesimals by Katz and Sherry \cite{KatzSherry}. In that article, Katz and Sherry explain that for Leibniz (as in the intuition we have been developing so far), what we call as real numbers are the ``assignable'' numbers on the abstract number line, while the numbers that are infinitesimals (such as the numbers one would see if one could zoom in at an infinite level while focused at the point representing the number $0$) are mere mental fictions. We might not be able to assign numerical values to these mental fictions, but that does not mean that we cannot work with them, for they follow the same structure as the numbers that are assignable --- both occupy the same abstract number line and are bound by the structure thereof. The fact that they follow the same structure was one of Leibniz's main tools in his development of Calculus. Called the \text{law of continuity}, one of its formulations was summarized by Katz and Sherry \cite[p. 579]{KatzSherry} as follows:
\begin{quote}
    ``The rules of the finite succeed in
the infinite, and conversely.''
\end{quote}

Think of ``rules'' as ``mathematical properties'' or ``mathematical structure'' in some sense that has not been made very precise yet. We can perhaps excuse Leibniz for that as he did not have the language of mathematical logic developed only three hundred years later, which was needed to give a more mathematically precise meaning. Let us interpret ``the rules of the finite'' as intuitively meaning the mathematical facts about the ``finite mathematical universe'', by which we mean the mathematical objects culturally created by mathematicians assuming that whatever the abstract object called ``number'' is can only ever be finite (or accessible, as there would be no infinite or infinitesimal numbers in such a framework). This is essentially \textit{standard} mathematics where the continuum is modeled by standard real numbers---no infinitesimals or other ``mental fictions'' of Leibniz being allowed.

With this perspective, and with the benefit of an understanding of Robinson's nonstandard analysis that developed three hundred years after Leibniz, we can give a possible anthropological interpretation of Leibniz's law of continuity as follows. (By way of this interpretation, we will also begin describing how Leibniz's ideas are rigorized in Robinson's theory, thus providing us an intuition for Robinson's nonstandard analysis before we start filling in more details in the next section(s).)

\subsubsection{Leibniz's philosophy of infinitesimals and implications to the anthropology and ethics of mathematical knowledge}
Unlike now, mathematicians at the time of Leibniz had not invented a standardized way of understanding the mathematical nature of the abstract number line. There was no precise definition of what it meant to be a \textit{real number}. Philosophically, we knew that what we were modeling had connections with the nature of reality, as we had always used numbers to explain the real world, which we could only access a certain part of. There had been mathematicians (see Section 2 of Katz and Sherry \cite{KatzSherry}) who had obtained verifiable results about the real world by assuming that the magnitudes (numbers) we could access need not constitute \textit{all} magnitudes. Perhaps some of those mathematicians might even feel philosophically justified in assuming so because they might have a spiritual or religious belief in the infinitely large, and, thus logically (, by consequently taking reciprocals,) in the infinitely small. 
    
    Leibniz's philosophy entails that even if you do not believe in a metaphysical ontology of the infinitely large or infinitely small, it is possible to logically argue about them as mental fictions. And the logical structure of a mathematical universe in which they are assumed to exist (``the rules of the infinite'') would match that in which they are not assumed to exist (``the rules of the finite''). For Leibniz, the infinite (or nonstandard universe in modern terminology) only existed in the mind and played an instrumental role in describing the finite (or standard universe in modern terminology) because there was a correspondence of truths between the two via his law of continuity. This was made precise later by Robinson and his \textit{transfer principle} (to be discussed in the upcoming section) when the language of mathematical logic had developed sufficiently to describe this phenomenon more rigorously. 
    
    Culturally, mathematics after the time of Leibniz could have proceeded in a manner in which actual numbers of infinite magnitudes (and, consequently, of infinitesimal magnitudes) would become part of the standardized way of thinking about the number line. Fenstad \cite[p. 298]{Fenstad1985} sketches one such idea in which the standard real numbers could have been obtained from the rationals by first extending the rationals to a bigger set consisting of infinitesimals, instead of the other way around that happened in our history. Nelson's internal set theory \cite{NelsonInroduction} is another, perhaps much more radical, example of how our mathematics could have proceeded had the developments in mathematical logic from the $20^{\text{th}}$ century happened earlier than when we had settled on our standard interpretation of the continuum. Joel David Hamkins' recent article \cite{Hamkins_2024} explores such an alternate history of mathematics in much more detail, explaining how the continuum hypothesis could have become ``a fundamental axiom of set theory, one furthermore necessary for mathematics and indeed indispensable for making sense of the core ideas underlying calculus,'' if only Newton and Leibniz were able to provide ``somewhat fuller accounts of their ideas about infinitesimals.''

    In particular, the law of continuity shows that there could be two types of ontologies of the real number line (one in which we think of it as the standard $\mathbb{R}$, and another where we think of an ordered field containing $\mathbb{R}$ but also containing infinitesimals). \textit{Believing} in either of these types of ontologies is a choice that one makes --- a personal choice to not believe in infinitesimals does not preclude one from still discovering new mathematics by treating the infinitesimals as mental fictions like Leibniz did, and like Robinson did as well.\footnote{The following quote from Katz and Sherry \cite[p. 572]{KatzSherry} captures Robinson's stance on this matter:
    ``Like Leibniz, Robinson denies that infinitary entities are real, yet he promotes the development of mathematics by means of infinitary concepts (Robinson 1966, p. 282; 1970, p. 45). Leibniz’s was a remarkably modern insight that mathematical expressions need not have a referent, empirical or otherwise, in order to be meaningful.''}

    Once Robinson was able to provide rigor to Leibniz's idea in his nonstandard analysis, Nelson \cite{NelsonInroduction, NelsonREPT} saw it for what it really was---an alternate interpretation of what our numbers are, and how we are perhaps epistemologically incapable of actually knowing which ontology is the ``real'' one. He showed that one could truly believe in either ontologies and still arrive at essentially the same knowledge, and yet one could actually never be sure of either of those beliefs without faith  at different levels. Nelson \cite{nelsonMathFaith} talks about faith at the level of a belief in the consistency of the foundations of mathematics itself---but even someone who does have faith in the general foundations would then still require some amount of faith in their chosen ontology of the numbers, as they can never be sure that they are not living in a universe where infinitesimals are ``real'' but just not accessible (this is an intuition that we will develop in the next section). Leibniz's law of continuity essentially says that it does not matter which ontology of numbers you personally believe in --- if you already believe in the infinite and infinitesimal, then you can attempt to learn more about it by transferring the ``rules of the finite'' into the realm you cannot access; on the other hand, even if you do not believe in that realm in any metaphysical sense, you have to accept the fact that the truths about your standard accessible universe are still bound by the logical rules of the mental fiction of the infinitesimals.
        
We should be clear that our emphasis in this article on the socio-cultural nature of how mathematics developed through history and how that impacts the mathematics that we currently work on as mathematicians does not mean that we are debating about the absolute nature of mathematical truths or any such thing. We are not advocating either a platonist or a non-platonist viewpoint of mathematics. We are, in fact, not claiming to know whether infinitesimals exist or not in any real sense. 

Indeed, those are messy philosophical debates that one does not have to have strong opinions on in order to still recognize the important fact that mathematics has been an important tool for creating theories that help us better understand the universe throughout human history. Being too overconfident about our current ways of understanding the universe has never been conducive to the seemingly eternal quest of humanity to understand the universe it inhabits. In our past, we used to believe in lots of erroneous things that held us back in our understanding of the universe until new mathematical ideas were needed to get us closer to truth (a belief in geocentrism in various cultures comes to mind). In a sense, the development of Calculus and the ``rules of the infinite'' have changed the ways in which we understand our universe for the better, yet we are not doing ourselves any favor by standardizing infinity and thus limiting access to the imaginative people such as Ely's student Sarah \cite{ely2010nonstandard}---such people could have been empowered in their lives and could have empowered mathematics back had mathematics as a socio-cultural endeavor been more inclusive to their equally valid ways of thinking. 

Henson and Keisler \cite{Henson--Keisler} show that there are truths in arithmetic that can be proven using nonstandard thinking (or something equivalent) but not without. While this proof-theoretic strength of nonstandard analysis is a great feature to have, in an inclusive system of mathematics education it should not matter even if nonstandard analysis was incapable of proving ``new'' things. The ``things'' that a mathematics student studies at much lower levels are still things that they struggle with and we, the mathematics educators, have a moral obligation to uphold and empower our students' personal intuitions if it is possible to do so. 

We shall provide a more mathematical discussion of Robinson's nonstandard analysis in the remainder of this appendix, but, as we have been trying to emphasize, the ideas that shall be made precise were already philosophically anticipated by Leibniz's law of continuity. Indeed, the basic premise of nonstandard analysis is the existence of a standard universe which is contained inside a bigger and richer (the technical term will be \textit{saturated}) nonstandard universe that structurally behaves the same in some precise model theoretic sense (a property we shall call the \textit{transfer principle}).
    


\subsection{The ``alien intuition'' for Robinson's nonstandard analysis}
As already alluded in the previous paragraph above, Robinson's nonstandard analysis can be thought of as a theory in which the central object of study is a metamathematical correspondence between two abstract mathematical constructions: one called the standard universe and one called the non-standard universe. Using such a correspondence (together with an access to infinitesimals in the latter universe) will allow us to interpret certain standard mathematical concepts (such as those in calculus) in an arguably more intuitive manner allowing us to obtain new results in the standard universe. Philosophically, we are not claiming whether infinitesimals exist or not in a metaphysical sense. Let us now make all this a bit more precise mathematically.

Consider the so-called \textit{standard universe} $\mathfrak{S}$, which \textit{(at least)} includes all the standard mathematical objects that appeared in the standard theorems in the main body of this manuscript; or in general, it may include all mathematical objects that appear in any given standard discourse. Here, we are using the word ``standard'' as an adjective. Thus, for instance, a \textit{standard discourse} is the type of mathematics that does not involve any nonstandard analysis.

When we write down a mathematical statement about our standard universe, the mathematical symbols used to describe our statement are philosophically just abstract symbols whose meanings depend on how we (mathematically) interpret them. The point we are trying to emphasize will be clearer from a thought experiment. 

Consider the following sentence in formal logic:
\begin{align}\label{square root}
    \forall x \in \mathbb{R} ((x > 0) \rightarrow \exists y \in \mathbb{R} (y = x \cdot x)),
\end{align}
Imagine communicating it to an alien civilization from a different universe (which we call the \textit{nonstandard universe}). For the sake of discussion, let us pretend that these aliens somehow know how to read formal logic symbols (`$\forall$' (``for all''), `$\exists$' (``there exists''), `$\lor$' (``or''), `$\land$' (``and''), `$\rightarrow$ (``implies''), `$=$' (``equals''), and punctuation using parentheses `:', `$($' and `$)$'), and suppose further that they know the language of set theory (which just has one additional symbol `$\in$' (``belongs to'')). While they know how to interpret these basic foundational symbols, they do not know our standard mathematical practices. In other words, they do not know which set the symbol $\mathbb{R}$ stands for, but they can infer from the sentence \eqref{square root} that it must stand for some set on which there is a binary relation $>$ and a function $\cdot \colon \mathbb{R} \times \mathbb{R}$. They can also notice that there is at least one element $0$ with the property that anything related to it via the relation $>$ can always be expressed as a square. 

If this alien civilization sees this sentence from us and nothing else, then they might think: ``Hmm. We also know of many sets in our universe that have this property. For all we know, the set $\mathbb{R}$ could just be a singleton that contains some element that the humans are calling $0$.'' But then they would quickly recognize their thought as incorrect, as they would soon see some other sentences we have sent about $\mathbb{R}$ that make it clear that it is not just a singleton. For instance, perhaps they next see all the ordered field axioms that $\mathbb{R}$ satisfies. The aliens would recognize that whatever $\mathbb{R}$ is, it must be an ordered field \footnote{The aliens would probably not call it an ``ordered field'' literally, as they would have come up with a different name in their own natural language for the concept of an ``ordered field'', but we shall not make this pedantic distinction explicit henceforth.}, which is a concept they can define using the language of set theory that they know. (And since they remember our first sentence \eqref{square root} that we sent, they further know that $\mathbb{R}$ is an ordered field in which every positive element has a square root.) 

They still do not know what the humans meant by $\mathbb{R}$, but they now have a better intuition of how it might \textit{look} like because they understand its \textit{structure} better. 

Even though this is not physically possible, imagine for the sake of this discussion that we are able to send \textit{all} of the sentences that we can express about objects in $\mathbb{R}$ (and about other standard objects of interest). In particular, this would include statements about specific numbers such as the following:
\begin{align}
    3.9 &< 4 \nonumber \\ 
    3.99 &< 4 \nonumber \\ 
    3.999 &< 4 \ldots \label{Sarah inequality}
\end{align}
To make sure that these are sentences about elements of the same set $\mathbb{R}$ which we were talking about earlier, there would also be a sentence $c \in \mathbb{R}$ for each of the uncountably many real numbers $c$. 

Suppose it turns out that their (mathematical) universe is actually rich enough to have an interpretation of \text{all} of our standard mathematical symbols in a consistent way. So, it is not just that they have at least one set in their universe that is an ordered field with the square root property, but they have a set, denoted by ${^*}\mathbb{R}$, which satisfies this and all other properties of $\mathbb{R}$ that we may write as sentences in our standard mathematical language. 

Well, perhaps then we might think they have also studied our set of real numbers. However, some general results in model theory guarantee that this is not always the case (as one can show the existence of nonstandard universes with saturation and transfer property that were informally alluded to at the end of the previous section). 

Philosophically, both the aliens and the humans are modeling the same abstract idea: the continuum of the number line, but their models will be different because they have different notions of \textit{infinities}. The aliens somehow have access to actual numbers that are bigger than all standard numbers (like how Leibniz had, at least mentally, in his formulation of Calculus), yet the set of all of the aliens' numbers behaves structurally similar to $\mathbb{R}$. Doing nonstandard analysis in modern mathematical practice can be thought of as being a neutral observer living in neither the standard nor the nonstandard universe, but being aware of the facts about both universes, traversing between them as needed. The fact that this traversal of truth between different ideologies of numbers is possible at all is a remarkable achievement of the $20^{\text{th}}$-century developments in modern mathematical logic. 

Psychologically, it might be better to not think of the aliens' universe in the above intuition as a different universe than ours. We could actually imagine that these aliens live in the same physical universe as us, but their cognitive abilities are infinitely higher than us in ways that lets them perceive and measure the world at infinitely finer scales than us. This is perhaps comparable to the popular intuition that a hypothetical being who lives and perceives only in a two-dimensional plane might never know what lies outside their plane of existence, even though there is a three-dimensional world containing that plane. Along similar lines, we can imagine beings (namely the aliens in our intuition) who experience the physical universe at scales incomprehensible by us. When they count something, they obtain a number in ${^*}\mathbb{N}$ and that number may as well be infinite to us (if it is bigger than all of our natural numbers) even though it is finite to them. Just like us, they can think of the set of all counting numbers, but they can never isolate the set $\mathbb{N}$ of \textit{limited counting numbers}, as they do not understand the \textit{limited} adjective that requires a different perception of reality in order to make sense. To them, all elements of ${^*}\mathbb{N}$ are limited.

The discussion in the previous paragraph can be thought of as an intuition for an important nonstandard analytic concept called \textit{internal sets}. Internal sets are what the aliens can construct if they perceive the world in the manner that we have described above. 

For instance, for each of \textit{our} counting numbers $n$, (that is, for each $n \in \mathbb{N}$), we can think of, and hence ``construct'', the set $[n]$ of the first $n$ counting numbers. In the same manner, each element $N$ of ${^*}\mathbb{N}$ is just a number up to which the aliens can count, and so they can internally think of the set $[N] \defeq \{n \in {^*}\mathbb{N} : n \leq N\}$. In this sense, the sets $[N]$ (for each $N \in {^*}\mathbb{N}$) are internal in the nonstandard universe, while the set $\mathbb{N}$ is not, as the aliens are not capable of mentally visualizing this set at all. This intuition is reflected in the technique called \textit{overflow}, one version of which says that an internal subset of ${^*}\mathbb{N}$ that contains all finite natural numbers must ``overflow'' and also contain an element from ${^*}\mathbb{N} \backslash \mathbb{N}$, which must necessarily be infinite from the standard/human perspective because of the following simple argument based on the transfer principle that we have informally built so far. The motto for us to understand, like for Leibniz, will be that ``the rules of the finite succeed in the infinite, and conversely''! 

\begin{proposition}\label{infinite natural0}
If $N \in {^*}\mathbb{N} \backslash \mathbb{N}$, then $N$ is infinite in the sense that $N > n$ for all $n \in \mathbb{N} \subseteq {^*}\mathbb{N}$. We express this by writing $N > \mathbb{N}$, and calling $N$ a \textit{hyperfinite natural number}.
\end{proposition}

\begin{proof}
    Suppose, if possible, that $N \in {^*}\mathbb{N} \backslash \mathbb{N}$, yet there exists $n_0 \in \mathbb{N} \subseteq {^*}\mathbb{N}$ such that $n \leq n_0$. It is true in the standard interpretation of numbers that a counting number less than or equal to the particular counting number $n_0$ must equal one of the finitely many elements of $[n_0]$. By transfer principle, the same truth must hold in the nonstandard universe. Namely:
    \begin{align}\label{transfer illustration}
        \forall n \in {^*}\mathbb{N} \left\{(n \leq n_0) \rightarrow \left[(n = 1) \lor \ldots \lor (n = n_0) \right] \right\}.
    \end{align}
This, when applied to $n = N$, contradicts our assumption that $N \not\in \mathbb{N}$.  
\end{proof}

Let us analyze the above proof a bit to build intuition for the transfer principle that we will explore in more detail in the upcoming section. The reason that \eqref{transfer illustration} is true is because we knew that the following sentence was true about the standard universe:

\begin{align}\label{transfer illustration2}
        \forall n \in \mathbb{N} \left\{(n \leq n_0) \rightarrow \left[(n = 1) \lor \ldots \lor (n = n_0) \right]\right\}.
    \end{align}

Thus, all we had to do was to identify where we were quantifying over a standard object in \eqref{transfer illustration2}, and then quantify it over the nonstandard interpretation of that object instead: truths about $\mathbb{N}$ expressible in our standard mathematical language using formal logic symbols get transferred and remain true even if our universe of counting numbers is now ${^*}\mathbb{N}$ --- the rules of the finite are succeeding in the infinite as Leibniz said!

Note that the fact that $n_0$ was a particular finite number fixed earlier was useful in this proof as well, as otherwise the ``$\ldots$'' in the logical expressions \eqref{transfer illustration} and \eqref{transfer illustration2} could not be made more precise, and so we could not really think of either of those expressions as \textit{well-formed sentences} in our formal language. The intuition to keep in mind is that the statements about the standard universe that we are claiming to be equiveridical to their nonstandard interpretations must be first writable in our standard formal language in a precise and unambiguous way; otherwise exactly what would we be reinterpreting? This would exclude sentences that use imprecise expressions such as  ``$\ldots$'', unless we can make them precise using finitely many symbols in the specific contexts we are interested in. This is what happened in the above proof.

For instance, if $n_0 = 3$, then the sentence \eqref{transfer illustration2} is merely an abbreviation of 
\begin{align}\label{transfer illustration3}
        \forall n \in \mathbb{N} \left\{(n \leq 3) \rightarrow \left[(n = 1) \lor (n = 2) \lor (n = 3) \right]\right\},
    \end{align}
and is hence a valid sentence whose truth gets transferred to the nonstandard universe (with a similar reasoning holding for all fixed $n_0 \in \mathbb{N}$). 

Proposition \ref{infinite natural0} showed that if the aliens can perceive any new counting number in their mathematical universe, then the transfer principle would imply that any such number would also be infinite from \textit{our} perspective. We built the intuition for this by being artistic and imagining that these aliens somehow perceive reality at a different scale than us, thus seeing lots of new numbers that are accessible to them, but not to us. However, this is clearly not a mathematical argument. The foundational mathematical reason we can always use this intuition is because we can assume there exist nonstandard universes that are \textit{saturated}, a concept we shall explain in the next section where we convert the intuition presented so far into a bit more rigorous mathematical exposition.

Yet, we do not need to understand saturation to do basic Calculus if we believe in an existence of a bigger nonstandard universe and understand the intuition of transfer principle between the two universes that we presented so far. We will thus be able to finish this section by finally giving an exposition of how to think about concepts from Calculus using the actual infinitesimals provided by nonstandard analysis. For brevity of exposition, we shall focus on understanding limits and continuity from a nonstandard perspective, with the aim of providing a bit more flavor of how transfer principle can be applied to various situations.

Let ${^*}\mathbb{R}_{\text{fin}}$ denote the set of those nonstandard real numbers that are not infinite. One can show that this is a \textit{ring}, but not a \textit{field} (as the multiplicative inverses of nonzero infinitesimals are not in this set). The next result says that one can think of a finite nonstandard real number $z$ as having a real part, and an infinitesimal part. Intuitively, each finite nonstandard real number must arise by ``performing an infinite zoom while focusing'' at a specific real number which we call its standard part.

\begin{proposition}\label{standard part map}
For all $z \in {^*}\mathbb{R}_{\text{fin}}$, there is a unique $x \in \mathbb{R}$ (called the \textit{standard part} of $z$) such that $(z - x)$ is infinitesimal. We write $\st(z) = x$ or $z \approx x$. 
\end{proposition}
\begin{proof}
     Fix $z \in {^*}\mathbb{R}_{\text{fin}}$. Use the least upper bound property of $\mathbb{R}$ in order to find the supremum $x$ of the set $\{y \in \mathbb{R}: y \leq z\}$. We claim that this suffices, but leave the routine verification to the curious reader.
\end{proof}

\begin{remark}\label{st part morphism}
    It is not difficult to check that $\st(ax + by) = \st(a)\st(x) + \st(b)\st(y)$ for all $a,b,x,y \in {^*}\mathbb{R}_{\text{fin}}$. For those who are algebraically inclined, we may thus think of the map $\st \colon \mathbb{R}_{\text{fin}} \to \mathbb{R}$ as a surjective ring homomorphism with the infinitesimals as the kernel. This expresses the field of real numbers as a quotient of the ring of finite nonstandard numbers. 
\end{remark}

For $x, y \in {^*}\mathbb{R}$, we use the notation $x \approx y$ to mean that $(x - y)$ is an infinitesimal. We also say ``$x$ and $y$ are infinitely/infinitesimally close''. 

Recall that for a standard sequence $\{a_n\}_{n \in \mathbb{N}}$ of real numbers, we say that $L \in \mathbb{R}$ is a \textit{limit} if for each $\epsilon \in \mathbb{R}_{>0}$, there exists $n_{\epsilon} \in \mathbb{N}$ such that $\abs{a_n - L} < \epsilon$ for all $n \in \mathbb{N}_{>n_\epsilon}$. The following nonstandard characterization not only provides rigor to how we intuitively think of limits (basically through the motto that whenever $n$ is ``very large'', $a_n$ must be ``very close to its limit''), but also has the uniqueness of the limit of a sequence (something that one proves separately in the standard way of thinking about limits) essentially built into the characterization.

The key idea is that any sequence of real numbers $(a_n)_{n \in \mathbb{N}}$ is really just a function, say $a \colon \mathbb{N} \to \mathbb{R}$, and hence has a nonstandard interpretation which is now a function $^\ast a \colon {^*}\mathbb{N} \to {^*}\mathbb{R}$. This nonstandard interpretation agrees with the original interpretation when restricted to the standard natural numbers, because given the real numbers $(a_n)_{n \in \mathbb{N}}$, the sentences `$a(n) = a_n$' about the function $a$ transfer and yield `${^*}a(n) = a_n$' in the nonstandard universe for each fixed $n \in \mathbb{N}$. Because of this consistency with the original sequence, we can actually drop the asterisk, and just write $a_n$, always interpreting it as ${^*}a(n)$ without any ambiguity regardless of whether $n$ is finite or hyperfinite. (Recall by Proposition \ref{infinite natural0} that any $N \in {^*}\mathbb{N} \backslash \mathbb{N}$ is hyperfinite which we denote by writing $N > \mathbb{N}$.) 

\begin{proposition}\label{limits of sequences}
Let $\{a_n\}_{n \in \mathbb{N}}$ be a sequence of real numbers. Then, $\lim a_n = L \iff a_N \approx L$ for all $N > \mathbb{N}$. 
\end{proposition}

\begin{proof}
  First, suppose $\lim a_n = L$. Fix $\epsilon \in \mathbb{R}_{>0}$. By the definition of limit, we find an $n_{\epsilon} \in \mathbb{N}$ such that the following holds:
    \begin{align*}
        \forall n \in \mathbb{N}((n \geq n_{\epsilon}) \rightarrow (|a_n - L| < \epsilon)).
    \end{align*}
Interpreting the truth of this sentence in the nonstandard universe (or transferring), we obtain:
     \begin{align*}
        \forall n \in {^*}\mathbb{N}((n \geq n_{\epsilon}) \rightarrow (|a_n - L| < \epsilon)).
    \end{align*}
Since any hyperfinite $N$ also satisfies $N > n_{\epsilon}$ for all $\epsilon \in \mathbb{R}_{>0}$, it follows that $0 \leq \abs{a_N - L} < \epsilon$ for all $\epsilon \in \mathbb{R}_{>0}$ and $N > \mathbb{N}$. Hence $a_N \approx L$ for all $N > \mathbb{N}$.

Conversely, suppose $a_N \approx L$ for all $N > \mathbb{N}$. Then, for any given $\epsilon \in \mathbb{R}_{>0}$, the truth of the following sentence in the nonstandard universe is witnessed by any hyperfinite $m$ (and transfer principle then yields the appropriate sentence in the standard universe that we need for this $\epsilon$):
\begin{align*}
    \exists m \in {^*}\mathbb{N} ~~\forall n \in {^*}\mathbb{N} \left((n \geq m) \rightarrow (|a_n - L| < \epsilon) \right).
\end{align*}
\end{proof}

Note that in the above proof, the function $\abs{\cdot}$ that appeared in the statements for the nonstandard universe was really the nonstandard interpretation ${^*}\abs{\cdot} \colon {^*}\mathbb{R} \to {^*}\mathbb{R}$ of the usual standard absolute value function on $\mathbb{R}$, where we had suppressed the asterisk for brevity of exposition, something that was formally permissible because the nonstandard interpretation must agree with the original function when restricted to the original domain (by transfer). It is a useful fact (an easy verification can be done using transfer) that ${^*}\abs{\cdot}$ agrees with the natural absolute value map on ${^*}\mathbb{R}$ as an ordered field---otherwise, if that were not the case, then dropping this asterisk could have been a source of confusion.

A similar application of transfer in a different context provides us an intuitive nonstandard characterization of continuous functions on the real line. Recall that a function $f \colon \mathbb{R} \to \mathbb{R}$ is called \textit{continuous} at some point $x \in \mathbb{R}$ if for each $\epsilon \in \mathbb{R}_{>0}$, there exists $\delta \in \mathbb{R}_{>0}$ such that $\abs{y - x} < \delta$ implies $\abs{f(y) - f(x)} < \epsilon$.  

The nonstandard characterization says that $f \colon \mathbb{R} \to \mathbb{R}$ is continuous at $x \in \mathbb{R}$ if and only if whenever we take an input infinitesimally close to $x$, we obtain an output that is infinitesimally close to $f(x)$. (This is a sensible statement because we can work with the nonstandard or alien's interpretation ${^*}f \colon {^*}\mathbb{R} \to {^*}\mathbb{R}$ of the original function, which would agree with $f$ when restricted to $\mathbb{R}$.) More formally, we have:

\begin{proposition}\label{real-valued continuous}
    A function $f \colon \mathbb{R} \to \mathbb{R}$ is continuous at $x \in \mathbb{R}$ if and only if ${^*}f(y) \approx f(x)$ whenever $y \approx x$. 
    
    In view of the map $\st \colon {^*}\mathbb{R}_{\text{fin}} \to \mathbb{R}$, we may rewrite this characterization of continuity at $x$ as saying that ${^*}f(\st^{-1}(x)) \subseteq \st^{-1}(f(x))$.
\end{proposition}

\begin{proof}
    Suppose $f$ is continuous at $x \in \mathbb{R}$. Fix $y \in \st^{-1}(x)$. We want to show that ${^*}f(y) \approx f(x)$. To that end, we fix an arbitrary $\epsilon \in \mathbb{R}_{>0}$ and we now aim to show that $\abs{{^*}f(y) - f(x)} < \epsilon$. Given this $\epsilon \in \mathbb{R}_{>0}$, the standard definition of continuity yields a number $\delta \in \mathbb{R}_{>0}$ such that the following is true:
    \begin{align*}
        \forall z \in \mathbb{R} \{\left(\abs{z - x} < \delta \right) \rightarrow \left(\abs{f(z) - f(x)} < \epsilon \right)\}.
    \end{align*}
Transferring its truth to the nonstandard universe yields (noting that $x$ and $f(x)$ are just real constants, so they are interpreted ``as they are''):
\begin{align*}
     \forall z \in {^*}\mathbb{R} \{\left(\abs{z - x} < \delta \right) \rightarrow \left(\abs{{^*}f(z) - f(x)} < \epsilon \right)\}.
\end{align*}
Since $\abs{y - x} < \delta$ is true (in fact, $\abs{y - x} \approx 0$ by assumption, so $\abs{y-x}$ is smaller than all fixed positive real numbers), we obtain $\abs{{^*}f(y) - f(x)} < \epsilon$. Since $\epsilon \in \mathbb{R}_{>0}$ was arbitrary, we obtain ${^*}f(y) \approx f(x)$, as desired.

Conversely, suppose ${^*}f(\st^{-1}(x)) \subseteq \st^{-1}(f(x))$. Fix $\epsilon \in \mathbb{R}_{>0}$. Then the following is true in the nonstandard universe by assumption (as any positive infinitesimal $\delta$ witnesses its truth):
\begin{align*}
    \exists \delta \in {^*}\mathbb{R}_{>0} \left[\forall z \in {^*}\mathbb{R} \{\left(\abs{z - x} < \delta \right) \rightarrow \left(\abs{{^*}f(z) - f(x)} < \epsilon \right)\} \right].
\end{align*}
Transferring this sentence to the standard universe shows that there is a standard $\delta \in \mathbb{R}_{>0}$ as well with the requisite property. 
\end{proof}

Note that we can, of course, give analogous characterizations if the domain of the function is not all of $\mathbb{R}$. For a subset $A \subseteq \mathbb{R}$, a function $f \colon A \to \mathbb{R}$ is continuous at some $a \in A$ if and only if ${^*}f(y) = f(a)$ for all $y \in {^*}A$ such that $y \approx a$. In fact, we can change both the domain and the co-domain to general topological spaces, and nonstandard interpretations will still allow us to think infinitesimally in that setting despite not having a natural notion of distance or metric (see Proposition \ref{general continuity characterization}). To work with nonstandard ideas in such generality, we need to be more precise about the standard universe we are extending. We build toward that generality by first fleshing out some more mathematical details of our alien intuition next.

\subsection{More mathematical details of our alien intuition}
In technical model theory terms, one thinks of the nonstandard universe as a \textit{saturated elementary extension} of the standard universe in some sense we have not made precise in this article. But we shall not need to understand such model theoretic terminology rigorously in order to understand the gist of what is happening. We shall be able to continue the ``alien intuition'' from the previous section, while still able to be more rigorous about what is truly happening mathematically!    

Here is the gist of transfer principle---if the humans use a symbol $A$ for some specific standard set, whose formal properties have been written among the set of formal sentences that are true in the standard universe, then the aliens would identify a set ${^*}A$ in their universe that would satisfy all the same formal properties. 

Let us consider a toy example. For each particular real number $c$, there is a sentence ``$c \in \mathbb{R}$'' which is true in our standard theory. By transfer, therefore, the sentence ``$c \in {^*}\mathbb{R}$'' is true in the nonstandard universe. In this way, ${^*}\mathbb{R}$, called the set of nonstandard real numbers, contains $\mathbb{R}$. Similarly, we have $\mathbb{N} \subseteq {^*}\mathbb{N}$, etc. \footnote{If this is unsatisfactory because of the usage of the symbol $c$ that is seemingly interpreted the same way in both universes, here is a way to intuitively think about what is happening. For any two constants  for which we use the symbols $a, b$ in $\mathbb{R}$ in the standard universe, we must have communicated to the aliens which one of $a = b$, $a < b$, and $a > b$ is true, and thus (, because ${^*}\mathbb{R}$ has the same formal properties as $\mathbb{R}$,) there should be a unique pair $(^*a, ^*b) \in {^*}\mathbb{R} \times {^*}\mathbb{R}$ corresponding to each standard pair $(a,b) \in \mathbb{R} \times \mathbb{R}$ such that ${^*}a$ and ${^*}b$ satisfy the same sorts of relations (interpreted appropriately in the nonstandard universe) with each other and with the rest of the elements of ${^*}\mathbb{R}$ as what $a$ and $b$ have with each other and with the rest of the elements of $\mathbb{R}$. Thus, for example, if $a = b^2$, then so is ${^*}a = {^*}b^2$. 

For all intents and purposes, we might as well make the identification that ${^*}a = a$ for any standard real number $a$. That is, the set 
\begin{align}
    \{{^*}a : a \in \mathbb{R}\} \subseteq {^*}\mathbb{R}
\end{align}
is identified as the copy of $\mathbb{R}$ inside ${^*}\mathbb{R}$.}

Among the properties of $\mathbb{R}$ that we can write down also include second-order properties (that is, those that talk about universal or existential properties of subsets, as opposed to those of elements of a particular set), but disguised as first-order sentences. For instance, we can quantify over the elements of the power set $\mathcal{P}(\mathbb{R})$ if we want to express a statement about subsets of $\mathbb{R}$. (Recall that we do not have `$\subseteq$' as an allowed symbol in our formal language, but we do have the set-membership relation symbol `$\in$' in it.) Why are we being pedantic about this aspect? Isn't it true that if we can write down a statement about all subsets of $\mathbb{R}$, then transfer should imply that the nonstandard interpretation is true for all subsets of ${^*}\mathbb{R}$? This is actually not true, and understanding why leads us to thinking about what internal sets (introduced more intuitively in the previous section) actually are. 

To build intuition before the definition comes, here is a simple statement that says that all elements of $\mathcal{P}(\mathbb{R})$ are \textit{subsets} of $\mathbb{R}$, but without using the `$\subseteq$' symbol:
\begin{align}\label{subset sentence}
    \forall A \in \mathcal{P}(\mathbb{R}) \left(\forall x ((x \in A) \rightarrow (x \in \mathbb{R})) \right).
\end{align}

If ${^*}(\mathcal{P}(\mathbb{R}))$ is the actual object in the nonstandard universe that the aliens must interpret in order to transfer the standard universe truths expressed using the symbol $\mathcal{P}(\mathbb{R})$, then transfering the above standard sentence expresses the following truth about the nonstandard universe:
\begin{align}\label{subset sentence 2}
    \forall A \in {^*}(\mathcal{P}(\mathbb{R})) \left(\forall x ((x \in A) \rightarrow (x \in {^*}\mathbb{R})) \right).
\end{align}
Thus, whatever object ${^*}(\mathcal{P}(\mathbb{R}))$ is, its elements are subsets of ${^*}\mathbb{R}$. We often drop the additional parentheses and abbreviate ${^*}(\mathcal{P}(\mathbb{R}))$ as ${^*}\mathcal{P}(\mathbb{R})$, so that we can rewrite the sentence \eqref{subset sentence 2} as saying that ${^*}\mathcal{P}(\mathbb{R}) \subseteq \mathcal{P}({^*}\mathbb{R})$. Let us try to understand the structure of ${^*}\mathcal{P}(\mathbb{R})$ a bit more. 

Here is a slightly more complicated statement which says that every non-empty subset of the real numbers that is bounded above has a least upper bound:
    \begin{align}\label{lub sentence}
    &\forall A \in \mathcal{P}(\mathbb{R}) \nonumber \\
    &\langle ~[{\exists c \in \mathbb{R}(c \in A)} \land {\exists x \in \mathbb{R}} {(\forall y \in \mathbb{R}} \{{(y \in A) \rightarrow (y \leq x)} \})] \rightarrow \nonumber \\
&{ \exists z \in \mathbb{R}} \nonumber \\
&\{{(\forall y \in \mathbb{R}} ~ [ {(y \in A) \rightarrow (y \leq z)} ]) \nonumber \\
&\hspace{2.7cm} \land {{[\forall w \in \mathbb{R} {( \forall y \in \mathbb{R}} { \{ [} {(y \in A) \rightarrow (y \leq w)} {]}} \rightarrow ({z \leq w})\}} { )} {]}\} ~\rangle. 
\end{align}


By transferring its truth, all non-empty members of ${^*}\mathcal{P}(\mathbb{R})$ that have an upper bound in ${^*}\mathbb{R}$ also have a least upper bound in ${^*}\mathbb{R}$. Note that the boundedness in ${^*}\mathbb{R}$ is being discussed in terms of the binary relation ${^*}<$ on ${^*}\mathbb{R}$ that is interpreted by the aliens whenever they encounter the standard $<$ symbol written by the humans as a binary relation on $\mathbb{R}$. (As is customary for reasons of brevity, we still use the same symbol `$<$' to denote what really is `${^*}{<}$' if it is clear from context which universe we are interpreting this symbol in.)  

In order to think about nonstandard interpretations of relations when we started with an assumption that all sets have such interpretations, we merely need to be more pedantic and recognize that any particular binary relation on $\mathbb{R}$ is just a subset of the Cartesian product $\mathbb{R} \times \mathbb{R}$, and therefore each standard relation on $\mathbb{R}$ does give rise to its nonstandard interpretation as a subset of ${^*}(\mathbb{R} \times \mathbb{R}) = {^*}\mathbb{R} \times {^*}\mathbb{R}$ (and hence as a relation on ${^*}\mathbb{R}$).

The above fact that ${^*}$ commutes over Cartesian products can also be seen as a simple consequence of transfer! Indeed, if $A$ and $B$ are standard sets and $\pi_A \colon A \times B \to A$ and $\pi_B \colon A \times B \to B$ are the corresponding projection maps, then the following is true in the standard universe: 
\begin{align*}
    \forall x \left\{x \in (A \times B) \leftrightarrow [\left(\pi_A(x) \in A\right) \land \left(\pi_B(x) \in B\right)]\right\},
\end{align*}
which transfers to the following:
\begin{align*}
    \forall x \left\{x \in {^*}(A \times B) \leftrightarrow [\left({^*}(\pi_A)(x) \in {^*}A\right) \land \left({^*}(\pi_B)(x) \in B\right)]\right\}.
\end{align*}
Here is what is really happening in the above sentence. We can pedantically identify any function with its \textit{graph}, similarly to how we identified relations as subsets of an appropriate Cartesian product. For instance, since $\pi_A \colon A \times B \to A$ is really just the subset $\{(x, y): x \in A \times B, y = \pi_A(x)\} \subseteq (A \times B) \times A$, there is a nonstandard interpretation ${^*}\pi_A$ of this set, which, due to transfer, can be seen to be (the graph of) a function from ${^*}(A \times B)$ to ${^*}A$. In fact, we can further use transfer to show that ${^*}(\pi_A) = \pi_{{^*}A}$\footnote{Just transfer the sentence that says that any element $x \in A \times B$ is equal to $(\pi_A(x), \pi_B(x))$.} (and similar statement for $B$), which is what finally allows us to conclude the distributivity of ${^*}$ over Cartesian products that we were trying to prove. 

Let us take a step back and re-evaluate what we proved about the nonstandard interpretation of the usual order on $\mathbb{R}$ when we transferred the sentence \eqref{lub sentence}. We showed that those non-empty members of ${^*}\mathcal{P}(\mathbb{R})$ that happen to be bounded above in ${^*}\mathbb{R}$ must also have a least upper bound in ${^*}\mathbb{R}$. As soon as we have a mathematical reason to believe that ${^*}\mathbb{R}$ contains at least one infinite element $N$ (that is, $N > n$ for all $n \in \mathbb{N}$), we would be able to conclude that the set $\mathbb{N}$ is a counterexample to the least upper bound property in ${^*}\mathbb{R}$ (as $\mathbb{N}$ would then be bounded above by $N$, but evidently subtracting one from any upper bound of $\mathbb{N}$ would still yield an upper bound of $\mathbb{N}$ due to transfer, thus there being no least upper bound of $\mathbb{N}$). 

Therefore, it would follow that $\mathbb{N} \in \mathcal{P}({^*}\mathbb{R}) \backslash {^*}\mathcal{P}(\mathbb{R})$. The elements of ${^*}\mathcal{P}(\mathbb{R})$, are indeed special subsets of ${^*}\mathbb{R}$ over which even second-order properties of $\mathbb{R}$ transfer, and not all subsets of ${^*}\mathbb{R}$ are as nice as long as ${^*}\mathbb{R}$ has infinite numbers.

So far, nothing non-trivial has happened, as we do not really know if ${^*}\mathbb{R}$ contains even one new element, even if transfer principle is true (recall that our proof of Proposition \ref{infinite natural0} relied on an \textit{assumption} that there was a new element), and it could be that the aliens' nonstandard mathematical universe is also the same as our standard universe. Indeed, merely an assumption of transfer principle is not enough to guarantee new elements, much less infinite ones. For instance, using transfer, we were able to prove Proposition \ref{infinite natural0} only under the assumption that that there was an $N \in {^*}\mathbb{N} \backslash \mathbb{N}$, but transfer alone cannot prove that there indeed does exist such an $N$. Fortunately, there are results in model theory (see, for instance, Chang and Keisler \cite[Lemma 5.1.4, p. 294 and Exercise
5.1.21, p. 305]{Model_Theory}) that guarantee that there are nonstandard universes that not only satisfy the transfer principle, but are also \textit{saturated}, which is a compactness-type property that is defined as follows:

\begin{definition}
A set in a nonstandard universe is called \textit{internal} if it belongs to ${^*}\mathcal{P}(S)$ for some standard set $S$ (, in which case, it is an internal subset of ${^*}S$, by a reasoning analogous to \eqref{subset sentence 2}). A nonstandard universe is called saturated \footnote{In the literature, what we are defining is sometimes called \textit{polysaturation}.} if for any index set $I$ in the standard universe, and any collection $\{A_i : i \in I\}$ of internal sets in the nonstandard universe that have the finite intersection property (that is, for any finitely many $i_1, \ldots, i_n \in I$, we have $A_{i_1} \cap \ldots \cap A_{i_n} \neq \emptyset$), we have $\cap_{i \in I}A_i \neq \emptyset$.
\end{definition}

Note that when we transferred second-order sentences earlier, their truths were only preserved when working with internal sets. For instance, all non-empty internal subsets of ${^*}\mathbb{R}$ that have an upper bound have a least upper bound. Also, in particular, the set ${^*}\mathbb{R}$ itself is internal (a fact revealed by transferring the sentence `$\mathbb{R} \in \mathcal{P}(\mathbb{R})$'). Similarly, ${^*}A$ is an internal subset of ${^*}\mathbb{R}$ for all $A \in \mathcal{P}(\mathbb{R})$. There are internal subsets of ${^*}\mathbb{R}$ that are not of this type, but in order to understand how they might look like, we must first finally prove that any saturated nonstandard universe does contain lots of new elements. 

\begin{proposition}\label{existence of infinites}
In a saturated nonstandard universe, ${^*}\mathbb{R}$ contains infinite numbers (that is, numbers that are bigger than all real numbers, which are viewed as elements of ${^*}\mathbb{R}$), as well as non-zero infinitesimal elements (which are elements whose absolute values are smaller than all positive real numbers).
\end{proposition} 
\begin{proof}
Any element in the non-empty intersection $\cap_{n \in \mathbb{N}} \{x \in {^*}\mathbb{R}: x>n\}$ (which is non-empty by saturation) must be infinite. The multiplicative inverse of any infinite element is infinitesimal.
\end{proof}

Technical details aside, we want to emphasize the intuition that if there is a \textit{standard collection} (that is, a collection indexable by a set in the standard universe) of statements/conditions that a typical element may or may not satisfy, then just verifying that any finite number of those conditions are satisfiable (by possibly different elements depending on the finite subset of conditions) is enough to justify the existence of a single object that satisfies \textit{all} of the conditions simultaneously. This is very powerful and is indeed the philosophical reason why nonstandard analysis is successful in applications, as it allows one to go from the finite to the infinite (or vice versa) in a logically precise way. 

In the above proof, for each $n \in \mathbb{N}$, the set $\{x \in {^*}\mathbb{R} : x > n\}$ was internal because we can write a formal logic sentence that says that for each $n \in \mathbb{N}$, there is a unique set consisting of those real numbers $x$ that are bigger than $n$. By transfer, for each $n \in {^*}\mathbb{N}$, there is a unique \textit{internal} subset of ${^*}\mathbb{R}$ whose membership criterion can be expressed as `$x > n$'. For any finitely many $n_1, \ldots, n_k \in \mathbb{N}$, the number $n_1 + \ldots + n_k$ is in the finite intersection  $A_{n_1} \cap \ldots \cap A_{n_k}$, which is why we could apply saturation.

The details in the preceding paragraph are completely pedantic. The rule of thumb to keep in mind is the following:

\boldbox{\text{Rule of thumb (Internal Definition Principle)}}: \\If there is a formal logic formula (which is just a grammatically correct finite string of symbols) $\phi(x, v_1, \ldots, v_k)$ in which $x, v_1, \ldots, v_k$ are variables and the rest of the symbols appearing in this formula are standard mathematical symbols, then for any fixed members $a_1, \ldots, a_k$ of some internal set, the set of those $x$ (in the nonstandard universe) for which $\phi(x, a_1, \ldots, a_k)$ is true is an internal set. We can think of $\phi$ as the property that defines this internal set with \textit{parameters} $a_1, \ldots, a_k$. All internal sets arise this way.

The internal definition principle is rigorously proved using induction on how long the formula $\phi$ is, but we shall only be content with understanding the intuition of how it allows us to build more internal sets out of the internal sets that are easier to see. For instance, in the proof of Proposition \ref{existence of infinites}, we used the formula $\phi(x, n) \defeq \text{`}x > n\text{'}$ in order to define an internal set for each $n \in {^*}\mathbb{N}$ (though we only used those internal sets when $n$ took values in $\mathbb{N}$, which is contained in ${^*}\mathbb{N}$). Note that our earlier (alien) intuition of the internal set as being a set that the aliens can ``perceive'' given that they perceive the universe at scales infinitely finer than us still works. Just like we can define (``perceive'') the set of all numbers greater than a fixed number we have access to, so can they in their richer mathematical universe. We will see a more complicated application of the idea of using known internal sets to define other internal sets again when we study overflow and underflow in Proposition \ref{Over and under}.

When one works using nonstandard analysis, as long as we have saturation, the actual nonstandard universe that one has usually does not matter, as any individual formal properties about the standard universe still transfer over to the nonstandard universe (while saturation guarantees that we do get new elements, thus allowing us to work with a richer mathematical structure to prove things in, and hopefully transfer back our conclusions to the standard universe). 

In any application of (Robinson's) nonstandard analysis to standard mathematics (and hence in this paper), we fix a saturated nonstandard universe that satisfies the transfer principle, which is what we shall do going forward. This can be made more precise if desired (with the superstructure framework of the next section a step in that direction), but this level of understanding is sufficient to start working with nonstandard analysis. 

For the sake of completion, we note the following consequence of saturation which shows that the hypothesis in Proposition \ref{infinite natural0} is always true for saturated nonstandard universes.  

\begin{proposition}\label{infinite natural}
The set ${^*}\mathbb{N} \backslash \mathbb{N}$ is nonempty.
\end{proposition}

\begin{proof}
    The collection $\{{^*}\mathbb{N} \backslash \{n\}: n \in \mathbb{N}\}$ is a countable collection of internal sets satisfying the finite intersection property. Hence by saturation, there exists $N \in {^*}\mathbb{N} \backslash \mathbb{N}$. 
\end{proof}

As discussed earlier, this implies that $\mathbb{N}$ is not an internal set in the nonstandard universe (as it is bounded by any infinite number, but does not have a least upper bound). While we assumed that $\mathbb{N} \subseteq {^*}\mathbb{N}$ by making the identification that ${^*}n = n$ for each $n \in \mathbb{N}$, we cannot do the same identification when we talk at the level of (nonstandard extensions of) sets in the standard universe. Indeed, for any $A \in \mathcal{P}(\mathbb{N})$, we can only use transfer to conclude that ${^*}A \in {^*}\mathcal{P}(\mathbb{N})$, where ${^*}A$ contains $A$ but may in general be much bigger than $A$ (it is the same as $A$ if and only if $A$ is finite --- another fact that can be proved by transfer and saturation). Thus, we do not have $\mathcal{P}(\mathbb{N}) \subseteq {^*}\mathcal{P}(\mathbb{N})$ in the literal sense, but we do have $\{{^*}A: A \in \mathcal{P}(\mathbb{N})\} \subseteq {^*}\mathcal{P}(\mathbb{N})$. The reason this was not an issue in saying that $\mathbb{N} \subseteq {^*}\mathbb{N}$ or $\mathbb{R} \subseteq {^*}\mathbb{R}$ was because we hid some technical details under the rug when we identified ${^*}c$ with $c$ for each $c \in \mathbb{R}$, something we shall touch upon a bit when we discuss the superstructure framework in the next section. 

In particular, we have been working with a standard universe and a saturated nonstandard universe at an intuitive level, but we have yet not made mathematically precise what the standard universe is. We postpone that until the next section when we talk about the superstructure framework.

Note that if $N > \mathbb{N}$ (which is our shorthand for $N \in {^*}\mathbb{N}\backslash\mathbb{N}$ in view of Proposition \ref{infinite natural0}), then
$[N] = \{n \in {^*}\mathbb{N}: n \leq N\}$ is an example of an internal subset of ${^*}\mathbb{N}$ that is not of the type ${^*}A$ for any $A \in \mathcal{P}(\mathbb{N})$.     

If $\phi$ is a property that defines an internal set (possibly with some fixed parameters $a_1, \ldots, a_k$), then that internal set cannot contain all natural numbers without containing a hyperfinite number $N > \mathbb{N}$ (as the set $\mathbb{N}$ is not internal). More roughly, one way to think of this is to recognize the intuition that ``if something is true for arbitrarily large $n \in \mathbb{N}$, then it must be true for a hyperfinite $N > \mathbb{N}$ (and vice versa).'' This or other variations of this technique is called \textit{overflow} (or \textit{underflow} for the reverse direction), a technique that is quite useful in applications.

\begin{proposition}\label{Over and under} Let $A$ be an internal set.
\begin{enumerate}[(i)]
    \item\label{overflow}[\textbf{Overflow}] If $\mathbb{N} \subseteq A$, then there is an $N > \mathbb{N}$ such that $$[N] \subseteq A.$$ 
    \item\label{underflow}[\textbf{Underflow}] If $A$ contains all hyperfinite natural numbers, then there is an $n_0 \in \mathbb{N}$ such that ${^*}\mathbb{N}_{\geq n_0} \defeq \{n \in {^*}\mathbb{N}: n \geq n_0\} \subseteq A$.
\end{enumerate}
\end{proposition}  
\begin{proof}
  We only prove overflow, as the proof of underflow can be thought of as the dual of this proof, and is hence left as an exercise for the interested reader. Assume $\mathbb{N} \subseteq A$. Since $A$ is internal, and since $[n]$ is internal for each $n \in {^*}\mathbb{N}$ the set $\{n \in {^*}\mathbb{N}: \forall x \in [n] (x \in A)\}$ is internal by the internal definition principle. We now obtain an element $N > \mathbb{N}$ in this set due to the fact that $\mathbb{N}$ is not internal (and is yet a subset of this internal set).
\end{proof}

    \subsection{Yet more mathematical details: The superstructure framework}
    The goal of this last section of our introduction is to provide an intuitive sketch of the key features of the superstructure framework of nonstandard analysis, which is the framework that we used in the mathematical applications in this manuscript, as it allows us to talk about nonstandard versions of structures other than just the number line (so, we can talk about nonstandard extensions of topological spaces, etc.).  

In general, we fix a ``ground'' set $\mathbb{S}$ consisting of atoms (that is, we view each element of $\mathbb{S}$ as an ``individual'' without any structure, set-theoretic or otherwise), and extend what is called the superstructure $V(\mathbb{S})$ of $\mathbb{S}$, which is defined inductively as follows:

\begin{equation}
    \label{superstructure} 
    \begin{array}{rcl} 
     V_0({\mathbb{S}}) &\defeq &  \mathbb{S}, \\
     V_{n}(\mathbb{S}) &\defeq & \mathcal{P}(V_{n-1}(\mathbb{S})) \text{ for all } n \in \mathbb{N}, \\
     V(\mathbb{S}) &\defeq &  \bigcup_{n \in \mathbb{N} \cup \{0\}} V_n(\mathbb{S}).
     \end{array}
\end{equation}
    
It is the superstructure $V(\mathbb{S})$ that we may think of as what we have so far been informally calling the standard universe $\mathfrak{S}$. For some set $\mathbb{S}'$, We extend the standard universe via a \textit{nonstandard map},
$${^*}\co V(\mathbb{S}) \rightarrow V(\mathbb{S}'),$$
which, by definition, is any map satisfying the following axioms that were described in the previous section(s):
    \begin{enumerate}
        \item[(i)] ${^*}\mathbb{S} = \mathbb{S}'$,
        \item[(ii)] Transfer principle holds.
        \item[(iii)] Saturation holds.
    \end{enumerate}

For any $A \in V(\mathbb{S})$, we call ${^*}A$ the nonstandard interpretation or extension of $A$. Thus, $V({^*}\mathbb{S})$ is our nonstandard universe. We assume that the ground set $\mathbb{S}$ contains (at least) all real numbers as elements, so as to have a rich enough universe to extend. In fact, choosing $\mathbb{S}$ suitably, the superstructure $V(\mathbb{S})$ can be made to contain all mathematical objects relevant for a particular study (say, all standard objects that appeared in the main body of this paper). Indeed, since $\mathbb{R} \subseteq \mathbb{S}$, all collections of subsets of $\mathbb{R}$ live as objects in $V_2(\mathbb{S}) \subseteq V(\mathbb{S})$. For a finite subset consisting of $k$ objects from $V_m(\mathbb{S})$, the ordered $k$-tuple of those objects is an element of $V_n(\mathbb{S})$ for some larger $n$; and hence the set of all $k$-tuples of objects in $V_m(\mathbb{S})$ lies as an object in $V_{n+1}(\mathbb{S})$. For example, if $x, y \in V_m(\mathbb{S})$, then the ordered pair $(x, y)$ is just the set $\{\{x\}, \{x,y\}\} \in V_{m+2}(\mathbb{S})$. Identifying functions and relations with their graphs, $V{(\mathbb{S})}$ also contains, for each $n \in \mathbb{N}$, all functions from $\mathbb{R}^n$ to $\mathbb{R}$, all relations on $\mathbb{R}^n$, etc. If we wish to study a topological space $T$, we can just include its elements as part of $\mathbb{S}$, and then $T \in \mathcal{P}(S)= V_1(S)$ will have a nonstandard interpretation ${^*}T$. 

Since an element $a$ of $\mathbb{S}$ is not assumed to have any additional structure, we can safely identify ${^*}a$ with $a$. \footnote{What we are describing is still a somewhat intuitive picture that captures how we think. Strictly speaking, if we are working in the language of set theory (ZFC) as we have been, then the objects in our (super)structure should also be sets, so how can we think that the elements of the ground set have no set-theoretic structure without going out of the framework of ZFC? Well, this is only a pedantic issue that can be easily worked around by replacing any desired ground set by one that mimics the properties we want in a way that we might as well assume that the new ground set consists of atoms without loss of generality. The property we really require from $\mathbb{S}$ is for it to be a so-called \textit{base set}, which is defined as a set for which $\emptyset \not\in \mathbb{S}$, and such that for all $x \in \mathbb{S}$, we have $x \cap V(\mathbb{S}) \neq \emptyset$. Chang and Keisler \cite[Section 4.4]{Model_Theory} outline how to replace any given set with a base set of the same cardinality, which is why we are intuitively allowed to assume without loss of generality that the elements of the ground set $\mathbb{S}$ have no further set-theoretic structure.} On the other hand, once we start looking at a set $A$ from the standard universe, we can no longer identify $A$ with ${^*}A$ materially (though we can indeed say that ${^*}A$ satisfies the same formal properties, while containing $A$)---the additional set-theoretic structure of $A$ (along with saturation) forces ${^*}A$ to be a bigger set whenever $A$ is infinite.

\begin{remark}
    As discussed earlier, in the superstructure framework any standard function, relation, etc. can be identified with a set at an appropriate level of the iterated power sets of the ground set $\mathbb{S}$, and hence we can safely talk about nonstandard extensions of functions, relations, etc. In general, just like internal sets are those sets that are also elements of the nonstandard extension of a standard power set, we can talk about internal functions, internal relations, etc. Indeed, since standard sets have nonstandard extensions, we can see (using transfer) that the nonstandard extension of the graph of a standard function (which is what we are identifying a function with as an element of the superstructure) is again the graph of some function in the nonstandard universe, which we call the \textit{nonstandard extension} of the original function (as it matches with the original function when restricted to that domain).
\end{remark}

  In practice, we are almost never as pedantic as in the previous remark. Using the internal definition principle, the informal rule of thumb to keep in mind is that all standard functions have nonstandard extensions, and that in general, there can be other types of internal functions, which can be identified if we are able to define them in terms of an appropriate ``defining condition'' or ``formula''. 
    
    In fact, without emphasizing it earlier, we had already been working with internal functions when we saw that ${^*}\mathbb{R}$ is an ordered field containing the ordered subfield $\mathbb{R}$ (as, for example, the function $+ \colon \mathbb{R} \times \mathbb{R} \to \mathbb{R}$ gets extended to a function ${^*}+ \colon {^*}(\mathbb{R} \times \mathbb{R}) \to {^*}\mathbb{R}$ such that the restriction $({^*}+) \vert_{\mathbb{R} \times \mathbb{R}} = +$. Note that for standard sets $A$ and $B$, the set ${^*}(A \times B)$ can be seen to be the same as ${^*}A \times {^*}B$ by transfer (we sketched this on p. 85), so that $+$ (where we drop the asterisk for notational brevity) really is a function on ${^*}\mathbb{R} \times {^*}\mathbb{R}$. 

    An important class of examples of internal functions that are not the nonstandard extensions of any standard functions are obtained by considering the so-called hyperfinite sets. For each $N \in {^*}\mathbb{N}$, the initial segment $[N]$ is internal by the internal definition principle. If $N > \mathbb{N}$, then the initial segment $[N]$ is technically an infinite set (for instance, it contains $\mathbb{N}$), but it behaves like a finite set and there is a sense in which we can think it has $N$ elements. Let us make this precise in a more general setting. 

 For any standard set $A$, let $\mathcal{P}_{\text{fin}}(A)$ denote the set of all finite subsets of $A$. Then there is a map $\# \colon \mathcal{P}_{\text{fin}}(A) \to \mathbb{N} \cup \{0\}$ which we call the cardinality function. The nonstandard extension ${^*}\# \colon {^*}\mathcal{P}_{\text{fin}}(A) \to {^*}\mathbb{N} \cup \{0\}$ will be called the \textit{internal cardinality function}. We typically suppress the asterisk and use the same notation $\#$ in both universes. The members of ${^*}\mathcal{P}_{\text{fin}}(A)$ are called the \text{hyperfinite subsets} of ${^*}A$. (Note that a hyperfinite subset can also be finite.)

 By transfer of the corresponding statement for each $n \in \mathbb{N}$, we can thus conclude that $\#([N]) = N$. By a slightly more involved, but still routine, transfer argument, we can also show that an internal set $H$ is hyperfinite if and only if there is an $N \in {^*}\mathbb{N}$ and an internal bijection $f\co H \rightarrow \{1, \ldots, N\}$. 

There is a ``sum function'' that takes any finite tuple of real numbers as an input and produces the sum of those real numbers. By transfer, we can thus abstractly make sense of ``hyperfinite sums'' (that is, the sum of hyperfinitely many nonstandard real numbers). For nonstandard real numbers $a_i$, this is the sense in which we interpret objects such as $\sum_{i = 1}^N a_i$ where $N \in {^*}\mathbb{N}$ (or in general, $\sum_{i \in H} a_i$, where $H$ is a hyperfinite set). 

Armed with the knowledge of hyperfinite sums that formally behave just like finite sums, we can go back to the inequalities in \eqref{Sarah inequality} and recognize that (, in the alien intuition whereby we are communicating our standard mathematical facts,) there would also be a sentence that would include all of those inequalities in a single statement as follows:
\begin{align}
    \forall n \in \mathbb{N} ~\left[\left(3 + \sum_{i = 1}^n \frac{9}{10^i}\right) < 4\right].
\end{align}

In the nonstandard or alien interpretation, this remains true for all $n \in {^*}\mathbb{N}$. If $N > \mathbb{N}$, then one can prove (using transfer) that $\textstyle{\left[4 - (3 + \sum_{i = 1}^N \frac{9}{10^i})\right]}$ equals $\textstyle{\left[1 - \frac{9}{10}\cdot\frac{1 - \frac{1}{10^N}}{\frac{9}{10}} \right] = \frac{1}{10^N}}$. Had nonstandard thinking been more standardized in our mathematics education, this is how Ely's student Sarah could have rigorized her intuition of $4$ and ``$3.999$ repeating forever'' being ``infinitely close'' \cite[p. 128]{ely2010nonstandard}. \footnote{One might incorrectly think that Sarah would have never been able to rigorize her intuition this way because she would have never reached a level of mathematics education where she can see nonstandard analysis as a ``research topic'' in order to learn these concepts we just described. While one would be, unfortunately, almost certainly correct in this asssessment of Sarah's practical mathematical career in the current system of mathematics education, it would still be an incorrect assessment because it is possible to teach mathematics with infinitesimals rigorously to students at Sarah's (students taking their first Calculus course) level, as Ely \cite{ely2021teaching} and several other educators cited there have recognized.}



    

The superstructure framework allows us to use nonstandard thinking in settings much more general than just real analysis---we shall illustrate this next when we work with topology and probability theory (and the interplay between them) using nonstandard methods. But we shall first need to provide some background in the basic notation and conventions from both subjects that we will be using, and hence we conclude our introduction to basic nonstandard methods here, resuming with more advanced nonstandard methods in the next appendix.

    \section{On topological measure theory, guided by nonstandard perspectives}\label{AppA}
Given a nonempty set $\Omega$, an \textit{algebra} $\mathcal{A}$ on it is a nonempty collection of subsets of $\Omega$ that is closed under finite unions and complements (and hence under finite intersections by de Morgan's laws). A \textit{sigma algebra} $\mathcal{F}$ on it is an algebra that is also closed under countable unions, in which case we call the pair $(\Omega, \mathcal{F})$ a \textit{measurable space}. A \textit{finitely additive probability space} is a triple $(\Omega, \mathcal{A}, \mathbb{P})$ where $\Omega$ is a set, $\mathcal{A}$ is an algebra on it, and $\mathbb{P} \colon \mathcal{A} \to [0,1]$ satisfies the following conditions:
\begin{enumerate}
    \item[(P1)] $\mathbb{P}(\Omega) = 1$.
    \item[(P2)] $\forall A_1, A_2 \in \mathcal{A} ~\left( (A_1 \cap A_2 = \emptyset) \rightarrow \mathbb{P}(A_1 \cup A_2) = \mathbb{P}(A_1) + \mathbb{P}(A_2)\right)$.
\end{enumerate}
A \textit{probability space} is a triple $(\Omega, \mathcal{F}, \mathbb{P})$ where $(\Omega, \mathcal{F})$ is a measurable space and $\mathbb{P} \colon \mathcal{F} \to [0,1]$ satisfies both axiom (P1) and the following stronger form of (P2) that makes it a \textit{countable additive measure} as opposed to only a finitely additive one:
   \begin{enumerate}
       \item[(P2')] For each $\{A_n : n \in \mathbb{N}\} \subseteq \mathcal{A}$, we have  \begin{align*}\left( \left[\forall i, j \in \mathbb{N} (i \neq j) \rightarrow (A_i \cap A_j = \emptyset)\right] \rightarrow \mathbb{P}\left(\bigcup_{n \in \mathbb{N}}A_n\right) = \sum_{n \in \mathbb{N}}\mathbb{P}(A_n)\right).
       \end{align*}
   \end{enumerate}
By a \textit{probability measure} without any further qualifiers, unless otherwise specified, we shall mean a countably additive probability function on a measurable space, although we will often sometimes work with \textit{finitely additive probability measures} defined on an algebra of sets. 

Given a topological space $(T, \tau)$ (that is, $T$ is a non-empty set, while $\tau$ is a non-empty collection of subsets of $T$ that is closed under finite intersections and arbitrary unions --- the members of $\tau$ being called \textit{open sets}, while their complements being called \textit{closed sets}), we use $\mathcal{B}(T)$ to denote the \textit{Borel sigma algebra of $T$}, which is defined as the smallest sigma algebra on $T$ that contains all open subsets of $T$. By a Borel probability measure on $T$, we shall mean a countably additive probability measure on $(T, \mathcal{B}(T))$. We are assuming readers to have familiarity with basic point-set topological notions such as Hausdorffness, compactness, etc, and also with basic measure theoretic constructions such as integrals (or \textit{expected values} in the probabilistic settings).

We now recall the definitions of some important classes of probability measures.  

\begin{definition}\label{tightness definition}
For a Hausdorff space $T$, a (finitely or countably additive) Borel probability measure $\mu$ is called \textit{tight} if given any $\epsilon \in \mathbb{R}_{>0}$, there is a compact subset $K_{\epsilon}$ such that the following holds:
\begin{align}\label{tightness 1}
    \mu(K_{\epsilon}) > 1 - \epsilon.
\end{align}
\end{definition}

An alternative way to write the above condition for tightness is the following:
\begin{align}\label{tightness 2}
    \mu(T) = \sup\{\mu(K): K \text{ is a compact subset of } T\}.
\end{align}

If a measure $\mu$ satisfies \eqref{tightness 2} with the occurrence of $T$ replaced by any Borel subset of $T$, then we call it a Radon measure. More formally we make the following definition (the second line following from the first since we are only considering probability, and in particular finite, measures).

\begin{definition}\label{definition of Radon probability measures}
For a Hausdorff space $T$, a (finitely or countably additive) Borel probability measure $\mu$ is called \textit{Radon} if for each Borel set $B \in \mathcal{B}(T)$, the following holds:
\begin{align*}
    \mu(B) &= \sup \{\mu(K): K \subseteq B \text{ and } K \text{ is compact}\} \\
    &= \inf \{\mu(G): B \subseteq G \text{ and } G \text{ is open}\}.
\end{align*}
\end{definition}

\begin{remark}\label{Schwartz remark}
    Note that the Hausdorffness of the topological space $T$ was assumed in the previous definitions so as to ensure that the compact sets appearing in them were Borel measurable (as a compact subset of any Hausdorff space is automatically closed). While not typically done (as many results do not generalize to those settings), these definitions can be made for arbitrary topological spaces if we replace the word ``compact'' by ``closed and compact''. See Schwartz \cite[pp. 82-88]{Schwartz} for more details on this generalization. (Schwartz uses the phrase `quasi-compact' instead of `compact' in this discussion.) In this manuscript, we will always have an underlying assumption of Hausdorffness of $T$ during any discussions involving tight or Radon measures.
\end{remark}

\begin{remark}\label{tightness and Radon}
It is clear that all Radon measures are tight. Note that any countably additive Borel probability measure on a $\sigma$-compact Hausdorff space (that is, a Hausdorff space that can be written as a countable union of compact spaces) is tight. Vakhania--Tarladze--Chobanyan \cite[Proposition 3.5, p. 32]{Vakhania-Tariladze-Chobanyan} constructs a non-Radon Borel probability measure on a particular compact Hausdorff space (the construction being attributed to Dieudonn\'e). Thus, not all tight measures are Radon. 
\end{remark}

With these basic definitions from measure theory and topology out of the way, we now resume our introduction to nonstandard methods started in the previous appendix by providing reviews (in the next two sections respectively) of the basic ideas from Loeb measure theory and nonstandard topology that we shall need. 

\subsection{Loeb probability theory}
Naively, internal sets have the same formal properties as standard sets, and therefore internal functions (which are just functions whose graphs are internal) have the same formal properties as standard functions. In general, this type of naive thinking can guide us when encountering internal versions of more complicated structures. 

In the standard universe, a probability space is a triple $(\Omega, \mathcal{F}, \mbp)$ where $\Omega$ is a non-empty set, $\mathcal{F}$ is a sigma algebra on $\Omega$, and $\mbp \colon \mathcal{F} \to [0,1]$ is a  function satisfying the well-known probability axioms. 

Thus, formally, an internal probability space is a triple $(\mathfrak{T}, \mathcal{F}, \nu)$, where $\mathfrak{T}$ is a non-empty internal set, $\mathcal{F}$ is an internal sigma algebra on $\mathfrak{T}$, and $\nu \colon \mathcal{F} \to {^*}[0,1]$ is an internal function satisfying the nonstandard interpretations of the standard axioms for probability measures. 

There are several complications with the above description. Firstly, one of the axioms for sigma algebras requires one to check closure under countable unions. By transfer, verifying that a collection is an internal sigma algebra requires us to verify closure under \textit{hypercountable unions} --- more formally, if there is an internal function $f \colon {^*}\mathbb{N} \to \mathcal{F}$ (which is the nonstandard analog of taking a countable sequence of sets in the standard universe), then we require $\cup_{n \in {^*}\mathbb{N}} f(n) \in \mathcal{F}$. Secondly, the countable additivity axiom for probabilities now gets interpreted as a statement saying that the measure of a hypercountable disjoint union is the appropriate hypercountable sum, where the latter can be defined by interpreting the definition of countable sums in the nonstandard universe. All of this can be too cumbersome to work with, but thanks to Peter Loeb, we can do usual countably additive probability theory in the nonstandard universe without worrying about the complicated axioms that an internal probability space must satisfy. 

Loeb \cite{Loeb-1975} found a way to convert an internal finitely additive probability on an internal algebra \footnote{In analogy with our approach of understanding internal objects as having the same formal properties as the corresponding standard objects, we can think of an internal algebra as an internal set $\mathcal{A}$ consisting of internal subsets of some nonstandard sample space $\mathfrak{T}$ such that $\mathcal{A}$ is an algebra of sets in the standard sense. An internal finitely additive probability $\nu$ is then an internal function $\nu \colon \mathcal{A} \to {^*}[0,1]$ such that: \begin{enumerate}[(i)]
    \item $\nu(A \cup B) = \nu(A) + \nu(B) \text{ if } A \cap B = \emptyset.$
    \item $\nu(\Omega) = 1$.
\end{enumerate}
Note that, by transfer, an internal finitely additive measure $\nu$ is actually \textit{hyperfinitely} additive in the sense that it maps the internal union of a hyperfinite collection of mutually disjoint internal measurable sets to the hyperfinite sum of the individual measures of those sets.} into a legitimate countably additive probability measure on a sigma algebra containing the internal algebra. Loeb's method relies on the following consequence of saturation:

\begin{proposition}\label{countable union}
A countable union of disjoint internal sets is internal if and only if all but finitely many of them are empty.
\end{proposition}
\begin{proof}
Suppose $\{A_i\}_{i \in \mathbb{N}}$ is a countable collection of disjoint internal sets. Let $A = \cup_{i \in \mathbb{N}}A_i$. If all but finitely many of the $A_i$ are empty, then $A$ being a finite union of internal sets is also internal due to transfer.

Conversely, if $A$ is internal, then $A \backslash A_i$ is internal for each $i \in \mathbb{N}$ by transfer. In that case, if all but finitely many of the $A_i$ are not empty, then the collection $\{A \backslash A_i\}_{i \in \mathbb{N}}$ would satisfy the finite intersection property. By saturation, this would lead to $\cap_{i \in \mathbb{N}} (A \backslash A_i) \neq \emptyset$, which is absurd. This completes the proof by contradiction.
\end{proof}

More precisely, consider an internal finitely additive probability space $(\mathfrak{T}, \mathcal{A}, \nu)$. Then, using Remark \ref{st part morphism}, the function $\st(\nu) \colon \mathcal{A} \to [0,1]$ is a finitely additive measure on an algebra. Due to Proposition \ref{countable union}, the hypothesis of Carath\'eodory's extension theorem is trivially satisfied for the finitely additive measure $\st(\nu)$ on the algebra $\mathcal{A}$. Therefore, by that theorem, there exists a unique countably additive probability measure $L\nu$ (called the \textit{Loeb measure induced by $\nu$}) on a sigma algebra $L(\mathcal{A})$ containing $\mathcal{A}$ such that $(\mathfrak{T}, L(\mathcal{A}), L\nu)$ is a complete probability space---that is, a probability space in which all subsets of measure zero sets are measurable.\footnote{Loeb's method works for any internal finitely additive finite measure, although we only need it in the case of probability measures.}

We next illustrate an alternative, more explicit, construction of Loeb measures, that is often more useful in practice. If $(\mathfrak{T}, \mathcal{A}, \nu)$ is an internal finitely additive probability space, then the corresponding inner and outer measures on all subsets of $\mathfrak{T}$ are defined as follows:
\begin{align}
    \underline{\nu}(A) &\defeq \sup\{\st(\nu(B)): B \in \mathcal{A} \text{ and } B \subseteq A\}, \text{ and } \nonumber \\
     \overline{\nu}(A) &\defeq \inf\{\st(\nu(B)): B \in \mathcal{A} \text{ and } A \subseteq B\}. \label{inner and outer}
\end{align}

The collection of sets for which the inner and outer measures agree form a sigma algebra called the $\textit{Loeb sigma algebra}$ $L(\mathcal{A})$. The common value $ \underline{\nu}(A) =  \overline{\nu}(A)$ in that case is defined as the \textit{Loeb measure} of $A$, written $L\nu(A)$. We call $(\mathfrak{T}, L(\mathcal{A}), L\nu)$ the \textit{Loeb space} of $(\mathfrak{T}, \mathcal{A}, \nu)$. More formally, we have:
\begin{align}
    L(\mathcal{A}) \defeq \{A \subseteq \mathfrak{T}:  \underline{\nu}(A) =  \overline{\nu}(A)\}, \label{loeb measurable}
\end{align}
 and
\begin{align}
    L\nu(A) \defeq \underline{\nu}(A) =  \overline{\nu}(A) \text{ for all } A \in L(\mathcal{A}). \label{loeb measure}
\end{align}

When the internal finitely additive measure $\nu$ is clear from context, we will frequently write `\text{Loeb measurable}' (in the contexts of both sets and functions) to mean measurable with respect to the corresponding Loeb space $(\mathfrak{T}, L(\mathcal{A}), L\nu)$. Note that the Loeb sigma algebra $L(\mathcal{A})$, as defined above, does on the choice of $\nu$---we will use appropriate notation such as $L_\nu(\mathcal{A})$ to indicate this dependence if there is any chance of confusion regarding the original measure inducing the Loeb sigma algebra. If we use the notation $L(\mathcal{A})$, then it is understood that a specific internal finitely additive measure $\nu$ has been fixed on $(\mathfrak{T}, \mathcal{A})$ during that discussion. 


For each finitely additive probability space $(\Omega, \mathcal{A}, \mathbb{P})$ in the standard universe, there is a corresponding $\mathbb{R}$-vector space $L^1(\mathbb{P})$ of integrable functions from $\Omega$ to $\mathbb{R}$, and a corresponding linear map $\int \colon L^1(\mathbb{P}) \to \mathbb{R}$. We often write $\int_{\Omega} f d\mathbb{P}$ instead of just $\int f$, in order to emphasize the dependence on the underlying space. We shall also use $\mathbb{E}(f)$ or $\mathbb{E}_{\mathbb{P}}(f)$ to denote the same.

Interpreting this formally in the nonstandard universe, for each internal finitely additive probability space $(\mathfrak{T}, \mathcal{A}, \nu)$, there is a corresponding internal ${^*}\mathbb{R}$-vector space ${^*}L^1(\nu)$ of \textit{internally integrable} functions, and a corresponding linear map $\starint \colon {^*}L^1(\nu) \to {^*}\mathbb{R}$. Analogously, to the standard situation, we often write $\starint_{\mathfrak{T}} f d\mathbb{P}$ instead of just $\starint f$, in order to emphasize the dependence on the underlying internal probability space. We shall also use ${^*}\mathbb{E}(f)$ or ${^*}\mathbb{E}_{\nu}(f)$ to denote the same.

Let us fix an internal finitely additive probability space $(\mathfrak{T}, \mathcal{A}, \nu)$. If we work with an internally integrable function $f \colon \mathfrak{T} \to {^*}\mathbb{R}$ such that ${^*}\mathbb{E}_{\nu}(f) \in {^*}\mathbb{R}_{\text{fin}}$, then $L\nu(f \in {^*}\mathbb{R}_{\text{fin}})$ is the same as $\lim_{n \in \mathbb{N}}(1 - L\nu(\abs{f} > n))$, which can be shown to be equal to one by transfer of Chebyshev's inequality. In other words, $f$ being internally integrable implies that it takes only finite values Loeb almost surely, in which case $\st(f)$ is well-defined Loeb almost surely. In that situation, it is interesting to see when we have the following equality (which, if true, allows us to use ordinary probability theory methods on nonstandard spaces!):
\begin{align}\label{S-integrable}
    \st\left(\starint_{\mathfrak{T}} f d\nu \right) \overset{?}{=}  \int_{\mathfrak{T}} \st(f) dL\nu.
\end{align}

The internally integrable functions for which the above equality holds are known as \textit{S-integrable} functions. In applications, $S$-integrability is usually verified by checking one of the other characterizations of \eqref{S-integrable} (see Ross \cite[Theorem 6.2, p.110]{Ross_NATO}), one useful characterization being the following condition:
\begin{align}\label{tail S-integrability}
    \st\left(\starint_\mathfrak{T} \abs{f} \mathbbm{1}_{\{\abs{f} > M\}}d \nu \right) = 0 \text{ for all } M > \mathbb{N}.
\end{align}

For the purposes of the applications in this paper, it suffices to observe (using \eqref{tail S-integrability}) that any internal function that is \textit{standardly bounded} (that is, there exist real numbers $a$ and $b$ such that the range of the function is contained in ${^*}[a,b]$) is $S$-integrable. 

In the context of the present paper, we have often worked in the situation in which $\mathfrak{T}$ is the nonstandard extension of a topological space $T$. That is, $\mathfrak{T} = {^*}T$, and $\mathcal{A}$ is the algebra ${^*}\mathcal{B}(T)$ of internally Borel subsets of ${^*}T$. Before we review the relevant topological measure theoretic background in that setting, we conclude our tutorial on basic nonstandard analysis by illustrating nonstandard ways to think about topological spaces next. 

\subsection{A glimpse of nonstandard topological thinking} If $x \in {^*}\mathbb{R}$ and $y \in \mathbb{R}$ are such such that $x \approx y$, then one way to think of this situation is that in the nonstandard universe $x$ cannot be separated from $y$ by (the nonstandard interpretation of) any open neighborhood of $y$. That is, if $\tau_y$ denotes the set of open neighborhoods of $y$, then:
\begin{align}\label{topological closeness}
    x \approx y \iff x \in {^*}U \text{ for all } U \in \tau_y. 
\end{align}
Indeed, if $x \approx y$, then $\abs{x - y} < \epsilon$ for all $\epsilon \in \mathbb{R}_{>0}$, so that whenever $U$ is an open neighborhood of $y$, we can first find an $\epsilon \in \mathbb{R}_{>0}$ small enough for which $(y - \epsilon, y + \epsilon) \subseteq U$, so that by transfer ${^*}(y - \epsilon, y + \epsilon) \subseteq {^*}U$, where the former set equals $\{z \in {^*}\mathbb{R}: y - \epsilon < z < y + \epsilon\}$ (by transfer) and hence contains $x$ by assumption. The converse direction is easier to see by letting $U$ vary over intervals around $y$ with arbitrarily small real radii.

Even though the standard part of a (finite) nonstandard real number $x$ was originally defined as the unique real number at an infinitesimal distance from $x$, the characterization \eqref{topological closeness} allows us to generalize this thinking about infinitesimal closeness in the setting of nonstandard extensions of topological spaces, even if no notion of distance or metric is assumed!

For a topological space $T$ in the standard universe, a point $x \in {^*}T$ is said to be \textit{nearstandard} to some $y \in T$ if $x$ belongs to ${^*}O$ for all open sets $O$ containing $y$. For a subset $A \subseteq T$, define its ``standard inverse'' as the set of points nearstandard to elements of $A$. That is, define
\begin{align}\label{standard inverse definition}
    \st^{-1}(A) \defeq \{x \in {^*}T: x \textit{ is nearstandard to } y \text{ for some } y \in A\}.
\end{align}
When the set $A$ is a singleton $\{x\}$, we write $\st^{-1}(x)$ in place of $\st^{-1}(\{x\})$. We denote $\st^{-1}(T)$ by $\ns({^*}T)$, as these are the nearstandard points in the nonstandard universe.  

The following characterization of \textit{Hausdroff spaces} follows almost immediately from this definition. 

\begin{lemma}\label{Hausdorff lemma}
 A topological space $T$ is Hausdorff if and only if for any disjoint collection $(A_i)_{i \in I}$ of subsets of $T$ (indexed by some set $I$), we have \begin{align}\label{disjoint standard inverse}
    \st^{-1}\left(\bigsqcup_{i \in I}A_i\right) = \bigsqcup_{i \in I} \st^{-1}(A_i),
\end{align}
where $\sqcup$ denotes a disjoint union.
\end{lemma}

If $\tau$ is the topology on $T$, then we know by transfer (of the sentences $O \in \tau$, one for each such $O$) that ${^*}O \in {^*}\tau$ for all $O \in \tau$. But these are not all of the elements of ${^*}\tau$. For instance, one can show by transfer that if $\tau$ is the usual topology on $\mathbb{R}$, then all intervals under the order on ${^*}\mathbb{R}$ are members of ${^*}\tau$, even those with infinitesimal lengths. The members of ${^*}\tau$ are sometimes called \textit{${^*}$-open} or \textit{internal open sets}.

In the rest of this section, unless otherwise specified, $(T, \tau)$ is a topological space in the standard universe. Also, unless otherwise specified, for each $x \in T$, we will denote by $\tau_x$ the set of all open neighborhoods of $x$. While $\st^{-1}(x)$ may in general be non-internal, one can approximate it from inside via internal open sets. This is a very useful feature of saturation as we show next.

\begin{lemma}[Approximation Lemma]
    For each $x \in T$, there exists an internal open set $U \in {^*}\tau$ such that $x \in U \subseteq \st^{-1}(x)$.
    \end{lemma}

\begin{proof}
   Note that we have $\st^{-1}(x) = \cap_{O \in \tau_x}{^*}O$ by definition. For each $O \in \tau_x$, consider the collection:
    \begin{align*}
      \mathcal{G}_O \defeq \{V \in {^*}\tau_x : V \in {^*}\mathcal{P}(O)\}  
    \end{align*}

    Recall that ${^*}\mathcal{P}(O)$ is the set of all internal subsets of of ${^*}O$. As a set in the nonstandard universe, $\mathcal{G}_O$ is internal for each $O \in \tau_x$ by the internal definition principle. The collection $\{\mathcal{G}_O\}_{O \in \tau_x}$ satisfies the finite intersection property. (Indeed, if $O_1, \ldots, O_n \in \tau_x$ are finitely many open neighborhoods around $x$, then ${^*}(O_1 \cap \ldots \cap O_n) \in {^*}\tau_x \cap {^*}\mathcal{P}(O_i)$ for each $i \in [n]$.) Now, any $U \in \cap_{O \in \tau_x} \mathcal{G}_O$ (which is non-empty by saturation) suffices.  
\end{proof}

Using the approximation lemma allows us to think about general continuous functions nonstandardly in an intuitive manner. Roughly, a function is continuous at a point $x$ if points ``close'' to $x$ are mapped to points ``close'' to $f(x)$ (while being ``close'' to a standard point is made precise via the $\st^{-1}$ operation just like in the case when the topological spaces were $\mathbb{R}$ in Proposition \ref{real-valued continuous}).

\begin{proposition}\label{general continuity characterization}
    Suppose $f \colon (T_1, \tau_1) \to (T_2, \tau_2)$ is a function between two topological spaces. Then, $f \colon T_1 \to T_2$ is continuous at $x \in T_1$ if and only if ${^*}f(\st^{-1}(x)) \subseteq \st^{-1}(f(x))$. (Note that we are using the same symbol $\st^{-1}$ for the different set-valued functions on both $T_1$ and $T_2$, the usage being unambiguous from context.)
\end{proposition}

\begin{proof}
    First suppose that $f$ is continuous at $x \in T_1$ and let $y \in \st^{-1}(x)$. If $V$ is an open neighborhood of $f(x)$ in $T_2$, then continuity at $x$ implies that there exists an open neighborhood $U$ of $x$ in $T_1$ such that $f(U) \subseteq V$. Since $\st^{-1}(x) \subseteq {^*}U$, we obtain ${^*}f(\st^{-1}(x)) \subseteq {^*}f({^*}U) = {^*}(f(U)) \subseteq {^*}V$. Since $y \in \st^{-1}(x)$ and the open neighborhood $V$ of $f(x)$ were arbitrary, this implies that ${^*}f(y) \in \st^{-1}(f(x))$ for all $y \in \st^{-1}(x)$, as desired. 

    For the converse direction, suppose that ${^*}f(\st^{-1}(x)) \subseteq \st^{-1}(f(x))$, where $x \in T_1$. Let $V$ be an open neighborhood of $f(x)$ in $T_2$. Then $\st^{-1}(f(x)) \subseteq {^*}V$, so that ${^*}f(\st^{-1}(x)) \subseteq {^*}V$ by assumption. Also, by the approximation lemma, there exists an internal open set $U \subseteq \st^{-1}(x)$ such that $x \in U$. Such a $U$ thus witnesses the truth of the following sentence in the nonstandard universe:
    \begin{align*}
        \exists W \in {^*}\tau_1 \left\{(x \in W) \land ({^*}f(W) \in {^*}\mathcal{P}(V)) \right\}.
    \end{align*}
    Transferring this sentence to the standard universe completes the proof. 
\end{proof}

We define the set of \textit{nearstandard points} of ${^*}T$ as follows: $$\ns({^*}T) \defeq \st^{-1}(T).$$
Thus, by Lemma \ref{Hausdorff lemma}, if $T$ is Hausdorff then $\st\co \ns({^*}T) \rightarrow T$ is a well-defined map. 

If $T$ is a topological space and a subset $T'$ is viewed as a topological space equipped with the subspace topology (i.e., a subset $G' \subseteq T'$ is open in $T'$ if and only if $G' = T' \cap G$ for some open subset $G$ of $T$), then there could be multiple ways to interpret the notation $\st^{-1}(A)$. In general, this may be an issue whenever we have two topological structures with respect to which we could be taking standard inverses. We will use `$\st$' and `$\st^{-1}$' whenever the underlying topological space is clear from context. If it is not clear from context, then we mention the relevant space in a subscript. Thus in the above situation where $T' \subseteq T$, we denote by $\st_{T}^{-1}$ and $\st_{T'}^{-1}$ the corresponding set functions on subsets of $T$ and $T'$ respectively. The following relation is immediate from the fact that the nonstandard extension of a finite intersection of sets is the same as the intersection of the nonstandard extensions.

\begin{lemma}\label{useful standard inverse relation}
Let $T$ be a topological space and let $T' \subseteq T$ be viewed as a topological space under the subspace topology. For a subset $A \subseteq T' \subseteq T$, we have:
\begin{align*}
    {^*}T' \cap \st_{T}^{-1}(A) = \st_{T'}^{-1}(A).
\end{align*}
\end{lemma}

The following  intuitive nonstandard characterizations of various topological properties are well-known.

\begin{proposition}\label{appendix topological characterizations}
Let $T$ be a topological space.
\begin{enumerate}[(i)]
    \item\label{a-open} A set $G \subseteq T$ is open if and only if $\st^{-1}(G) \subseteq {^*}G$.
    \item\label{a-closed} A set $F \subseteq T$ is closed if and only if for all  $x \in {^*}F \cap \ns({^*}T)$, the condition $x \in \st^{-1}(y)$ implies that $y \in F$.
    \item\label{a-compact} A set $K \subseteq T$ is compact if and only if ${^*}K \subseteq \st^{-1}(K)$.
\end{enumerate}
\end{proposition}

\begin{proof}
    \underline{Proof of \ref{a-open}}:
Suppose $G$ is open and $y \in \st^{-1}(G)$. Then there exists an $x \in G$ such that $y \in \st^{-1}(x) = \cap_{O \in \tau_x}{^*}O$. Since $G \in \tau_x$, we thus immediately obtain $y \in {^*}G$, as desired. 

Conversely, suppose $G \subseteq T$ is such that $\st^{-1}(G) \subseteq {^*}G$. Let $x \in G$ be arbitrary. In order to show that $G$ is open, it suffices to show that $G$ contains an open neighborhood of $x$. However, this follows from the combined use of approximation lemma and transfer. Indeed, by approximation lemma, there exists an internal open set $U \in {^*}\tau$ such that $x \in U \subseteq \st^{-1}(x) \subseteq \st^{-1}(G) \subseteq {^*}G$, so that the following statement is true in the nonstandard universe (whose transferred version is what we were looking for):
\begin{align*}
    \exists V \in {^*}\tau_x (V \in {^*}\mathcal{P}(G)).
\end{align*}

\underline{Proof of \ref{a-closed}}: Suppose $F$ is closed. Suppose, if possible, that there exists an $x \in {^*}F \cap \ns({^*}T)$ such that $x \in \st^{-1}(y)$ for some $y \in T \backslash F$. Since $T \backslash F$ is open, the nonstandard characterization \ref{a-open} for open sets implies that:
\begin{align*}
    x \in \st^{-1}(y) \subseteq \st^{-1}(T \backslash F) \subseteq {^*}(T \backslash F) = {^*}T \backslash {^*}F,
\end{align*}
where the last set equality is true by transfer. However, this is a contradiction since we started by assuming that $x \in {^*}F$.

Conversely, suppose $F \subseteq T$ has the property that for all $x \in {^*}F \cap \ns({^*}T)$, the condition $x \in \st^{-1}(y)$ implies $y \in F$. We shall show that $T \backslash F$ is open by verifying that it satisfies the nonstandard characterization \ref{a-open} for open sets. To that end, we must show that $\st^{-1}(T \backslash F) \subseteq {^*}(T \backslash F) = {^*}T \backslash {^*}F$. Suppose, if possible, that $x \in \st^{-1}(T \backslash F)$ but $x \not\in {^*}T \backslash {^*}F$. Then there exists $y \in T \backslash F$ such that $x \in \st^{-1}(y)$. But $x \in {^*}F \cap \ns({^*}T)$ and $y \not\in F$, which is a contradiction. 

\underline{Proof of \ref{a-compact}}: Suppose $K$ is a compact subset of $T$, and let $y \in {^*}K$. If $y \notin \st^{-1}(x) \text{ for all } x \in K$, then for each $x \in K$, there is $U_x \in \mathcal{\tau}_x$ with $y \not\in {^*}U_x$, that is, $y \in {^*}(K \backslash U_x)$. By compactness, $K \subseteq \cup_{x \in B}U_{x}$ for some $B \in \mathcal{P}_{\text{fin}}(K)$. Thus, ${^*}K \subseteq {^*}(\cup_{x \in B}U_{x}) = \cup_{x \in B}{^*}U_{x}$\footnote{The finiteness of the set $B$ is useful in obtaining this last set equality via transfer.}, which implies $\cap_{x \in B} {^*}(K \backslash U_x) = \emptyset$, a contradiction since $y$ belongs to this intersection.

For the converse direction, we shall again prove its contrapositive. if $K$ is not compact, there is an open cover $\{U_i : i \in I\}$ of $K$ that does not admit a finite subcover. Thus, for each finite $J \subseteq I$, there is $y_J \in {^*}(K \backslash \cup_{j \in j}U_j)$. By saturation applied to the collection $\{{^*}K \backslash U_i\}_{i \in I}$, we know that there exists $y \in \cap_{i \in I}{^*}K \backslash {^*}U_i$. Then $y \notin \st^{-1}(x)$ for any $x \in K$, as desired.
\end{proof}

The above characterization of compactness is sometimes called Robinson's characterization of compactness. In this author's opinion, Robinson's characterization of compactness provides one of the most poignant illustrations of the alien intuition for nonstandard analysis we built in the previous appendix. All points in the alien's interpretation of a compact space must be nearstandard to some point in the original space. Thus, as a closing motto, a topological space is compact if and only if even the aliens cannot make it too inaccessible!

\subsection{Pushing down Loeb measures defined on nonstandard extensions of topological spaces}\label{Nonstandard review}
One of the aims of this appendix is to describe the method of \textit{pushing down} Loeb measures. This will allow us to precisely talk about when an internal measure on the nonstandard extension of a topological space is, in a reasonable sense, close to a standard measure. (This idea will be made more precise at the end of our discussion on the $A$-topology in the next subsection; see, for example, Theorem \ref{Landers and Rogge main result} and Remark \ref{pushing down remark}.)

For the remainder of this section, we work with an internal measurable space $(\mathfrak{T}, \mathcal{T})$ that arises by nonstandardly interpreting the Borel measurable space induced by a topological space---that is, $\mathfrak{T}$ is the nonstandard extension of a topological space $T$, while $\mathcal{T}$ is the internal sigma algebra ${^*}\mathcal{B}(T)$ of internally Borel subsets of ${^*}T$. 

The following technical consequence of Proposition \ref{appendix topological characterizations} will be useful in Appendix \ref{AppB} (specifically in the proof of Theorem \ref{Prokhorov for P_N}), so we prove it next.  
\begin{lemma}\label{compactness lemma}
Suppose $(F_i)_{i \in I}$ is a collection of closed subsets of a Hausdorff space $T$ (where $I$ is an index set in the standard universe). Suppose that $K \defeq \cap_{i \in I} F_i$ is compact. Then for any open set $G$ with $K \subseteq G$, we have:
\begin{align}\label{K observation}
{^*}K \subseteq \left[ \left(\bigcap_{i \in I}{^*}{F_i}\right) \cap \ns({^*}T) \right] \subseteq {^*}G.
\end{align}
\end{lemma}
\begin{proof}
The first inclusion in \eqref{K observation} is true since ${^*}K \subseteq {^*}F_i$ for all $i \in I$ (which follows by transfer, because $K \subseteq F_i$ for all $i \in I$), and since $K$ is compact (so that all elements of ${^*}K$ are nearstandard by Proposition \ref{appendix topological characterizations}\ref{a-compact}). 

To see the second inclusion in \eqref{K observation}, suppose we take $x \in \cap_{i \in I}\left({^*}F_i \cap \ns({^*}T)\right)$. Since $T$ is Hausdorff, $x \in \ns({^*}T)$ has a unique standard part, say $\st(x) = y \in T$. Since $F_i$ is closed for each $i \in I$, it follows from Proposition \ref{appendix topological characterizations}\ref{a-closed} that $y \in F_i$ for all $i \in I$. As a consequence, $y \in K \subseteq G$. Hence, by Proposition \ref{appendix topological characterizations}\ref{a-open}, it follows that $x \in \st^{-1}(y) \subseteq {^*}G$, thus completing the proof.
\end{proof}

Given an internal probability space $({^*}T, {^*}\mathcal{B}(T), \nu)$, where $T$ is a Hausdorff space, if we know that $\st^{-1}(B)$ is Loeb measurable with respect to the corresponding Loeb space $({^*}T, L({^*}\mathcal{B}(T)), L\nu)$ for all Borel sets $B \in \mathcal{B}(T)$, then one can define a Borel measure on $(T, \mathcal{B}(T))$ by defining the measure of a Borel set $B$ as $L\nu(\st^{-1}(B))$. The fact that this defines a Borel measure in this case is not difficult to check. This measure is a probability measure only in the case that the set of {nearstandard points} $\ns({^*}T) \defeq \st^{-1}(T)$ is Loeb measurable with Loeb measure equaling one. 

Thus, an internal probability space $({^*}T, {^*}\mathcal{B}(T), \nu)$ naturally induces a probability measure on $(T, \mathcal{B}(T))$ if the following conditions two are met:
\begin{enumerate}[(i)]
    \item\label{verify i} The set $\st^{-1}(B)$ must be Loeb measurable for any Borel set $B \in \mathcal{B}(T)$.
    \item\label{verify ii} It must be the case that $L\nu(\ns({^*}T)) = 1$.
\end{enumerate}

Verifying when $\st^{-1}(B)$ is Loeb measurable for all Borel sets $B \in \mathcal{B}(T)$ is a tricky endeavor in general, and has been studied extensively. It is interesting to note that if the underlying space $T$ is regular\footnote{Recall that a topological space $T$ is called \textit{regular} if any closed set and a point outside that closed set can be separated via open sets. That is, given a closed set $F \subseteq T$ and given $x \in T \backslash F$, there exist disjoint open sets $G_1$ and $G_2$ such that $x \in G_1$ and $F \subseteq G_2$. }, then this condition is equivalent to the Loeb measurability of $\ns({^*}T)$ (this was investigated by Landers and Rogge as part of a larger project on universal Loeb measurability---see \cite[Corollary 3, p. 233]{Landers-Rogge}; see also Aldaz \cite{Aldaz}). Prior to Landers and Rogge, the same result was proved for locally compact Hausdorff spaces by Loeb \cite{Loeb-1979}. Also, Henson \cite{Henson} gave characterizations for measurability of $\st^{-1}(B)$ when the underlying space is either completely regular or compact. See also the discussion after Theorem 3.2 in Ross \cite{Ross_NATO} for other relevant results in this context. We will, however, not assume any additional hypotheses on our spaces, and hence we must study sufficient conditions for \ref{verify i} and \ref{verify ii} that work for any Hausdorff space. 

The results in Albeverio et al. \cite[Section 3.4]{Albeverio} are appropriate in the general setting of Hausdorff spaces. Their discussion is motivated by the works of Loeb \cite{Loeb-1976, Loeb-1979} and Anderson \cite{Anderson-thesis, Anderson-transactions}. The main 
result we borrow from their work is as follows (see \cite[Theorem 3.4.6, p. 89]{Albeverio} for its proof).\footnote{\label{Albeverio footnote}In the previous iterations of this work, we had only stated a simpler special case of the this theorem, the case in which $\nu$ is an internal measure on $({^*}T, {^*}\mathcal{B}(T))$. We had stated that we only needed that simpler version in our proof of de Finetti--Hewitt--Savage theorem. However, it has now become clear with the current revision that we do need the more general result from Albeverio et al. For more details, see Theorem \ref{Prokhorov for P_N} and Footnote \ref{main footnote}.}

\begin{theorem}\label{Albeverio theorem}
Let $T$ be a Hausdorff space. Suppose $\mathscr{T}$ is a base for the topology on $T$ that is closed under finite unions. Let $({^*}T, \mathcal{A}, \nu)$ be an internal, finitely additive probability space such that ${^*}O \in \mathcal{A}$ for all $O \in \mathscr{T}$. Let $({^*}T, L(\mathcal{A}), L\nu)$ denote the corresponding Loeb space.  

Assume further that for each $\epsilon \in \mathbb{R}_{>0}$, there is a compact set $K_{\epsilon}$ with
\begin{align}\label{nonstandard tightness assumption}
    \alpha_{K_{\epsilon}} \defeq \inf\{L\nu({^*}O): K_{\epsilon} \subseteq O \text{ and } O \in \mathscr{T}\} \geq 1 - \epsilon.
\end{align}
Then $L\nu \circ \st^{-1}$ is a Radon probability measure on $T$ such that $(L\nu \circ \st^{-1})(K) = \alpha_K$ for all compact sets $K$. 
\end{theorem}

The following corollary now follows from the definition of tightness.

\begin{corollary}\label{tightness corollary}
Let $T$ be a Hausdorff space and let $\mu$ be a tight probability measure on it. Then $L{^*}\mu \circ \st^{-1}$ is a Radon probability measure on $T$.
\end{corollary}

In the next section, we will study a useful topology on the space of all Borel probability measures over a topological space $T$. It will turn out that under the assumptions of Theorem \ref{Albeverio theorem}, the internal measure $\nu$ on $({^*}T, {^*}\mathcal{B}(T))$ is nearstandard to $L\nu \circ \st^{-1}$ with respect to this topology (see Theorem \ref{Landers and Rogge main result}). 

\subsection{The Alexandroff topology on the space of probability measures on a topological space}\label{A-topology}

Consider the set $\pt$ of all Borel probability measures on a topological space $T$. For each bounded measurable $f\co T \rightarrow \mathbb{R}$, define the map $E_f \co \pt \rightarrow \mathbb{R}$ by
\begin{align}\label{definition of E_f}
    E_f(\mu) \defeq \mathbb{E}_\mu(f) = \int_T f d\mu.
\end{align}

Then, the \textit{weak topology} on $\pt$ is defined to be the coarsest topology on $\pt$ which makes the maps $\mu \mapsto E_f(\mu)$ continuous whenever $f \colon T \to \mathbb{R}$ is bounded continuous. In the 1960s, Billingsley \cite{Billingsley1968} and Parthasarathy \cite{Parthasarathy1967Book} extensively studied the weak convergence of probability measures over metric spaces, obtaining many fundamental results in this subject. 

The weak topology on $\pt$, however, is interesting only when there are many real-valued continuous functions on $T$ to work with. If $T$ is not completely regular then the weak topology may actually be too coarse to be of any interest. Indeed, identifying the most general conditions on a topological space $T$ that guarantee the existence of at least one non-constant continuous real-valued function was part of Urysohn's research program (see \cite{Urysohn} where he posed this question). Hewitt \cite{Hewitt} and later Herrlich \cite{Herrlich} both showed that regularity of the space $T$ is generally not sufficient for this purpose. 

Also in the 1960s, Varadarajan \cite{varadarajan} studied the space of signed measures over a given topological space by interpreting it as a dual of the space of bounded continuous functions on the original topological space. 

Tops{\o}e's PhD thesis, \cite{Topsie}, was inspired by these works of Billingsley, Parthasarathy, and Varadarajan, as well as by Schwartz's announcement in \cite{SchwartzAnnouncement} that these works generalize to the setting of Radon measures over arbitrary Hausdorff spaces\footnote{Schwartz's own work \cite{Schwartz} in this context got published only a few years after Tops{\o}e's.} (see Tops{\o}e \cite[p. V]{Topsie} for more details). In this work, Tops{\o}e investigated the space $\prt$ of all \textit{Radon measures} over an arbitrary Hausdorff space $T$, aiming to generalize the known results about weak convergence to this setting. The topology on $\prt$ that Tops{\o}e studied extensively is the coarsest topology with respect to which the map $E_f$ is \textit{upper semicontinuous}\footnote{For a topological space $T$ and a function $f\co T \rightarrow \mathbb{R}$, we say that $f$ is \textit{upper semicontinuous at $x_0 \in T$} if the set $\{x \in T: f(x) < \alpha\}$ is open for all $\alpha \in \mathbb{R}$. We say that $f$ is \textit{upper semicontinuous} if it is upper semicontinuous at every point in $T$. 
}
whenever $f \colon T \to \mathbb{R}$ is a bounded upper semicontinuous function. Although Tops{\o}e still called this the ``weak topology'', in order to avoid confusion with the more common usage of that term, we are following Bogachev \cite[8.10(iv), p. 226, vol. 2]{Bogachev-measure} in calling this the \textit{$A$-topology} on $\prt$, with the ``$A$'' referring to A.D. Alexandroff \cite{ALexandroff}. 

More generally, if $T$ is any topological space, one may also consider the $A$-topology on the space $\pt$ of \textit{all} Borel probability measures on $T$ analogously---it being the coarsest topology on $\pt$ with respect to which the map $\mu \mapsto E_f(\mu)$ is upper semicontinuous whenever $f \colon T \to \mathbb{R}$ is a bounded upper semicontinuous function. In the case when $T$ is Hausdorff, the $A$-topology on $\prt$ defined in the previous paragraph coincides with the subspace topology induced by the $A$-topology on $\pt$. 

For a topological space $T$, let us denote the set of all bounded upper semicontinuous functions on $T$ by $USC_b(T)$. Similarly, $LSC_b(T)$ will denote the set of all bounded lower semicontinuous functions on $T$, which are precisely those functions $f$ for which $-f$ is bounded upper semicontinuous. For a bounded Borel measurable function $f\co T \rightarrow \mathbb{R}$ and $\alpha \in \mathbb{R}$, define the following sets:
\begin{align}
    \mathfrak{U}_{f, \alpha} &\defeq \{\mu \in \pt: \mathbb{E}_\mu(f) < \alpha\}, \label{U_f}\\
    \text{ and }   \mathfrak{L}_{f, \alpha} &\defeq \{\mu \in \pt: \mathbb{E}_\mu(f) > \alpha\} \label{L_f}.
\end{align}

Thus, the collection $\{\mathfrak{U}_{f, \alpha}: f \in USC_b(T), \alpha \in \mathbb{R}\}$ is a subbasis for the $A$-topology on $\pt$. (And furthermore, this collection \textit{coincides} with the collection $\{\mathfrak{L}_{f, \alpha}: f \in LSC_b(T), \alpha \in \mathbb{R}\}$.) Lemma \ref{characterization of weak topology 2} strengthens this observation by showing that focusing on the collection of indicator functions of either closed or open sets also suffices for constructing a subbasis for the $A$-topology. While readers may find a proof in Fremlin \cite[437J (d) and (f), p. 61]{Fremlin-4}\footnote{Fremlin works with measures that are not necessarily probability measures, and he defines a topology on the space of measures that he calls the \textit{narrow topology}, a terminlogy also used by Schwartz \cite{Schwartz}. The narrow topology on the space of Borel probability measures over a topological space is the same as what we are calling the $A$-topology.}, we shall provide our own nonstandard proof of this fact that relies on the following intuitive result from probability theory.

\begin{lemma}\label{tail expectation}
Suppose $\mathbb{P}_1$ and $\mathbb{P}_2$ are probability measures on the same space and $X$ is a bounded random variable such that 
\begin{align}\label{greater than relation}
    \mathbb{P}_1(X > x) \geq \mathbb{P}_2(X > x) \text{ for all } x \in \mathbb{R}.
\end{align}
Then, we have $\mathbb{E}_{\mathbb{P}_1}(X) \geq \mathbb{E}_{\mathbb{P}_2}(X)$.
\end{lemma}

\begin{proof}
With $\lambda$ denoting the Lebesgue measure on $\mathbb{R}$, we have the following representation of the expected value of any bounded random variable $X$ (see, for example, Lo \cite[Proposition 2.1]{demystifying}):
\begin{align}\label{tail formula}
    \mathbb{E}_{\mathbb{P}}(X) = \int_{(0, \infty)} \mathbb{P}(X > x) d\lambda(x) -\int_{(-\infty, 0)} \mathbb{P}(X < x) d\lambda(x).  
\end{align}

Let $\mathbb{P}_1$, $\mathbb{P}_2$ and $X$ be as in the statement of the lemma. Then, using \eqref{greater than relation}, we obtain the following for each $x \in \mathbb{R}$:
\begin{align}
\mathbb{P}_1(X < x) &= 1 - \mathbb{P}_1(X \geq x) \nonumber \\
&=  1 - \mathbb{P}_1 \left(\bigcap_{n \in \mathbb{N}} \left\{X > x - \frac{1}{n}\right\} \right) \nonumber \\
&=  1 - \lim_{n \rightarrow \infty} \mathbb{P}_1 \left( X > x - \frac{1}{n} \right) \nonumber \\
&\leq 1 - \lim_{n \rightarrow 
\infty} \mathbb{P}_2 \left( X > x - \frac{1}{n} \right) \nonumber \\
&= \mathbb{P}_2(X < x)\label{less than relation}.
\end{align} 

Using \eqref{tail formula}, \eqref{greater than relation} and \eqref{less than relation}, we thus obtain:
\begin{align*}
     \mathbb{E}_{\mathbb{P}_1}(X) &= \int_{(0, \infty)} \mathbb{P}_1(X > x) d\lambda(x) -\int_{(-\infty, 0)} \mathbb{P}_1(X < x) d\lambda(x) \\
     &\geq \int_{(0, \infty)} \mathbb{P}_2(X > x) d\lambda(x) -\int_{(-\infty, 0)} \mathbb{P}_2(X < x) d\lambda(x) \\
     &= \mathbb{E}_{\mathbb{P}_2}(X),
\end{align*}
completing the proof. 
\end{proof}

\begin{lemma}\label{characterization of weak topology 2}
For each Borel set $B \in \mathcal{B}(T)$, let 
\begin{align}
    \mathfrak{U}_{B, \alpha} &\defeq \{\mu \in \pt: \mu(B) < \alpha\}
    \text{ and }   \mathfrak{L}_{B, \alpha} \defeq \{\mu \in \pt: \mu(B) > \alpha\} \label{L_B}.
\end{align}
Then the $A$-topology on $\pt$ is generated by $\{\mathfrak{U}_{F, \alpha}: \alpha \in \mathbb{R} \text{ and }  F \text{ is closed}\}$ (which also equals the collection $\{\mathfrak{L}_{G, \alpha}: \alpha \in \mathbb{R} \text{ and } G \text{ is open}\}$) as a subbasis.
\end{lemma}

\begin{proof}
If $G$ is an open subset of $T$ and $\alpha \in \mathbb{R}$, then we have
\begin{align}
    \mathfrak{L}_{G, \alpha} = \bigcup_{\epsilon \in \mathbb{R}_{>0}} \mathfrak{U}_{ T\backslash G, 1 - \alpha + \epsilon}.
\end{align}

Since the complement of an open set is closed, this shows that a basic open set in the topology on $\pt$ generated by $\{\mathfrak{L}_{G, \alpha}: \alpha \in \mathbb{R} \text{ and } G \text{ is open}\}$ as a subbasis, is a finite intersection of sets that are unions of elements in the collection $\{\mathfrak{U}_{F, \alpha}: \alpha \in \mathbb{R} \text{ and }  F \text{ is closed}\}$. That is, a basic open set in the topology on $\pt$ generated by $\{\mathfrak{L}_{G, \alpha}: \alpha \in \mathbb{R} \text{ and } G \text{ is open}\}$ as a subbasis, is also open in the topology on $\pt$ generated by $\{\mathfrak{U}_{F, \alpha}: \alpha \in \mathbb{R} \text{ and }  F \text{ is closed}\}$ as a subbasis. A similar argument shows that a basic open set in the latter topology is also open in the former topology, thus proving that the two topologies are equal.

Let $\tau_1$ be the $A$-topology and $\tau_2$ be the topology induced by $\{\mathfrak{L}_{G, \alpha}: G \text{ open}, \alpha \in \mathbb{R}\}$ as a subbasis. From the discussion preceding this lemma, it is clear that $\tau_2 \subseteq \tau_1$. Conversely, let $U \in \tau_1$ and $\nu \in U$. By Lemma \ref{characterization of weak topology 1}, there exist finitely many $f_1, \ldots f_k \in LSC_b(T)$ and $\beta_1, \ldots, \beta_k \in \mathbb{R}$ such that the following holds:
\begin{align}\label{nu in these sets}
    \nu \in \cap_{i = 1}^{k} \mathfrak{L}_{f_i, \beta_i} \subseteq U.
\end{align}

Let $\mathbb{E}_{\nu}(f_i) = \delta_i > \beta_i$  for all $i \in \{1, \ldots, k\}$. For each $i \in \{1, \ldots, k\}$ and $\alpha \in \mathbb{R}$, let $G_{i, \alpha} = \{x \in T: f_i(x) > \alpha\}$, which is an open set by Lemma \ref{characterization of upper and lower semicontinuity}. Define 
\begin{align}\label{bad sets}
    \mathfrak{L}_{\alpha, \epsilon} \defeq  \cap_{i = 1}^k \mathfrak{L}_{G_{i, \alpha}, \nu(G_{i, \alpha}) - \epsilon} \text{ for all } \alpha \in \mathbb{R} \text{ and } \epsilon \in \mathbb{R}_{>0}.
\end{align}

Note that $\nu \in \mathfrak{L}_{\alpha, \epsilon}$ for all $\alpha \in \mathbb{R}$ and  $\epsilon \in \mathbb{R}_{>0}$, where $\mathfrak{L}_{\alpha, \epsilon}$ is a subbasic set for the topology $\tau_2$. Thus it is sufficient to prove the following claim.

\begin{claim}\label{weird claim}
There exists $n \in \mathbb{N}$ and $\alpha_1, \ldots, \alpha_n \in \mathbb{R}$, $\epsilon_1, \ldots, \epsilon_n \in \mathbb{R}_{>0}$ such that $$\cap_{j = 1}^n \mathfrak{L}_{\alpha_j, \epsilon_j} \subseteq \cap_{i = 1}^{k} \mathfrak{L}_{f_i, \beta_i} \subseteq U.$$
\end{claim}

\begin{proof}[Proof of Claim \ref{weird claim}]\renewcommand{\qedsymbol}{}
Suppose, if possible, that the claim is not true. Then for each $n \in \mathbb{N}$ and any $\alpha_1, \ldots, \alpha_n \in \mathbb{R}$ and $\epsilon_1, \ldots, \epsilon_n \in \mathbb{R}_{>0}$, there must exist some $\mu \in \pt$ such that $\mu \in \cap_{i = 1}^k \mathfrak{L}_{G_{i, \alpha_j}, \nu(G_{i, \alpha_j}) - \epsilon_j}$ for all $j \in \{1, \ldots, n\}$, but $\mu \not\in  \cap_{i = 1}^{k} \mathfrak{L}_{f_i, \beta_i}$. By transfer, the following internal set is non-empty for each $n \in \mathbb{N}$, $\vec{\alpha} = (\alpha_1, \ldots, \alpha_n) \in \mathbb{R}^n$ and  $\vec{\epsilon} \defeq (\epsilon_1, \ldots, \epsilon_n) \in (\mathbb{R}_{>0})^n$.
\begin{align}
    B_{\vec{\alpha}, \vec{\epsilon}} \defeq \{\mu \in {^*}\pt : &~\mu({^*}G_{i, \alpha_j}) > \nu(G_{i, \alpha_j}) - \epsilon_j \text{ for all } i \in \{1, \ldots, k\},j \in \{1, \ldots, n\} \nonumber \\ 
    &\text{ but }  {^*}\mathbb{E}_{\mu}({^*}f_i) \leq \beta_i \text{ for some } i \in \{1, \ldots, k\}
    \}. \label{bad internal set}
\end{align}

By the same argument (after concatenating different finite sequences of $\vec{\alpha}$\hspace{0.3mm}'s and $\vec{\epsilon}$\hspace{0.5mm}'s, we note that the collection $\cup_{n \in \mathbb{N}}\{B_{\vec{\alpha}, \vec{\epsilon}} : \vec{\alpha} \in \mathbb{R}^n, \vec{\epsilon} \in (\mathbb{R}_{>0})^n\}$ has the finite intersection property. By saturation, there exists $\mu \in {^*}\pt$ such that the following holds:
\begin{align}
\exists i_o \in \{1, \ldots, k\} \text{ such that } {^*}\mathbb{E}_{\mu}({^*}f_{i_0}) \leq \beta_{i_0} < \mathbb{E}_\nu(f_{i_0}) \text{ but } \nonumber  \\
    \mu({^*}G_{i_0, \alpha}) > \nu(G_{i_0, \alpha}) - \epsilon \text{ for all } \alpha \in \mathbb{R}, \epsilon \in \mathbb{R}_{>0}. \label{peculiar properties of mu}
\end{align}
But this implies that 
$L\mu({^*}G_{i_0, \alpha}) \geq L{^*}\nu({^*}G_{i_0, \alpha})$ for all $\alpha \in \mathbb{R}_{>0}$, which yields:
\begin{align}
    L\mu(\st({^*}f_{i_0}) > \alpha) &\geq \lim_{\epsilon \rightarrow 0} L\mu({^*}f_{i_0} > \alpha + \epsilon) \nonumber \\ 
    &\geq \lim_{\epsilon \rightarrow 0} L{^*}\nu({^*}f_{i_0} > \alpha + \epsilon) \nonumber \\
    &=  L{^*}\nu(\st({^*}f_{i_0}) > \alpha)\label{toward a contradiction}.
\end{align}
By Lemma \ref{tail expectation} and \eqref{toward a contradiction}, we thus obtain:
\begin{align}\label{toward a contradiction 2}
  \mathbb{E}_{L\mu}(\st({^*}f_{i_0})) \geq  \mathbb{E}_{L{^*}\nu}(\st({^*}f_{i_0})).
\end{align}

However, using the fact that finitely bounded internally measurable functions are $S$-integrable and that $\beta_{i_0}$ and $\mathbb{E}_{\nu}(f_{i_0})$ are real numbers, taking standard parts in the first inequality of \eqref{peculiar properties of mu} yields
\begin{align*}
  \mathbb{E}_{L\mu}(\st({^*}f_{i_0})) <  \mathbb{E}_{L{^*}\nu}(\st({^*}f_{i_0})),
\end{align*}
which directly contradicts \eqref{toward a contradiction 2}, completing the proof. 
\end{proof}
\end{proof}

Given a topological space $T$, we can thus describe a base $\mathbb{B}(\pt)$ for the $A$-topology on $\pt$ as follows:
\begin{align}\label{definition of basic open sets}
    \mathbb{B}(\pt) = \left\{\bigcap_{i \in [n]} \mathfrak{L}_{G_i, \alpha_i} : n \in \mathbb{N}, \alpha_i \in \mathbb{R} \text{ and } G_i \text{ open} \text{ for all } i \in [n] \right\}.
\end{align}
Let $\sigma(\mathbb{B}(\pt))$ denote the smallest sigma algebra containing $\mathbb{B}(\pt)$. Using Dynkin's $\pi-\lambda$ theorem, Lemma \ref{characterization of weak topology 2} almost immediately implies that the evaluation maps from $\pt$ to $[0,1]$ are measurable when $[0,1]$ is equipped with its Borel sigma algebra and $\pt$ is equipped with the sigma algebra $\sigma(\mathbb{B}(\pt))$. Of course $\sigma(\mathbb{B}(\pt)) \subseteq \mathcal{B}(\pt)$, and thus the evaluation maps are also measurable when $\pt$ is equipped with its Borel sigma algebra, a fact crucial for ensuring that the statement of Theorem \ref{most general de Finetti} in the main body of the paper is meaningful.

\begin{lemma}\label{evaluation of Borel sets'}
Let $\pt$ be the space of all Borel probability measures on a topological space $T$. Let $\pt$ be equipped with the $A$-topology, and let $\mathbb{B}(\pt)$ be the base for this topology as defined in \eqref{definition of basic open sets}. Then, for each Borel set $B \in \mathcal{B}(T)$, the evaluation map $$e_B\co \left(\pt, \sigma(\mathbb{B}(\pt))\right) \rightarrow \left([0,1], \mathcal{B}([0,1])\right),$$ defined by $e_B(\mu) \defeq \mu(B)$ is measurable. 
\end{lemma}

\begin{proof}
Consider the collection
\begin{align*}
    \mathcal{B} = \{B \in \mathcal{B}(T): e_B \text{ is Borel measurable}\}.
\end{align*}

This collection contains $T$, since $f_T$ is the constant function $1$, which is continuous. It is also closed under taking relative complements. That is, if $A \subseteq B$ and $A, B \in \mathcal{B}$ then $B \backslash A \in \mathcal{B}$ as well, since $e_{B \backslash A} = e_B - e_A$ in that case. Finally, $\mathcal{B}$ is closed under countable increasing unions. That is, if $(B_n)_{n \in \mathbb{N}} \subseteq \mathcal{B}$ is a sequence of sets such that $B_n \subseteq B_{n+1}$ for all $n \in \mathbb{N}$, then $B \defeq \cup_{n \in \mathbb{N}}B_n \in \mathcal{B}$ as well (this is because $e_B = \lim_{n \rightarrow \infty} e_{B_n}$ is a limit of Borel measurable functions in that case). Thus, $\mathcal{B}$ is a Dynkin system. 

Furthermore, $\mathcal{B}$ contains all open sets since for any open set $G \subseteq T$, the set $\{\mu \in \pt: \mu(G) > \alpha\}$ is Borel measurable (in fact, open) for all $\alpha \in \mathbb{R}$. Thus, by Dynkin's $\pi\text{-}\lambda$ theorem, it contains, and hence is equal to, $\mathcal{B}(T)$, completing the proof. 
\end{proof}

Going back to the space $\prt$ over a Hausdorff space $T$, one reason why that space is more convenient to work with than the space $\pt$ is that $\prt$ is Hausdorff, while $\pt$ is not always Hausdorff (a fact we are able to deduce a bit later in Corollary \ref{not Hausdorff corollary}). While a proof of the Hausdorffness of $\prt$ may be found in Tops{\o}e \cite[Theorem 11.2, p.49]{Topsie}, we provide below a proof for illustrating the key ideas.

\begin{theorem}\label{Topsoe theorem}
If $T$ is a Hausdorff space, then $\prt$ is also Hausdorff. 
\end{theorem}
\begin{proof}
Let $T$ be a Hausdorff space. Suppose $\mu, \nu$ are two distinct elements of $\prt$. Since they are distinct Borel measures, there exists an open set $G \subseteq T$ such that $\alpha \defeq \nu(G)$ and $\beta \defeq \mu(G)$ are distinct. Without loss of generality, assume $\alpha < \beta$. Since $\mu$ and $\nu$ are Radon measures, we can find a compact set $K$ such that $K \subseteq G$ and the following holds:
\begin{align}\label{use mu and nu are Radon}
    \nu(K) \leq \nu(G) = \alpha < \alpha + \frac{3(\beta - \alpha)}{4} < \mu(K) \leq \beta = \mu(G).
\end{align}

Since $T$ is Hausdorff, all compact subsets of $T$ are closed. In particular, $K$ is closed. Consider the subbasic open set $\mathfrak{V}$ defined by:
\begin{align*}
    \mathfrak{V} \defeq \left\{\gamma \in \prt: \gamma(K) < \alpha + \frac{\beta - \alpha}{4}\right\}. 
\end{align*}

By \eqref{use mu and nu are Radon}, it is clear that $\nu \in \mathfrak{V}$ and $\mu \not\in \mathfrak{V}$. For each $\gamma \in \mathfrak{V}$, by Radonness, there exists an open set $G_{\gamma}$ such that $K \subseteq G_{\gamma} \subseteq G$ and we have:
\begin{align}\label{gamma G gamma}
    \gamma(G_{\gamma}) < \alpha + \frac{\beta - \alpha}{2} \text{ for all } \gamma \in \mathfrak{V}.
\end{align}

Thus the following set, being the complement of a closed set (owing to the fact that an arbitrary intersection of closed sets is closed), is open:
\begin{align*}
    \mathfrak{U} \defeq \prt \backslash \left( \bigcap_{\gamma \in \mathfrak{V}} \left\{\theta \in \prt: \theta(G_{\gamma}) \leq \alpha + \frac{\beta - \alpha}{2} \right\} \right).
\end{align*}

By \eqref{use mu and nu are Radon}, it is clear that 
\begin{align*}
    \mu(G_{\gamma}) \geq \mu(K) > \alpha + \frac{3(\beta - \alpha)}{4} > \alpha + \frac{\beta - \alpha}{2} \text{ for all } \gamma \in \mathfrak{V}.
\end{align*}

As a consequence, we have $\mu \in \mathfrak{U}$. Furthermore, by \eqref{gamma G gamma}, it is clear that $\mathfrak{V} \cap \mathfrak{U} = \emptyset$, thus completing the proof. 
\end{proof}

Most topological constructions related to $\pt$ can be immediately adapted to $\prt$ by simply restricting them. The following analog of Lemma \ref{evaluation of Borel sets'} compiles some such results for convenience.

\begin{lemma}\label{Radon evaluation of Borel sets'}
Let $\prt$ be the space of all Radon probability measures on a Hausdorff space $T$. Then the following collection is a base for the $A$-topology on $\prt$:
   \begin{align}\label{Radon definition of basic open sets}
    \mathbb{B}(\prt) = \left\{\bigcap_{i \in [n]} \{\mu \in \prt: \mu(G_i) > \alpha_i\} : n \in \mathbb{N}, \forall i \in [n] ((\alpha_i \in \mathbb{R}) \land (G_i \text{ open})) \right\}.
\end{align}
Furthermore, for each Borel set $B \in \mathcal{B}(T)$, the evaluation map $$e_B\co \left(\prt, \sigma(\mathbb{B}(\prt))\right) \rightarrow \left([0,1], \mathcal{B}([0,1])\right)$$ defined by $e_B(\mu) \defeq \mu(B)$ is measurable.
\end{lemma}

As we shall see in Appendix \ref{AppB}, the sigma algebra $\sigma(\mathbb{B}(\prt))$ will be more convenient to work with (in comparison to the typically larger Borel sigma algebra $\mathcal{B}(\prt))$ when we study the natural $\prt$-valued map induced by empirical distributions arising out of a sequence of random variables. 

A crucial tool in that eventual study of hyperfinite empirical distributions in Appendix \ref{AppB} will be the following generalization of a theorem of Prokhorov\footnote{Prokhorov \cite[Theorem 1.12]{Prokhorov} originally proved that a collection $\mathfrak{A}$ of Borel probability measures on a Polish space $T$ (that is, a complete and separable metric space) is relatively compact (that is, the closure $\bar{\mathfrak{A}}$ of $\mathfrak{A}$ is compact) if and only if $\mathfrak{A}$ satisfies the following property that is now known as \textit{tightness} (being a property that is uniformly satisfied by all measures in $\mathfrak{A}$, it is sometimes called ``uniform tightness'' to avoid confusion with ``tightness of a particular measure'').
\begin{gather*}
    \text{(Tightness of $\mathfrak{A}$)} : \text{For each $\epsilon \in \mathbb{R}_{>0}$, there exists a compact set $K_{\epsilon} \subseteq T$ such that } \\
     \mu(K_{\epsilon}) \geq 1 - \epsilon \text{ for all } \mu \in \mathfrak{A}.
\end{gather*}}, which shows that tightness of a set of Radon probability measures over a Hausdorff space $T$ implies its relative compactness in $\prt$ equipped with the $A$-topology. This was independently proved by Tops{\o}e \cite[Theorem 9.1(iii), p. 43]{Topsie} (see also \cite{Topsoe-compactness}) and Schwartz \cite[Theorem 3, pp. 379-381]{Schwartz}).

\begin{theorem}[Schwartz and Tops{\o}e]\label{Radon Prokhorov}
Let $T$ be a Hausdorff space and let $\prt$ be the space of all Radon probability measures on $T$, equipped with the $A$-topology. Let $\mathfrak{A} \subseteq \prt$ be such that for any $\epsilon \in \mathbb{R}_{>0}$, there exists a compact set $K_{\epsilon} \subseteq T$ for which
\begin{align}\label{Radon tightness condition}
    \mu(K_{\epsilon}) \geq 1 - \epsilon \text{ for all } \mu \in \mathfrak{A}.
\end{align}
Then the closure of $\mathfrak{A}$ in $\prt$ is compact. 
\end{theorem}

The next section studies the nonstandard extension of $\pt$ equipped with its $A$-topology. The results we obtain in that study are of fundamental importance for the rest of the manuscript. In addition, this study naturally leads to a Prokhorov's theorem on the space of probability measures over an arbitrary Hausdorff space. Such a generalization of Prokhorov's theorem seems new to the literature (as both Tops{\o}e and Schwartz worked with the space of Radon measures), and it immediately implies Theorem \ref{Radon Prokhorov} which we need. Therefore, for the sake of comprehensiveness, we shall take a small digression at the end of the next subsection to establish our generalization of Prokhorov's theorem.  

\subsection{Standard parts with respect to the $A$-topology, and a generalization of Prokhorov's theorem}\label{A-topology standard parts}

Returning to the theme of Loeb measures, we are in a position to show that for any internal probability $\nu$ on $({^*}T, {^*}\mathcal{B}(T))$, if $L\nu \circ \st^{-1}$ is a legitimate Borel probability measure on $(T, \mathcal{B}(T))$, then $\nu$ is infinitesimally close to $L\nu \circ \st^{-1}$ in the sense that the former is nearstandard to the latter in ${^*}\pt$ (see Theorem \ref{Landers and Rogge main result}). This will generalize similar results obtained in the context of the topology of weak convergence by Anderson \cite[Proposition 8.4(ii), p. 684]{Anderson-transactions}, and by Anderson--Rashid \cite[Lemma 2, p. 329]{Anderson--Rashid} (see also Loeb \cite{Loeb-1979}). Our main result in this context is as follows:

\begin{theorem}\label{Landers and Rogge main result}
Let $T$ be a Hausdorff space. Suppose $({^*}T, {^*}\mathcal{B}(T), \nu)$ is an internal probability space, and let $({^*}T, L({^*}\mathcal{B}(T)), L\nu)$ be the associated Loeb space. If $L\nu \circ \st^{-1} \co \mathcal{B}(T) \rightarrow [0,1]$ is a Borel probability measure on $T$, then $\nu$ is nearstandard in ${^*}\mathfrak{P}(T)$ to $L\nu \circ \st^{-1}$. That is,
\begin{align}\label{nu nearstandard}
    \nu \in \st^{-1}(L\nu \circ \st^{-1}).
\end{align}
\end{theorem}

\begin{proof}
Let $\nu$ be as in the statement of the theorem. Thus, $L\nu \circ \st^{-1} \in \pt$, which implicitly also requires that $\st^{-1}(B) \in L({^*}\mathcal{B}(T))$ for all $B \in \mathcal{B}(T)$. For brevity, denote $L\nu \circ \st^{-1}$ by $\mu$. Let $\mathfrak{U}$ be any open neighborhood of $\mu$ in $\pt$. By Lemma \ref{characterization of weak topology 2}, there exist finitely many open sets $G_1, \ldots, G_n$ and $\alpha_1, \ldots, \alpha_n \in \mathbb{R}$ such that:
\begin{align}\label{definition of U}
    \mu \in \bigcap_{i = 1}^n\{\gamma \in \pt: \gamma(G_i) > \alpha_i\} \subseteq \mathfrak{U}.
\end{align}

By Proposition \ref{appendix topological characterizations}\ref{a-open}, we thus obtain:
\begin{align*}
    L\nu({^*}G_i) \geq L\nu(\st^{-1}(G_i)) = \mu(G_i) > \alpha_i \text{ for all } i \in \{1, \ldots, n\}.
\end{align*}

Since the $\alpha_i$ are real, it thus follows that 
\begin{align*}
    \nu({^*}G_i) > \alpha_i \text{ for all } i \in \{1, \ldots, n\}.
\end{align*}

Since the nonstandard extension of a finite intersection is the intersection of the nonstandard extensions, it is thus clear from \eqref{definition of mu} that 
\begin{align}
    \nu \in\bigcap_{i = 1}^n{^*}\{\gamma \in \pt: \gamma(G_i) > \alpha_i\} \subseteq {^*}\mathfrak{U}.
\end{align} 

Since $\mathfrak{U}$ was an arbitrary neighborhood of $\mu$, it thus follows  that $\nu \in \st^{-1}(\mu)$, completing the proof.
\end{proof}

\begin{remark}\label{pushing down remark}
For an internal probability measure $\nu$ on ${^*}T$, whenever $L\nu \circ \st^{-1}$ is a probability measure on the underlying topological space $T$, we typically call the measure $L\nu \circ \st^{-1}$ as being obtained by ``pushing down'' the Loeb measure $L\nu$. In fact, Albeverio et al. \cite[Section 3.4]{Albeverio} denotes $L\nu \circ \st^{-1}$ by ``$\st(L\nu)$'', calling it the standard part of $\nu$. Theorem \ref{Landers and Rogge main result} makes this more precise by showing that, in this case, the measure $L\nu \circ \st^{-1}$ is indeed nearstandard to $\nu \in {^*}\pt$ when we equip the space of probability measures $\pt$ with the $A$-topology. 
\end{remark}

Since the subset $\prt$ of Radon probability measures on $T$ is Hausdorff, the above theorem allows us to show that if $\nu \in {^*}\prt$ is such that $L\nu \circ \st^{-1} \in \prt$ (a sufficient condition for which is provided by Theorem \ref{Albeverio theorem}), then $L\nu \circ \st^{-1}$ is actually \textit{the} standard part of $\nu$ as an element of ${^*}\prt$. Indeed, Lemma \ref{useful standard inverse relation} may be invoked to obtain this result as a consequence of Theorem \ref{Landers and Rogge main result}. More precisely, we have the following corollary.  

\begin{corollary}\label{Radon Landers and Rogge main result}
Let $T$ be a Hausdorff space. Suppose $({^*}T, {^*}\mathcal{B}(T), \nu)$ is an internal probability space, and let $({^*}T, L({^*}\mathcal{B}(T)), L\nu)$ be the associated Loeb space. If $L\nu \circ \st^{-1} \co \mathcal{B}(T) \rightarrow [0,1]$ is a Radon probability measure on $T$, then $\nu$ is nearstandard in ${^*}\mathfrak{P}_r(T)$ to $L\nu \circ \st^{-1}$. That is,
\begin{align}\label{Radon nu nearstandard}
    \st(\nu) = L\nu \circ \st^{-1} \in \prt.
\end{align}
\end{corollary}

Theorem \ref{Landers and Rogge main result}, applied together with Corollary \ref{tightness corollary}, shows also that the nonstandard extension of any tight measure is nearstandard to a Radon measure. Thus, each tight measure is roughly close to a Radon measure from a topological point of view. More precisely, for each tight measure, there is a Radon measure such that the former belongs to each open neighborhood of the latter. Intuitively, the process of pushing down the Loeb measure induced by the nonstandard interpretation of a tight measure ``regularizes'' the original tight measure without taking it too far topologically. We record this as a corollary.

\begin{corollary}\label{not Hausdorff corollary 0}
Let $T$ be a Hausdorff space and $\mu$ be a tight probability measure on it. Then there exists a Radon measure $\mu'$ on $T$ such that $\mu \in \mathfrak{U}$ for all open neighborhoods $\mathfrak{U}$ of $\mu'$ in $\pt$. 
\end{corollary}
\begin{proof}
By Corollary 
\ref{tightness corollary} and Theorem \ref{Landers and Rogge main result}, we have that $\mu' \defeq L{^*}\mu \circ \st^{-1}$ is a Radon probability measure such that ${^*}\mu \in \st^{-1}(\mu')$. Also, by definition of $\st^{-1}$, we have that ${^*}\mu \in {^*}\mathfrak{U}$ for any open neighborhood $\mathfrak{U}$ of $\mu'$ in $\pt$. By transfer, we have that $\mu \in \mathfrak{U}$ for any open neighborhood $\mathfrak{U}$ of $\mu'$ in $\pt$.
\end{proof}

This, in particular, shows that the $A$-topology is not always Hausdorff, as it is well-known that there exist Borel measures that are tight but not Radon\footnote{See Vakhania--Tarildaze--Chobanyan\cite[Proposition 3.5, p.32]{Vakhania-Tariladze-Chobanyan} for such an example/construction.}. 

\begin{corollary}\label{not Hausdorff corollary}
There exists a topological space $T$ such that the $A$-topology on its space of Borel probability measures $\pt$ is not Hausdorff. 
\end{corollary}

As a consequence of Corollary \ref{not Hausdorff corollary 0}, we obtain the following result that is helpful in further generalizing a de Finetti--Hewitt--Savage theorem for Radon distributed exchangeable random variables to an appropriate result for tightly distributed exchangeable random variables in the main body of this manuscript. 

\begin{corollary}\label{what to do with tight measures}
    Let $T$ be a Hausdorff space and let $\mu$ be a tight probability measure on it. Then there exists a Radon measure $\mu'$ on $T$ such that $\int_T f d\mu = \int_T f d\mu'$ for all bounded continuous functions $f \colon T \to \mathbb{R}$.
\end{corollary}

\begin{proof}
    Let $\mu'$ be as in Corollary \ref{not Hausdorff corollary 0}. Let $f \colon T \to \mathbb{R}$ be a bounded continuous function. Suppose $\int_T f d\mu' = \alpha$. Thus, $\mu' \in \left\{\nu \in \pt: \int_T fd\nu > \alpha - \frac{1}{n}\right\}$ for all $n \in \mathbb{N}$, where the latter set is open since all continuous functions are, in particular, lower semincontinuous. By Corollary \ref{not Hausdorff corollary 0}, it follows that $\int_T f d\mu > \alpha - \frac{1}{n}$ for all $n \in \mathbb{N}$. Similarly, using the upper semicontinuity of $f$, it follows that  $\int_T f d\mu < \alpha + \frac{1}{n}$ for all $n \in \mathbb{N}$. Letting $n \to \infty$ now completes the proof. 
\end{proof}

We now show how Theorem \ref{Landers and Rogge main result} leads to a Prokhorov's theorem for the space of probability measures on any Hausdorff space, as promised earlier. 

\begin{theorem}[Prokhorov's theorem for the space of probability measures on any Hausdorff space]\label{Prokhorov}
Let $T$ be a Hausdorff space, and let $\pt$ be the space of all Borel probability measures on $T$, equipped with the $A$-topology. Let $\mathfrak{A} \subseteq \pt$ be such that for any $\epsilon \in \mathbb{R}_{>0}$, there exists a compact set $K_{\epsilon} \subseteq T$ for which
\begin{align}\label{tightness condition}
    \mu(K_{\epsilon}) \geq 1 - \epsilon \text{ for all } \mu \in \mathfrak{A}.
\end{align}
Then the closure of $\mathfrak{A}$ in $\pt$ is compact. 
\end{theorem}
\begin{proof}
Let $\mathfrak{A}$ be as in the statement of the theorem. Let $\bar{\mathfrak{A}}$ be its closure in $\pt$. By the nonstandard characterization of compactness, it suffices to show that ${^*}\bar{\mathfrak{A}} \subseteq \st^{-1}(\bar{\mathfrak{A}})$. Since $\bar{\mathfrak{A}}$ is closed, any nearstandard element in ${^*}\bar{\mathfrak{A}}$ must be nearstandard to an element of $\bar{\mathfrak{A}}$. Thus, it suffices to show that all elements in ${^*}\bar{\mathfrak{A}}$ are nearstandard. Toward that end, let $\nu \in {^*}\bar{\mathfrak{A}}$. For each $\epsilon \in \mathbb{R}_{>0}$, let $K_\epsilon$ be as in the statement of the theorem. The following claim now completes the proof, in view of Theorems \ref{Albeverio theorem} and \ref{Landers and Rogge main result}. 

\begin{claim}\label{Prokhorov claim}
$L\nu({^*}K_{\epsilon}) \geq 1 - \epsilon \text{ for all } \epsilon \in \mathbb{R}_{>0}$.
\end{claim}
\begin{proof}[Proof of Claim \ref{Prokhorov claim}]\renewcommand{\qedsymbol}{}
Suppose, if possible, that there is some $\epsilon \in \mathbb{R}_{>0}$ such that $L\nu({^*}K_{\epsilon}) < 1 - \epsilon$. Since $\epsilon \in \mathbb{R}_{>0}$, this implies that $\nu({^*}K_{\epsilon}) < 1 - \epsilon$ as well. By transfer, we conclude that $\nu$ belongs to ${^*}\mathfrak{U}$, where $\mathfrak{U}$ is the following open subset of $\pt$.
\begin{align}\label{Prokhorov subbasic}
    \mathfrak{U} \defeq \{\gamma \in \pt: \gamma(K_{\epsilon}) < 1 - \epsilon\}.
\end{align}

Note that $\mathfrak{U}$ is indeed open in $\pt$, since $K_{\epsilon}$, being a compact subset of the Hausdorff space $T$, is closed in $T$. By the definition of closure, we know that any open neighborhood of an element in the closure of $\mathfrak{A}$ must have a nonempty intersection with $\mathfrak{A}$. By transfer, we thus find an element $\mu \in \mathfrak{U} \cap \mathfrak{A}$. But this is a contradiction (in view of \eqref{tightness condition} and \eqref{Prokhorov subbasic}), thus completing the proof. 
\end{proof} 
\end{proof}

To conclude this section, we show how the above proof immediately generalizes to a proof of Theorem \ref{Radon Prokhorov}.

\begin{proof}[Proof of Theorem \ref{Radon Prokhorov}]
Let $\mathfrak{A} \subseteq \prt$ be as in the statement of the theorem. It suffices to show that ${^*}\bar{\mathfrak{A}} \cap {^*}\prt \subseteq \st_{\prt}^{-1}(\bar{\mathfrak{A}})$, where $\bar{\mathfrak{A}}$ is the closure of $\mathfrak{A}$ in the space $\prt$. This follows immediately from Theorem \ref{Prokhorov} and the fact that $\st_{\prt}^{-1}(\bar{\mathfrak{A}}) = {^*}\prt \cap \st_{\pt}^{-1}(\bar{\mathfrak{A}})$ (which can be seen by Lemma \ref{useful standard inverse relation}).
\end{proof}

\subsection{Some uniqueness results pertaining to mixing measures}
We compile in the final section of this appendix some results pertaining to the uniqueness of the mixing measure in the context of Theorem \ref{most general de Finetti}. For technical reasons, we first recall the following generalization of the monotone class theorem (see Dellacherie and Meyer \cite[Theorem 21, p. 13-I]{Dellacherie} for a proof of this result).

\begin{theorem}\label{Dellacherie's theorem}
Let $\mathbb{H}$ be an $\mathbb{R}$-vector space of bounded real-valued functions on some set $\mathcal{S}$ such that the following hold:
\begin{enumerate}[(i)]
    \item $\mathbb{H}$ contains the constant functions.
    \item $\mathbb{H}$ is closed under uniform convergence. 
    \item For every uniformly bounded increasing sequence of nonnegative functions $f_n \in \mathbb{H}$, the function $\lim_{n \rightarrow \infty} f_n$ belongs to $\mathbb{H}$.
\end{enumerate}

If $\mathcal{C}$ is a subset of $\mathbb{H}$ which is closed under multiplication, then the space $\mathbb{H}$ contains all bounded functions measurable with respect to $\sigma(\mathcal{C})$ - the smallest sigma algebra with respect to which all functions in $\mathcal{C}$ are measurable. 
\end{theorem}

As an application, we have the following uniqueness result, which is different from the related uniqueness result of Hewitt--Savage \cite[Theorem 9.4, p. 489]{Hewitt-Savage-1955} in two ways. Firstly, we are focusing on the space of Radon probability measures (as opposed to the space of Baire probability measures), and secondly, we are working with the Borel sigma algebra induced by the $A$-topology (as opposed to the cylinder sigma algebra induced by Baire sets).

\begin{theorem}\label{uniqueness of Radon presentable}
Let $T$ be a Hausdorff space and let $\prt$ be the space of all Radon probability measures on $T$ under the $A$-topology. Suppose that $\mathscr{P}, \mathscr{Q} \in \mathfrak{P}_r(\prt)$ are such that the following holds:
\begin{align}
    \int_{\prt} \mu(B_1) \cdot \ldots \cdot \mu(B_n) d\mathscr{P}(\mu) = \int_{\prt} \mu(B_1) \cdot \ldots \cdot \mu(B_n) d\mathscr{Q}(\mu) \nonumber \\
    \text{ for all } n \in \mathbb{N} \text{ and } B_1, \ldots, B_n \in \mathcal{B}(T). \label{two measures}
\end{align}
Then it must be the case that $\mathscr{P} = \mathscr{Q}$.
\end{theorem}

\begin{proof}
For $m \in \mathbb{N}$, let $\mathbf{B}([0,1]^m, \mathbb{R})$ denote the space of all bounded Borel measurable functions $f\co [0,1]^m \rightarrow \mathbb{R}$. For each $m \in \mathbb{N}$, consider the following collection of functions:
\begin{align*}
    \mathcal{G}_m \defeq \{f \in \mathbf{B}([0,1]^m, \mathbb{R}): \mathbb{E}_{\mathscr{P}}\left[ f \left(\mu(B_1), \ldots, \mu(B_m) \right) \right] = \mathbb{E}_{\mathscr{Q}}\left[ f \left(\mu(B_1), \ldots, \mu(B_m) \right) \right] \\  \text{ for all } B_1, \ldots, B_m \in \mathcal{B}(S)
    \}.
\end{align*}

Note that the expected values in the definition of $\mathcal{G}_m$ are well-defined because of Lemma \ref{evaluation of Borel sets'}. It is clear that for each $m \in \mathbb{N}$, the collection $\mathcal{G}_m$ contains all polynomials over $m$ variables. Indeed, the collection $\mathcal{G}_m$ is an $\mathbb{R}$-vector space (that is, closed under finite linear combinations), and for a monomial $f\co [0,1]^m \rightarrow \mathbb{R}$ of the type $f(x_1, \ldots, x_m) = {x_1}^{a_1} \cdot \ldots \cdot {x_m}^{a_m}$ (where $a_1, \ldots, a_m \in \mathbb{Z}_{\geq 0}$), the expectation $\mathbb{E}_{\mathscr{P}}\left[ f \left(\mu(B_1), \ldots, \mu(B_m) \right) \right]$ is equal to $\mathbb{E}_{\mathscr{Q}}\left[ f \left(\mu(B_1), \ldots, \mu(B_m) \right) \right]$ by \eqref{two measures}. That $\mathcal{G}_m$ satisfies the conditions in Theorem \ref{Dellacherie's theorem} is also clear by dominated convergence theorem. It is straightforward to verify that the smallest sigma algebra on $[0,1]^m$ with respect to which all polynomials are measurable is the Borel sigma algebra on $[0,1]^m$. Since the set of polynomials over $m$ variables is closed under multiplication, it thus follows from Theorem \ref{Dellacherie's theorem} that for each $m \in \mathbb{N}$, the collection $\mathcal{G}_m$ contains all bounded Borel measurable functions $f \co [0,1]^m \rightarrow \mathbb{R}$. 

Let $\mathcal{G}$ be the collection of those Borel subsets of $\prt$ that are assigned the same measure by $\mathscr{P}$ and $\mathscr{Q}$. More formally, we define:
\begin{align}\label{where P and Q agree}
    \mathcal{G} \defeq \{\mathfrak{B} \in \mathcal{B}(\prt): \mathscr{P}(\mathfrak{B}) = \mathscr{Q}(\mathfrak{B})\}.
\end{align}

Taking $f$ to be the indicator function of a measurable rectangle in $[0,1]^m$, we have thus shown that $\mathcal{G}$ contains the following collection of cylinder sets:
\begin{align}\label{definition of cylinder sets}
    \mathcal{C} \defeq \{ C_{(B_1, \ldots, B_m), (A_1, \ldots A_m)}: m \in \mathbb{N}; B_1, \ldots, B_m \in \mathcal{B}(T); A_1, \ldots, A_m \in \mathcal{B}(\mathbb{R}) \}, 
\end{align}
where 
\begin{align*}
    C_{(B_1, \ldots, B_m), (A_1, \ldots A_m)} \defeq \{\mu \in \prs: \mu(B_1) \in A_1, \ldots, \mu(B_m) \in A_m\} \\
    \text{ for all } m \in \mathbb{N}; B_1, \ldots, B_m \in \mathcal{B}(T); A_1, \ldots, A_m \in \mathcal{B}(\mathbb{R}).
\end{align*}

It is clear that the collection $\mathcal{C}$ contains the basic open subsets with respect to the subbasis $\{\{\mu \in \prt: \mu(G) > \alpha\}: G \text{ an open subset of $T$} \text{ and } \alpha \in \mathbb{R}\}$ for the $A$-topology on $\prt$. Since $\mathcal{G}$ is a sigma algebra, all finite unions of these basic open sets are in $\mathcal{G}$. (In fact, all countable unions are in $\mathcal{G}$, but we do not need this fact here.) 

Let $\mathfrak{C}$ be a compact subset of $\prt$ and let $\epsilon \in \mathbb{R}_{>0}$ be given. Since $\mathscr{P}$ and $\mathscr{Q}$ are Radon measures, there exists an open subset $\mathfrak{U}$ of $\prt$ such that $\mathfrak{C} \subseteq \mathfrak{U}$ and 
\begin{align}\label{used Radonness of P and Q}
    \mathscr{P}(\mathfrak{U} \backslash \mathfrak{C}) < \epsilon \text{ and }  \mathscr{Q}(\mathfrak{U} \backslash \mathfrak{C}) < \epsilon.
\end{align}
Cover $\mathfrak{C}$ by finitely many basic open subsets contained in $\mathfrak{U}$ and let $\mathfrak{V}$ be the union of these basic open subsets. Then, we have (using \eqref{used Radonness of P and Q}):
\begin{align}\label{used Radonness of P and Q 2}
    \mathscr{P}(\mathfrak{V} \backslash \mathfrak{C}) < \epsilon \text{ and }  \mathscr{Q}(\mathfrak{V} \backslash \mathfrak{C}) < \epsilon.
\end{align}
Being a finite union of basic open sets, we have $\mathfrak{V} \in \mathcal{G}$, or in other words:
\begin{align}\label{agreed on V}
    \mathscr{P}(\mathfrak{V}) = \mathscr{Q}(\mathfrak{V}).
\end{align}
Using \eqref{used Radonness of P and Q 2} and \eqref{agreed on V} (and the triangle inequality), we thus obtain:
\begin{align}\label{used triangle inequality}
    \abs{\mathscr{P}(\mathfrak{C}) - \mathscr{Q}(\mathfrak{C})} < \epsilon.
\end{align}

Since $\mathfrak{C}$ was an arbitrary compact subset of $\prt$ and $\epsilon \in \mathbb{R}_{>0}$ was arbitrary, this shows that the measures $\mathscr{P}$ and $\mathscr{Q}$ agree on all compact subsets of $\prt$. Since they are Radon measures, it is thus clear now that they agree on all Borel subsets of $\prt$, completing the proof. 
\end{proof}

In the above proof, the only place where Radonness was used was in extending the uniqueness result from the cylinder sigma algebra on $\prt$ to the Borel sigma algebra on $\prt$. In particular, the same argument shows that without working with Radon measures, one still has uniqueness if we focus on measures over the smallest sigma algebra generated by cylinder sets. More precisely, we have the following corollary. 

\begin{corollary}\label{non-Radon de Finetti uniqueness}
Let $T$ be a topological space and let $\pt$ be the space of all Borel probability measures on $T$ under the $A$-topology. Let $\mathcal{C}(\pt)$ be the smallest sigma algebra such that for any $B \in \mathcal{B}(T)$, the evaluation function $e_B\co \ps \rightarrow \mathbb{R}$, defined by $e_B(\nu) = \nu(B)$, is measurable. Then $\mathcal{C}(\pt) \subseteq \mathcal{B}(\pt)$. 

Suppose $\mathscr{P}, \mathscr{Q}$ are two probability measures on $(\pt, \mathcal{C}(\pt))$ such that the following holds:
\begin{align*}
    \int_{\pt} \mu(B_1) \cdot \ldots \cdot \mu(B_n) d\mathscr{P}(\mu) = \int_{\pt} \mu(B_1) \cdot \ldots \cdot \mu(B_n) d\mathscr{Q}(\mu) \\
    \text{ for all } n \in \mathbb{N} \text{ and } B_1, \ldots, B_n \in \mathcal{B}(T). 
\end{align*}
Then it must be the case that $\mathscr{P} = \mathscr{Q}$.
\end{corollary}

\section{On empirical measures induced by hyperfinitely many random variables sampled from a Radon distribution}\label{AppB}
Let $(\Omega, \mathcal{F}, \mbp)$ be a probability space. Let $S$ be a Hausdorff space equipped with its Borel sigma algebra $\mathcal{B}(S)$. Suppose $X_1, X_2, \ldots $ is a sequence of identically distributed $S$-valued random variables on $\Omega$---that is, the pushforward measure $\mathbb{P} \circ {X_i}^{-1}$ on $(S, \mathcal{B}(S))$ is the same for all $i \in \mathbb{N}$.

Throughout this appendix, we will further assume that the common distribution of the $X_i$ is Radon. By tightness, there exists an increasing sequence of compact subsets $(K_n)_{n \in \mathbb{N}}$ of $S$ such that:

\begin{align}\label{K_n}
   \mathbb{P}(X_1 \in K_n) > 1 - \frac{1}{n^3} \text{ for all } n \in \mathbb{N}.
\end{align}

For each $\omega \in \Omega$ and $n \in \mathbb{N}$, define the \textit{empirical measure} $\mu_{\omega, n}$ on $\mathcal{B}(S)$ as follows:
\begin{align}\label{definition of mu}
    \mu_{\omega, n}(B) \defeq \frac{\# \{i \in [n]: X_i(\omega) \in B \}}{n} \text{ for all } B \in \mathcal{B}(S). 
\end{align}

Although we are calling $\mu_{\omega, N}$ a \textit{hyperfinite} empirical measure because $N \in {^*}\mathbb{N}$, we do not need to assume $N > \mathbb{N}$ for the results obtained in this appendix. In fact, the results we obtain would also be true if we started with any ${^*}$-finite collection of internatl random variables sampled from a common nearstandard internal Radon distribution, but we avoid that generality as a matter of convenience.  

Also, we are abusing notation by using $(X_i)$ to denote both the standard sequence $(X_i)_{i \in \mathbb{N}}$ of random variables and the nonstandard extension of this sequence, the intended usage being always clear from context. More precisely, if $\mathscr{X} \co \Omega \times \mathbb{N} \rightarrow S$ is defined by $\mathscr{X}(\omega, i) \defeq X_i(\omega)$ for all $\omega \in \Omega$ and $i \in \mathbb{N}$, then for any $i \in {^*}\mathbb{N}$, we shall also use $X_i$ to denote an internal function $X_i \co {^*}\Omega \rightarrow {^*}S$ is defined as follows:
\begin{align*}
    X_i(\omega) = {^*}\mathscr{X}(\omega, i) \text{ for all } \omega \in {^*}\Omega \text{ and } i \in {^*}\mathbb{N}.
\end{align*}

This appendix, which studies the structure of hyperfinite empirical measures as elements in the space of all Radon probability measures on $S$, is divided into four subsections. Section \ref{3.1} deals with some basic properties that are satisfied by almost all hyperfinite empirical measures. Section \ref{3.2} deals with the study of the pushforward measure induced on the space ${^*}\prs$ of internal Radon measures on ${^*}S$ by the map $\omega \mapsto \mu_{\omega, N}$. The goal of Section \ref{3.3} is to show in a precise sense that the standard part of a hyperfinite empirical measure evaluated at a Borel set is almost surely given by the standard part of the measure of the nonstandard extension of that Borel set (see Theorem \ref{commutativity theorem}). Section \ref{3.4} synthesizes the theory built so far in order to express some Loeb integrals on the space of all internal Radon probability measures in terms of the corresponding integrals on the space of Radon probability measures on $S$.

\subsection{Hyperfinite empirical measures as random elements in the space of all internal Radon measures}\label{3.1}
Being supported on a finite set, it is clear that $\mu_{\omega, n}$ is, in fact, a Radon probability measure on $S$ for all $\omega \in \Omega$ and $n \in \mathbb{N}$. For technical reasons it would be desirable to work with a sigma algebra on $\prs$ that ensures that the maps $\omega \mapsto \mu_{\omega, n}$ (where $n \in \mathbb{N}$) are measurable as functions from the measurable space $(\Omega, \mathcal{F})$. 

It turns out that the Borel sigma algebra $\mathcal{B}(\prs)$ is too large and unwieldy for this purpose, and we instead work with $\sigma(\mathbb{B}(\prs))$, the smallest sigma algebra containing the base $\mathbb{B}(\prs)$ for the $A$-topology on $\prs$ defined as follows:
\begin{align}\label{Bprs}
     \mathbb{B}(\prs) = \left\{\bigcap_{i \in [k]} \{\mu \in \prs: \mu(G_i) > \alpha_i\} : k \in \mathbb{N}, \forall i \in [n] ((\alpha_i \in \mathbb{R}) \land (G_i \text{ open})) \right\}.
\end{align}

Lemma \ref{Radon evaluation of Borel sets'} shows that the sigma algebra $\sigma(\mathbb{B}(\prs))$ is large enough to ensure that the evaluation maps are all measurable, while the next lemma shows that it is small enough to ensure that the empirical distributions $\mu_{\cdot, n} \co \Omega \to \prs$ are also measurable. 

\begin{lemma}\label{measurable}
For each $n \in \mathbb{N}$, the map $\mu_{\cdot, n} \co (\Omega, \mathcal{F}) \rightarrow (\prs, \sigmabprs)$ defined by \eqref{definition of mu} is measurable. In particular, for any $B \in \mathcal{B}(S)$, the map $\mu_{\cdot, n}(B) \co (\Omega, \mathcal{F}) \rightarrow ([0,1], \mathcal{B}([0,1])$ (that is, $\omega \mapsto \mu_{\omega, n}(B)$) is measurable for each $n \in \mathbb{N}$. 
\end{lemma}

\begin{proof}
We only need to prove the first assertion, as the second follows from it by composing with evaluation maps (see Lemma \ref{Radon evaluation of Borel sets'}). 

For each $n, k \in \mathbb{N}$, given any open subsets $G_1, \ldots, G_k \subseteq S$ and real numbers $\alpha_1, \ldots, \alpha_k \in \mathbb{R}$, we have:
\begin{align}
   \bigcap_{i \in [k]} \{\omega \in \Omega: \mu_{\omega, n}(G_i) >  \alpha_i\} = \bigcap_{i \in [k]} \left\{\omega \in \Omega: \sum_{j\in [n]} \mathbbm{1}_{G_i}(X_j(\omega)) > n\alpha_i \right\}, 
\end{align}
where the latter set belongs to $\mathcal{F}$, since for each $i \in [k]$ and $j \in [n]$, the composition  $\mathbbm{1}_{G_i} \circ X_j \colon \Omega \to \mathbb{R}$ is measurable. This implies that for each $n \in \mathbb{N}$, we have ${\mu_{\cdot, n}}^{-1}(B) \in \mathcal{F}$ for all $B \in \mathbb{B}(\prs)$. Since the collection $\{B \subseteq \prs: {\mu_{\cdot, n}}^{-1}(B) \in \mathcal{F}\}$ is a sigma algebra, it must contain $\sigmabprs$, thus completing the proof.
\end{proof}

By transfer, we obtain the following immediate consequence
\begin{corollary}\label{transfer corollary 1}
For each $N \in {^*}\mathbb{N}$, the map $\mu_{\cdot, N} \co ({^*}\Omega, {^*}\mathcal{F}) \rightarrow ({^*}\prs, {^*}\sigmabprs)$ is an internally measurable function. That is, $\mu_{\cdot, N} \co {^*}\Omega \rightarrow {^*}\prs$ is internal and the set $\{\omega \in {^*}\Omega: \mu_{\omega, N} \in \mathfrak{B}\}$ belongs to ${^*}\mathcal{F}$ whenever $\mathfrak{B} \in {^*}\sigmabprs$. Furthermore, for each $B \in {^*}\mathcal{B}(S)$, the map $\mu_{\cdot, N}(B) \co ({^*}\Omega, {^*}\mathcal{F}) \rightarrow ({^*}[0,1], {^*}\mathcal{B}([0,1]))$ is internally measurable. 
\end{corollary}

By the usual Loeb measure construction, we have a collection of complete probability spaces indexed by ${^*}\Omega$, namely $({^*}S, L_{\omega, N}({^*}\mathcal{B}(S)), L\mu_{\omega, N})_{\omega \in {^*}\Omega}$. 

We now prove that with respect to the Loeb measure $L{^*}\mathbb{P}$, almost all of the measures $L\mu_{\omega, N}$ assign full mass to the set $\ns({^*}S)$ of \textit{nearstandard} elements of ${^*}S$. This implicitly requires us to first show that for all $\omega$ in an $L{^*}\mathbb{P}$ almost sure subset of ${^*}\Omega$, the set $\ns({^*}S)$ is in the Loeb sigma algebra $L_{\omega, N}({^*}\mathcal{B}(S))$ corresponding to the internal probability space $({^*}S, {^*}\mathcal{B}(S), \mu_{\omega, N})$.

\begin{lemma}\label{neastandard for almost all omega}
Let $S$ be a Hausdorff space and $N \in {^*}\mathbb{N}$. There is a set $E_N \in L({^*}\mathcal{F})$ with $L{^*}\mathbb{P}(E_N) = 1$ such that for any $\omega \in E_N$, we have $L\mu_{\omega, N} (\ns({^*}S)) = 1$.
\end{lemma}

\begin{proof}
Let $(K_n)_{n \in \mathbb{N}}$ be as in \eqref{K_n}. By the second part of Corollary \ref{transfer corollary 1}, the function $\omega \mapsto \mu_{\omega, N}({^*}K_n)$ is an internal random variable for each $n \in \mathbb{N}$. Since it is finitely bounded, it is $\mathbf{S}$-integrable with respect to the Loeb measure $L{^*}\mathbb{P}$. Thus, for each $n \in \mathbb{N}$, the $[0,1]$-valued function $L\mu_{\cdot, N}({^*}K_n)$, defined by $\omega \mapsto L\mu_{\omega, N}({^*}K_n)$, is Loeb measurable, and furthermore we have:
\begin{align*}
   \mathbb{E}_{L{^*}\mathbb{P}} (L\mu_{\cdot, N}({^*}K_n)) &\approx {^*}\mathbb{E}_{{^*}\mathbb{P}} (\mu_{\cdot, N}({^*}K_n)) \\
   &= {^*}\mathbb{E}_{{^*}\mathbb{P}} \left[ \sum_{i = 1}^N \frac{1}{N} \mathbbm{1}_{{^*}K_n}(X_i) \right] \\
   &= \frac{1}{N} \left[ \sum_{i = 1}^N {^*}\mathbb{P}(X_i \in {^*}K_n) \right] \\
   &> \frac{1}{N}\left[ N \left(1 - \frac{1}{n^3} \right)\right] = 1 - \frac{1}{n^3},
\end{align*}
where the last line follows from \eqref{K_n} and the fact that each $X_i$ has the same distribution. 

For each $\omega \in {^*}\Omega$, the upper monotonicity of the measure $L_{\omega, N}$ implies that $\lim_{n \rightarrow \infty} L\mu_{\omega, N}({^*}K_n) = L\mu_{\omega, N} \left( \cup_{n \in \mathbb{N}} {^*}K_n\right)$. Thus, being a limit of Loeb measurable functions, $\lim_{n \rightarrow \infty} L\mu_{\cdot, N}({^*}K_n) = L\mu_{\cdot, N} \left( \cup_{n \in \mathbb{N}} {^*}K_n\right)$, is also Loeb measurable. Therefore, by the monotone convergence theorem, we obtain: 

\begin{align}
    \mathbb{E}_{L{^*}\mathbb{P}} \left[L\mu_{\cdot, N} \left( \cup_{n \in \mathbb{N}} {^*}K_n\right) \right] &= \mathbb{E}_{L{^*}\mathbb{P}} \left[ \lim_{n \rightarrow \infty} L\mu_{\cdot, N}({^*}K_n) \right] \nonumber \\
    &=\lim_{n \rightarrow \infty} \mathbb{E}_{L{^*}\mathbb{P}}(L\mu_{\cdot, N}({^*}K_n)) \nonumber \\
    &\geq \lim_{n \rightarrow \infty} \left( 1 - \frac{1}{n^3}\right) = 1. \label{expectation is 1} 
\end{align}
But $0\leq L\mu_{\omega, N} \left[ \cup_{n \in \mathbb{N}} {^*}K_n\right] \leq 1$ for all $\omega \in {^*}\Omega$. Therefore, by \eqref{expectation is 1}, we get:

\begin{align}
    L{^*}\mathbb{P}(E_N) = 1, 
\end{align}
where 
\begin{align}\label{almost sure set}
    E_N = \left\{ \omega: L\mu_{\omega, N} \left[ \cup_{n \in \mathbb{N}} {^*}K_n \right] = 1 \right\} \in L({^*}\mathcal{F}).
\end{align}

Since each $K_n$ is compact, we have ${^*}K_n \subseteq \ns({^*}S)$ for all $n \in \mathbb{N}$. Thus for each $\omega \in E_N$, we have the following inequality for the inner measure with respect to $\mu_{\omega, N}$ (see \eqref{inner and outer} for the definition of the inner measure):

\begin{align*}
    \underline{\mu_{\omega, N}} \left[  \ns({^*}S) \right] \geq L\mu_{\omega, N}({^*}K_n) \text{ for all } n \in \mathbb{N}.
\end{align*}

By taking the limit as $n \rightarrow \infty$ on the right side and using the definition of $E_N$, we obtain:
\begin{align*}
     \underline{\mu_{\omega, N}} \left[  \ns({^*}S) \right] \geq \lim_{n \rightarrow \infty} L\mu_{\omega, N}({^*}K_n) = L\mu_{\omega, N} \left[ \cup_{n \in \mathbb{N}} {^*}K_n \right] = 1 \text{ for all } \omega \in E_N.
\end{align*}
Since $$1 = \underline{\mu_{\omega, N}} \left[  \ns({^*}S) \right] \leq \overline{\mu_{\omega, N}} \left[  \ns({^*}S) \right] \leq 1,$$
it follows that $\ns({^*}S)$ is Loeb measurable, and that $L\mu_{\omega, N}\left[ \ns({^*}S) \right] = 1$ for all $\omega \in E_N$. 
\end{proof}

The specific form of the set $E_N$ obtained in the above proof allows us to use Theorem \ref{Albeverio theorem} to show that for each $N \in {^*}\mathbb{N}$, the measure $L\mu_{\omega, N} \circ \st^{-1}$ is Radon for all $\omega \in E_N$, and that $\mu_{\omega, N}$ is nearstandard in ${^*}\prs$ to this measure. This is proved in the next lemma. 

\begin{lemma}\label{standard part exists for almost all}
Let $S$ be a Hausdorff space. Let $N \in {^*}\mathbb{N}$ and $E_N$ be as in \eqref{almost sure set}. For all $\omega \in E_N$, we have:
\begin{enumerate}[(i)]
    \item\label{T21} $L\mu_{\omega, N} \circ \st^{-1} \in \prs$.
    \item\label{T22} $\mu_{\omega, N} \in \ns({^*}\prs)$, with $\st(\mu_{\omega, N}) = L\mu_{\omega, N} \circ \st^{-1}$. 
\end{enumerate}
\end{lemma}
\begin{proof}
By the proof of Lemma \ref{neastandard for almost all omega}, we know that 
\begin{align}\label{upper monotonicity}
    \lim_{n \rightarrow \infty} L\mu_{\omega, N} \left({^*} K_n \right) = 1 \text{ for all } \omega \in E_N,
\end{align}
where the $K_n$ are compact subsets of $S$.

Therefore, given $\epsilon \in \mathbb{R}_{>0}$, there exists an $n_{\epsilon}$ such that $L\mu_{\omega, N} \left({^*}K_n \right) > 1 - \epsilon$ for all $\omega \in E_N$ and $n \in \mathbb{N}_{>n_{\epsilon}}$. Thus the tightness condition \eqref{nonstandard tightness assumption} holds for $\mu_{\omega, N}$ whenever $\omega \in E_N$. Theorem \ref{Albeverio theorem} now completes the proof. 
\end{proof}

So far, we have used the idea that the expected value of a (random) probability being one implies that the probability is almost surely equal to one. This can be easily turned around and used to show that a certain probability is almost surely zero, by showing that the expected value of that probability is zero. We use this idea to prove next that almost surely, the measures $L\mu_{\omega, N}$ treat the nonstandard extension of a countable disjoint union as if it were the disjoint union of the nonstandard extensions, the leftover portion being assigned zero mass.

\begin{lemma}\label{leftover has zero mass}
Let $S$ be a Hausdorff space and $N \in {^*}\mathbb{N}$. Let $(B_n)_{n \in \mathbb{N}}$ be a sequence of disjoint Borel sets. There is a set $E_{(B_n)_{n \in \mathbb{N}}} \in L({^*}\mathcal{F})$ with $L{^*}\mathbb{P}(E_{(B_n)_{n \in \mathbb{N}}}) = 1$ such that 
\begin{align}\label{leftover mass equation}
    L\mu_{\omega, N} \left[ {^*}\left(\sqcup_{n \in \mathbb{N}} B_n \right) \right] = \sum_{n \in \mathbb{N}} L\mu_{\omega, N}\left({^*}B_n \right) \text{ for all } \omega \in E_{(B_n)_{n \in \mathbb{N}}},
\end{align}
where $\sqcup$ denotes a disjoint union.
\end{lemma}

\begin{remark}
Note that the above lemma does not follow from the disjoint additivity of the measure $L\mu_{\omega, N}$ alone, because $\sqcup_{n \in \mathbb{N}}{^*}B_n \subseteq {^*}\left(\sqcup_{n \in \mathbb{N}} B_n \right)$ with equality if and only if the $B_n$ are empty for all but finitely many $n$. Also, the almost sure set $ E_{(B_n)_{n \in \mathbb{N}}}$ depends on the sequence $(B_n)_{n \in \mathbb{N}}$. Since there are potentially uncountably many such sequences, therefore we cannot expect in the general situation to find a single $L{^*}\mathbb{P}$-almost sure set on which equation \eqref{leftover mass equation} is always valid for all disjoint sequences $(B_n)_{n \in \mathbb{N}}$ of Borel sets. 
\end{remark}

\begin{proof}[Proof of Lemma \ref{leftover has zero mass}]
Let $(B_n)_{n \in \mathbb{N}}$ be a disjoint sequence of Borel sets and let $$B \defeq \sqcup_{n \in \mathbb{N}} B_n.$$
For each $m \in \mathbb{N}$, let $B_{(m)} \defeq \sqcup_{n \in [m]} B_n$. Consider the map $\omega \mapsto \mu_{\omega, N} \left[{^*}\left(B \backslash B_{(m)}\right)\right]$, which is internally Borel measurable by Corollary \ref{transfer corollary 1}.  Since this map is finitely bounded, it is $\mathbf{S}$-integrable with respect to the Loeb measure $L{^*}\mathbb{P}$. In particular, for each $m \in \mathbb{N}$, the $[0,1]$-valued function $L\mu_{\cdot, N}\left[{^*}\left(B \backslash B_{(m)}\right)\right]$, defined by $\omega \mapsto L\mu_{\omega, N}\left[{^*}\left(B \backslash B_{(m)}\right)\right]$, is Loeb measurable. Taking expected values and using $\mathbf{S}$-integrability, we obtain:
\begin{align}
   \mathbb{E}_{L{^*}\mathbb{P}} \left[ L\mu_{\cdot, N} \left[{^*}\left(B \backslash B_{(m)}\right)\right] \right] &\approx {^*}\mathbb{E}_{{^*}\mathbb{P}} \left[\mu_{\cdot, N} \left[{^*}\left(B \backslash B_{(m)}\right)\right] \right] \nonumber \\
   &= {^*}\mathbb{E}_{{^*}\mathbb{P}} \left[ \sum_{i = 1}^N \frac{1}{N} \mathbbm{1}_{{^*}\left(B \backslash B_{(m)}\right)}(X_i) \right] \nonumber \\
   &= \frac{1}{N} \left[ \sum_{i = 1}^N {^*}\mathbb{P}(X_i \in {^*}\left(B \backslash B_{(m)}\right)) \right] \nonumber \\
   &= \frac{1}{N}\left[ N {^*}\mathbb{P}(X_1 \in {^*}\left(B \backslash B_{(m)}\right)) \right] \nonumber \\
   &=  {^*}\mathbb{P}(X_1 \in {^*}\left(B \backslash B_{(m)}\right)) \nonumber \\
   &= \mathbb{P}(X_1 \in B \backslash B_{(m)}) \nonumber \\
   &= \mathbb{P}(X_1 \in B) - \mathbb{P}(X_1 \in B_{(m)}). \label{real number}
\end{align}
Since the expression in \eqref{real number} is a real number, we have the following equality:
\begin{align}\label{equality}
      \mathbb{E}_{L{^*}\mathbb{P}} \left[ L\mu_{\cdot, N} \left[{^*}\left(B \backslash B_{(m)}\right)\right] \right] = \mathbb{P}(X_1 \in B) - \mathbb{P}(X_1 \in B_{(m)}) \text{ for all } m \in \mathbb{N}.
\end{align}
Note that for each $\omega \in {^*}\Omega$, the limit $$\lim_{m \rightarrow \infty} L\mu_{\omega, N} \left[{^*}\left(B \backslash B_{(m)}\right)\right]$$ exists and is equal to $L\mu_{\omega, N}\left[\cap_{m \in \mathbb{N}}{^*}\left(B \backslash B_{(m)}\right)\right]$, because $({^*}(B\backslash B_{(m)}))_{m \in \mathbb{N}}$ is a decreasing sequence of measurable sets. Also, by the upper monotonicity of the measure induced by $X_1$ on $S$, we know that $$\lim_{m \rightarrow \infty} \mathbb{P}(X_1 \in B_{(m)}) = \mathbb{P} \left(X_1 \in \cup_{m \in \mathbb{N}}B_{(m)} \right) = \mathbb{P}(X_1 \in B).$$

Using this in \eqref{equality}, followed by an application of the dominated convergence theorem, we thus obtain the following:
\begin{align}
    0 &= \lim_{m \rightarrow \infty}\mathbb{E}_{L{^*}\mathbb{P}} \left[ L\mu_{\cdot, N} \left[{^*}\left(B \backslash B_{(m)}\right)\right] \right] \nonumber \\
    &= \mathbb{E}_{L{^*}\mathbb{P}} \left[ \lim_{m \rightarrow \infty} L\mu_{\cdot, N} \left[{^*}\left(B \backslash B_{(m)}\right)\right] \right]. \label{DCT}
\end{align}

Also, since $\lim_{m \rightarrow \infty} L\mu_{\omega, N} \left[{^*}\left(B \backslash B_{(m)}\right)\right] \geq 0$, it follows from \eqref{DCT} that there is an $L{^*}\mathbb{P}$-almost sure set
$E_{(B_n)_{n \in \mathbb{N}}}$ such that
\begin{align}\label{almost sure set 2}
\lim_{m \rightarrow \infty} L\mu_{\omega, N} \left[{^*}\left(B \backslash B_{(m)}\right)\right] = 0 \text{ for all } \omega \in E_{(B_n)_{n \in \mathbb{N}}}.    
\end{align}
But for each $\omega \in E_{(B_n)_{n \in \mathbb{N}}}$, we have the following:
\begin{align}
    L\mu_{\omega, N} \left[{^*}\left(B \backslash B_{(m)}\right)\right] &=  L\mu_{\omega, N}({^*}B) -  L\mu_{\omega, N}\left(B_{(m)} \right) \nonumber \\
    &= L\mu_{\omega, N}({^*}B) -  L\mu_{\omega, N}\left(\sqcup_{n \in [m] B_n} \right) \nonumber\\
    &= L\mu_{\omega, N}({^*}B) - \sum_{n \in [m]} L\mu_{\omega, N}({^*}B_m)
    \text{ for all } m \in \mathbb{N}. \label{what happens for each omega}
\end{align}
The proof is completed by letting $m \rightarrow \infty$ in \eqref{what happens for each omega}, followed by an application of \eqref{almost sure set 2}.
\end{proof}

Before the next technical lemma, we need to state some notation. Let $\tprs$ denote the $A$-topology on $\prs$, which is generated by the base $\mathbb{B}(\prs)$ defined in \eqref{Bprs}. For $\mu \in \prs$, let $\tau_\mu$ denote the set of all open neighborhoods of $\mu$ in $\prs$. That is,
\begin{align*}
    \tau_\mu \defeq \{\mathfrak{U} \in \tprs: \mu \in \mathfrak{U}\}. 
\end{align*}
Also, for any open set $\mathfrak{U} \in \tprs$, let $\tau_\mathfrak{U}$ be the subspace topology on $\mathfrak{U}$. In other words, we define
\begin{align*}
    \tau_\mathfrak{U} \defeq \{\mathfrak{V} \in \tprs: \mathfrak{V} = \mathfrak{W} \cap \mathfrak{U} \text{ for some } \mathfrak{W} \in \tprs\} = \{\mathfrak{V} \in \tprs: \mathfrak{V} \subseteq \mathfrak{U}\}. 
\end{align*}

For internal sets $A, B$, we use $\mathfrak{F}(A, B)$ to denote the internal set of all internal functions from $A$ to $B$. 

\begin{lemma}\label{internal definition principle}
Let $S$ be Hausdorff and $N \in {^*}\mathbb{N}$. Let $E_N$ be as defined in \eqref{almost sure set}. For each internal subset $E \subseteq E_N$, there exists an internal function $U_{\cdot} \co E \rightarrow {^*}\bprs$ such that $$\mu_{\omega, N} \in U_{\omega} \subseteq \st^{-1}(L\mu_{\omega, N} \circ \st^{-1}) \text{ for all } \omega \in E.$$ 
\end{lemma}

\begin{proof}
Fix an internal set $E \subseteq E_N$. For each $\mathfrak{U} \in \bprs$, define the following set of internal functions:
\begin{align*}
    \mathcal{G}_\mathfrak{U} \defeq \left\{f \in \mathfrak{F}(E, {^*}\bprs) : f(\omega) \in {^*}\tau_\mathfrak{U} \text{ and }\mu_{\omega, N} \in f(\omega) \text{ for all } \omega \in E \cap {\mu_{\cdot, N}}^{-1}({^*}\mathfrak{U})\right\}
.\end{align*}

Since $E$ is internal and ${\mu_{\cdot, N}}^{-1}({^*}\mathfrak{U})$ is internal by Corollary \ref{transfer corollary 1}, therefore the set $\mathcal{G}_\mathfrak{U}$ is internal for all $\mathfrak{U} \in \bprs$ by the internal definition principle (see, for example, Loeb \cite[Theorem 2.8.4, p. 54]{Loeb-introduction}). Also, $\mathcal{G}_\mathfrak{U}$ is nonempty for each $\mathfrak{U} \in \bprs$. Indeed, if $E \cap {\mu_{\cdot, N}}^{-1}({^*}\mathfrak{U}) = \emptyset$, then $\mathcal{G}_\mathfrak{U} = \mathfrak{F}(E, {^*}\tprs)$; otherwise, if $\omega \in E \cap {\mu_{\cdot, N}}^{-1}({^*}\mathfrak{U})$,  then define $f(\omega) \defeq {^*}\mathfrak{U}$, and define $f$ (internally) arbitrarily on the remainder of $E$---it is clear that the function $f$ so defined is an element of $\mathcal{G}_\mathfrak{U}$. 

Now let $\mathfrak{U}_1, \mathfrak{U}_2$ be two distinct members of $\bprs$. Define a function $f$ on $E$ as follows:
\begin{align*}
    f(\omega) \defeq \left\{
        \begin{array}{ll}
            {^*}\mathfrak{U}_1 \cap {^*}\mathfrak{U}_2 & \text{ if } \quad \omega \in E \cap {\mu_{\cdot, N}}^{-1}({^*}\mathfrak{U}_1) \cap {\mu_{\cdot, N}}^{-1}({^*}\mathfrak{U}_2) \\
            {^*}\mathfrak{U}_1 & \text{ if } \quad \omega \in [E \cap {\mu_{\cdot, N}}^{-1}({^*}\mathfrak{U}_1)] \backslash {\mu_{\cdot, N}}^{-1}({^*}\mathfrak{U}_2) \\
            {^*}\mathfrak{U}_2 & \text{ if } \quad \omega \in [E \cap {\mu_{\cdot, N}}^{-1}({^*}\mathfrak{U}_2)] \backslash {\mu_{\cdot, N}}^{-1}({^*}\mathfrak{U}_1) \\
            {^*}\prs & \text{ if } \quad \omega \in E \backslash \left[ {\mu_{\cdot, N}}^{-1}({^*}\mathfrak{U}_1) \cup {\mu_{\cdot, N}}^{-1}({^*}\mathfrak{U}_2) \right].
        \end{array}
    \right.
\end{align*}

The above function is clearly in $\mathcal{G}_{\mathfrak{U}_1} \cap \mathcal{G}_{\mathfrak{U}_2}$, in view of the fact that $\bprs$ is closed under taking intersections of any two members. In general, to show the finite intersection property of the collection $\{\mathcal{G}_{\mathfrak{U}}: \mathfrak{U} \in \bprs\}$, the same recipe of ``disjointifying'' the union of finitely many basic open sets $\mathfrak{U}_1, \ldots, \mathfrak{U}_k$ works. More precisely, for a subset $\mathfrak{A} \subseteq \prs$, let $\mathfrak{A}^{(0)}$ denote $\mathfrak{A}$ and $\mathfrak{A}^{(1)}$ denote the complement $\prs \backslash \mathfrak{A}$. If $\mathfrak{U}_1, \ldots, \mathfrak{U}_k$ are finitely many members of $\bprs$, then for each $\omega \in E$, define $(i_1(\omega), \ldots, i_k(\omega)) \in \{0,1\}^k \text{ to be the unique tuple such that } \omega \in E \cap \left( \cap_{j \in [k]} {\mu_{\cdot, N}}^{-1}({^*}{\mathfrak{U}_j}^{(i_j(\omega))}) \right)$. Then the function $f$ on $E$ defined as follows is immediately seen to be a member of $\cap_{j \in [k]} \mathcal{G}_{\mathfrak{U}_j}$:
\begin{align*}
    f(\omega) \defeq \bigcap_{\{j \in [k]: i_j(\omega) = 1\}} {^*}{\mathfrak{U}_j} \text{ for all } \omega \in E.
\end{align*}

Thus the collection $\{\mathcal{G}_\mathfrak{U} : \mathfrak{U} \in \bprs\}$ has the finite intersection property. Pick a map $U_{\cdot}$ in the intersection of the $\mathcal{G}_\mathfrak{U}$ (which is nonempty by saturation). It is clear from the definition of the sets $\mathcal{G}_\mathfrak{U}$ that $\mu_{\omega, N} \in U_{\omega}$ for all $\omega \in E$. We now show that $U_{\omega} \subseteq \st^{-1}(L\mu_{\omega, N} \circ \st^{-1})$ for all $\omega \in E$

By Lemma \ref{standard part exists for almost all}, we know that $\mu_{\omega, N} \in \st^{-1}(L\mu_{\omega, N} \circ \st^{-1})$ for all $\omega \in E$. Thus for each $\omega \in E$, we have $\mu_{\omega, N} \in {^*}\mathfrak{U}$ for all $\mathfrak{U} \in \tau_{L\mu_{\omega, N} \circ \st^{-1}}$. Hence, for each $\omega \in E$, we have $\omega \in E \cap {\mu_{\cdot, N}}^{-1}({^*}\mathfrak{U})$ for all $\mathfrak{U} \in \tau_{L\mu_\omega \circ \st^{-1}}$. 

Fix $\omega \in \Omega$ and $\mathfrak{V} \in \tau_{L\mu_{\omega, N} \circ \st^{-1}}$. Then there exists a basic open set $\mathfrak{U} \in \tau_{L\mu_\omega \circ \st^{-1}} \cap \bprs$, such that $L\mu_{\omega, N} \circ \st^{-1} \in \mathfrak{U} \subseteq \mathfrak{V}$. In particular, this also implies that $\omega \in E \cap {\mu_{\cdot, N}}^{-1}({^*}\mathfrak{U})$. Since $U_{\cdot} \in \mathcal{G}_{\mathfrak{U}}$, we thus obtain the following from the definition of $\mathcal{G}_{\mathfrak{U}}$:
\begin{align}
    \mu_{\omega, N} \in U_{\omega} \in {^*}\tau_{\mathfrak{U}} \subseteq {^*}\tau_{\mathfrak{V}}.
\end{align}

Since $\omega \in E$ was arbitrarily chosen in the previous paragraph, this implies that $U_{\omega} \in {^*}\tau_\mathfrak{U}$ for all $\omega \in E$ and $\mathfrak{U} \in \tau_{L\mu_{\omega, N} \circ \st^{-1}}$. In particular, we have: $$U_{\omega} \subseteq \cap_{\mathfrak{U} \in \tau_{L\mu_{\omega, N} \circ \st^{-1}}} {^*}\mathfrak{U} = \st^{-1}(L\mu_{\omega, N} \circ \st^{-1}) \text{ for all } \omega \in E,$$ 
as desired.
\end{proof}

For each $N \in {^*}\mathbb{N}$, since $E_N$ is a Loeb measurable set of measure $1$, there exists an increasing sequence $(E_{N, n})_{n \in \mathbb{N}}$ of internal subsets of $E_N$ such that the following holds:

\begin{align}\label{F_n}
    {^*}\mathbb{P}(E_{N, n}) > 1 - \frac{1}{n} \text{ for all } n \in \mathbb{N}.
\end{align}

Lemma \ref{internal definition principle} applied to the internal sets $E_{N, n}$ will imply that the pushforward internal measure on ${^*}\prs$ induced by the internal random variable $\mu_{\cdot, N}$ is such that its Loeb measure assigns full measure to $\ns({^*}\prs)$. This will be the content of our next result. 

More precisely, for each $N \in {^*}\mathbb{N}$, define an internal finitely additive probability $P_N$ on $({^*}\prs, {^*}\sigmabprs)$ as follows: 
\begin{align}\label{definition of P}
P_N(\mathfrak{B}) \defeq {^*}\mathbb{P}\left(\{\omega \in {^*}\Omega : \mu_{\omega, N} \in \mathfrak{B}\}\right) = {^*}\mathbb{P}\left({\mu_{\cdot, N}}^{-1}(\mathfrak{B}) \right) \text{ for all } \mathfrak{B} \in {^*}{\sigmabprs}.
\end{align}
That this is indeed an internal probability follows from Corollary \ref{transfer corollary 1}. As promised, we now show that the corresponding Loeb measure $LP_N$ is concentrated on nearstandard elements of ${^*}\prs$.

\begin{theorem}\label{P has everything nearstandard}
Let $S$ be a Hausdorff space. Let $N \in {^*}\mathbb{N}$ and let $P_N$ be as in \eqref{definition of P}. Let $$({^*}\prs, {L_{P_N}}({^*}{\sigmabprs}), LP_N)$$ be the associated Loeb space. Then the set $\ns({^*}\prs)$ is Loeb measurable, with $$LP_N(\ns({^*}\prs)) = 1.$$
\end{theorem}

\begin{proof}
Let $E_N$ be as in \eqref{almost sure set} and let $(E_{N, n})_{n \in \mathbb{N}} \subseteq E_N$ be as in \eqref{F_n}. Fix $n \in \mathbb{N}$. With $E \defeq E_{N, n}$, apply Lemma \ref{internal definition principle} to obtain an internal function $U_{\cdot} \co E_{N,n} \rightarrow {^*}\bprs$ such that $$\mu_{\omega, N} \in U_{\omega} \text{ and } U_\omega \subseteq \st^{-1}(L\mu_{\omega, N} \circ \st^{-1}) \text{ for all } \omega \in E_{N, n}.$$ 

In particular, $U_{\omega} \subseteq \ns({^*}\prs)$ for all $\omega \in E_{N, n}$, so that $U \defeq \cup_{\omega \in E_{N,n}} U_{\omega} \subseteq \ns({^*}\prs)$. Since $\mu_{\omega, N} \in U_{\omega}$ for all $\omega \in E_{N, n}$, we have $E_{N, n} \subseteq {\mu_{\cdot, N}}^{-1}(U)$. Hence it follows from \eqref{definition of P} that  
\begin{align*}
    \underline{P_N}(\ns({^*}\prs)) &\geq \underline{P_N}(U) \\
    &= \sup\{\st(P_N(\mathfrak{B})): \mathfrak{B} \in {^*}{\sigmabprs} \text{ and } \mathfrak{B} \subseteq U\} \\
    &= \sup\{\st({^*}\mathbb{P}({\mu_{\cdot, N}}^{-1}(\mathfrak{B}))): \mathfrak{B} \in {^*}{\sigmabprs} \text{ and } \mathfrak{B} \subseteq U\} \\
    &\geq \underline{{^*}\mathbb{P}}\left({\mu_{\cdot, N}}^{-1}(U) \right) \\
    &\geq L{^*}\mathbb{P}(E_{N, n}).
\end{align*}
Using \eqref{F_n} and observing that $n \in \mathbb{N}$ was arbitrary, we thus obtain the following:
\begin{align*}
    \underline{P_N}(\ns({^*}\prs)) \geq 1 - \frac{1}{n} \text{ for all } n \in \mathbb{N}.
\end{align*}
This clearly implies that
$$1 =  \underline{P_N}(\ns({^*}\prs)) \leq  \overline{P_N}(\ns({^*}\prs)) \leq 1,$$
so that $\underline{P_N}(\ns({^*}\prs)) =  \overline{P_N}(\ns({^*}\prs)) = 1$. As a consequence, $\ns({^*}\prs)$ is Loeb measurable with $LP_N(\ns({^*}\prs)) = 1$, completing the proof. 
\end{proof}

The next lemma provides a useful dictionary between Loeb integrals with respect to $LP_N$ and those with respect to $L{^*}\mathbb{P}$: 

\begin{lemma}\label{bounded function transfer result}
Let $S$ be a Hausdorff space and $N \in {^*}\mathbb{N}$. Let $P_N$ be as in \eqref{definition of P}. For any bounded $LP_N$-measurable function $f \co {^*}\prs \rightarrow \mathbb{R}$, we have:
\begin{align}\label{bounded function of mu}
   \int_{{^*}\prs} f(\mu) dLP_N(\mu) = \int_{{^*}\Omega} f(\mu_{\omega, N}) dL{^*}\mathbb{P}(\omega). 
\end{align}
\end{lemma}

\begin{proof}
First fix an internally measurable set $\mathfrak{B} \in {^*}{\sigmabprs}$ and let $f = \mathbbm{1}_{\mathfrak{B}}$. Then the left side of \eqref{bounded function of mu} is equal to $LP_N(\mathfrak{B}) = \st(P_N(\mathfrak{B}))$, which also equals the following by \eqref{definition of P}:  $$\st \left[{^*}\mathbb{P}\left({\mu_{\cdot, N}}^{-1}(\mathfrak{B})\right) \right] = L{^*}\mathbb{P} \left[\{\omega \in {^*}\Omega: \mu_{\omega, N} \in \mathfrak{B} \} \right] = \int_{{^*}\Omega} \mathbbm{1}_{\mathfrak{B}}(\mu_{\omega, N}) dL{^*}\mathbb{P}(\omega).$$ 

Thus \eqref{bounded function of mu} is true when $f$ is the indicator function of an internally measurable subset of ${^*}\prs$. That is:
\begin{align}\label{LP(B)}
    LP_N(\mathfrak{B}) = L{^*}\mathbb{P}\left({\mu_{\cdot, N}}^{-1}(\mathfrak{B})\right) \text{ for all } \mathfrak{B} \in {^*}{\sigmabprs}. 
\end{align}

Now, let $\mathfrak{A} \in L_{P_N}({^*}{\sigmabprs})$ be a Loeb measurable set and let $f = \mathbbm{1}_{\mathfrak{A}}$. By the fact that the Loeb measure of a Loeb measurable set equals its inner and outer measure with respect to the internal algebra ${^*}{\sigmabprs}$, we obtain sets $\mathfrak{A}_{\epsilon}, \mathfrak{A}^{\epsilon} \in {^*}{\sigmabprs}$ for each $\epsilon \in \mathbb{R}_{>0}$, such that $\mathfrak{A}_{\epsilon} \subseteq \mathfrak{A} \subseteq \mathfrak{A}^{\epsilon}$ and such that the following holds:
\begin{align}\label{B_epsilon}
    LP_N(\mathfrak{A}) - \epsilon < LP_N(\mathfrak{A}_{\epsilon}) \leq LP_N(\mathfrak{A}) \leq LP_N(\mathfrak{A}^{\epsilon}) < LP_N(\mathfrak{A}) + \epsilon.    
\end{align}
Using \eqref{LP(B)} in \eqref{B_epsilon} yields the following for each $\epsilon \in \mathbb{R}_{>0}$:
\begin{align}\label{sandwich}
    LP_N(\mathfrak{A}) - \epsilon <  L{^*}\mathbb{P}\left({\mu_{\cdot, N}}^{-1}(\mathfrak{A}_{\epsilon})\right) \leq LP_N(\mathfrak{A}) \leq  L{^*}\mathbb{P}\left({\mu_{\cdot, N}}^{-1}(\mathfrak{A}^{\epsilon})\right) < LP_N(\mathfrak{A}) + \epsilon.  
\end{align}

Since ${\mu_{\cdot, N}}^{-1}(\mathfrak{A}_{\epsilon}), {\mu_{\cdot, N}}^{-1}(\mathfrak{A}^{\epsilon})$ are members of ${^*}\mathcal{F}$ by Lemma \ref{measurable}, it follows from \eqref{sandwich} that for any $\epsilon \in \mathbb{R}_{>0}$ we have:
\begin{align*}
    LP_N(\mathfrak{A}) - \epsilon &\leq \sup\{L{^*}\mathbb{P}(E): E \in {^*}\mathcal{F} \text{ and } E \subseteq {\mu_{\cdot, N}}^{-1}(\mathfrak{A_{\epsilon}})\} \\
    &\leq \sup\{L{^*}\mathbb{P}(E): E \in {^*}\mathcal{F} \text{ and } E \subseteq {\mu_{\cdot, N}}^{-1}(\mathfrak{A})\}\\
    &= \underline{{^*}\mathbb{P}}\left({\mu_{\cdot, N}}^{-1}(\mathfrak{A})\right), 
\end{align*}
and
\begin{align*}
LP_N(\mathfrak{A}) + \epsilon &\geq \inf\{L{^*}\mathbb{P}(E): E \in {^*}\mathcal{F} \text{ and } {\mu_{\cdot, N}}^{-1}(\mathfrak{A}^{\epsilon}) \subseteq E\} \\
     &\geq \inf\{L{^*}\mathbb{P}(E): E \in {^*}\mathcal{F} \text{ and } {\mu_{\cdot, N}}^{-1}(\mathfrak{A}) \subseteq E\} \\
     &= \overline{{^*}\mathbb{P}}\left({\mu_{\cdot, N}}^{-1}(\mathfrak{A})\right).
\end{align*}

Since $\epsilon \in \mathbb{R}_{>0}$ is arbitrary, it thus follows that $\underline{{^*}\mathbb{P}}\left({\mu_{\cdot, N}}^{-1}(\mathfrak{A})\right) = \overline{{^*}\mathbb{P}}\left({\mu_{\cdot, N}}^{-1}(\mathfrak{A})\right)$, both being equal to $ LP_N(\mathfrak{A})$. This shows that ${\mu_{\cdot, N}}^{-1}(\mathfrak{A})$ is Loeb measurable and that the following holds: 
\begin{align}\label{proving Loeb measurability}
    LP_N(\mathfrak{\mathfrak{A}}) = L{^*}\mathbb{P}\left[{\mu_{\cdot, N}}^{-1}(\mathfrak{A}) \right] \text{ for all } \mathfrak{A} \in {L_{P_N}}({^*}{\sigmabprs}).
\end{align}

This proves \eqref{bounded function of mu} for indicator functions of Loeb measurable sets. Since the functions $f$ satisfying \eqref{bounded function of mu} are clearly closed under taking $\mathbb{R}$-linear combinations, the result is true for simple functions (that is, those Loeb measurable functions that take finitely many values). The result for general bounded Loeb measurable functions follows from this (and the dominated convergence theorem) since any bounded measurable function can be uniformly approximated by a sequence of simple functions. 
\end{proof}

The result in \eqref{proving Loeb measurability} is interesting and useful in its own right. We record this observation as a corollary of the above proof.

\begin{corollary}\label{corollary}
Let $S$ be a Hausdorff space and let $N \in {^*}\mathbb{N}$. Let $P_N$ be as in \eqref{definition of P}. For any $\mathfrak{A} \in {L_{P_N}}({^*}{\sigmabprs})$, the set ${\mu_{\cdot, N}}^{-1}(\mathfrak{A})$  is $L{^*}\mathbb{P}$-measurable. Furthermore, we have:
\begin{align*}
    LP_N(\mathfrak{A}) = L{^*}\mathbb{P}\left[{\mu_{\cdot, N}}^{-1}(\mathfrak{A}) \right] \text{ for all } \mathfrak{A} \in {L_{P_N}}({^*}\sigmabprs).
\end{align*}
\end{corollary}

\subsection{An internal measure induced on the space of all internal Radon probability measures}\label{3.2}
Equipped with a way to compute the $LP_N$ measure of a large collection of sets, we are now in a position to use Prokhorov's theorem (Theorem \ref{Radon Prokhorov}) to verify that $P_N$ satisfies the conditions from Theorem \ref{Albeverio theorem}. Recall the base $\bprs$ for the $A$-topology on $\prs$ (see \ref{Bprs}) and observe that the set $\finunbprs$ of all finite unions of members of $\bprs$ is also a base for the $A$-topology on $\prs$, while this base is closed under finite unions by construction, and, furthermore, the nonstandard extension of any member of this base belongs to ${^*}{\sigmabprs}$ by construction as well. Thus only the tightness condition from Theorem \ref{Albeverio theorem} in this context remains to be verified, which we do in the next theorem.

\begin{theorem}\label{Prokhorov for P_N}
Let $S$ be a Hausdorff space and let $N \in {^*}\mathbb{N}$. Let $P_N$ be as in \eqref{definition of P} and let $({^*}\prs, {L_{P_N}}({^*}{\sigmabprs}), LP_N)$ be the associated Loeb space.  

Given $\epsilon \in \mathbb{R}_{>0}$, there exists a compact set $\mathfrak{K}_{(\epsilon)} \subseteq \prs$ satisfying the following:
$$LP_N({^*}\mathfrak{U}) \geq 1 - \epsilon \text{ for all } \mathfrak{U} \in \finunbprs \text{ such that }  \mathfrak{K}_{(\epsilon)} \subseteq \mathfrak{U}.$$
\end{theorem}

\begin{proof}\footnote{\label{main footnote}An error in the previously announced proof of this result was identified and corrected in the 2024 preprint \cite{AlamPLMS}. This correction, which is presented here, required a slight modification to the definition of the compact sets $(K_n)_{n \in \mathbb{N}}$ defined at the beginning of this appendix --- compare with their definition in Section 3 of the original preprint \cite{Alam-deFinetti-Hewitt-Savage}. 

The statement of the theorem being presented in the current revision has also undergone a slight change (compare with Theorem 3.11 in the earlier iterations of this manuscript and Theorem B.11 in \cite{AlamPLMS}), in view of the fact that $P_N$ can be thought of as an internal measure only on $({^*}\prs, {^*}{\sigmabprs})$, but not necessarily on $({^*}\prs, {^*}\mathcal{B}(\prs))$ as was incorrectly claimed in the previous iterations of this work. The new statement is intended to make use of the full strength of the result on pushing down Loeb measures from Albeverio et al. (see Theorem \ref{Albeverio theorem} and Footnote \ref{Albeverio footnote}).

The relevant amendment to take care of this technical obstacle in the present revision occurred in the statement of Lemma \ref{measurable} with ramifications on all subsequent results, including this one, that depend on it.}
Let $(K_n)_{n \in \mathbb{N}}$ be the increasing sequence of compact subsets of $S$ fixed in \eqref{K_n}, and let $\epsilon \in \mathbb{R}_{>0}$ be fixed. 

For each $n \in \mathbb{N}$, let $\alpha_n = {^*}\mbp \left(\left\{\omega \in {^*}\Omega: \mu_{\omega, N}({^*}K_n) \geq 1 - \frac{1}{n} \right\}\right)$. Then, by the linearity of internal expectation and the fact that the internal random variable $\mu_{\cdot, N}$ is bounded above by $1$, we have the following inequality for each $n \in \mathbb{N}$:
\begin{gather}
    1 - \frac{1}{n^3} < {^*}\mathbb{E}\left(\mu_{\cdot, N}({^*}K_n)\right) \leq 1\cdot \alpha_n + \left(1 - \frac{1}{n} \right)\cdot(1 - \alpha_n), \nonumber
\end{gather}
from which we obtain the following after simplification:
\begin{gather}
    \alpha_n = {^*}\mbp \left(\left\{\omega \in {^*}\Omega: \mu_{\omega, N}({^*}K_n) \geq 1 - \frac{1}{n} \right\}\right) \geq 1 - \frac{1}{n^2}. \label{alpha_n inequality}
\end{gather}

For each $n \in \mathbb{N}$, let us define the following \textit{closed} set\footnote{The sets $\mathfrak{F}_n$ defined here are closed, since their complements are open (see, for instance, Lemma \ref{characterization of weak topology 2}).} : 
\begin{align}\label{definition of F_n}
    \mathfrak{F}_n \defeq \left\{\gamma \in \prs : \gamma(K_n) \geq 1 - \frac{1}{n}\right\}. 
\end{align}
Given $\epsilon \in \mathbb{R}_{>0}$ that was fixed in the beginning, there exists $n_{\epsilon} \in \mathbb{N}$ such that 
\begin{align}\label{definition of n_epsilon}
    \sum_{n \in \mathbb{N}_{\geq n_\epsilon}} \frac{1}{n^2} < \epsilon. 
\end{align}
We now define $\mathfrak{K}_{(\epsilon)}$ as follows:
\begin{align}\label{K_epsilon}
    \mathfrak{K}_{(\epsilon)} = \bigcap_{n \in \mathbb{N}_{\geq n_{\epsilon}}} \mathfrak{F}_n
\end{align}
Being an intersection of closed sets, the set $\mathfrak{K}_{(\epsilon)}$ is closed. It is also relatively compact by Theorem \ref{Radon Prokhorov}. Being a closed set that is relatively compact, it follows that $\mathfrak{K}_{(\epsilon)}$ is a compact subset of $\prs$. 

Now let $\mathfrak{U} \in \finunbprs$ be any basic open subset of $\prs$ containing $\mathfrak{K}_{(\epsilon)}$. The proof will be completed once we show that $LP_N({^*}\mathfrak{U}) > 1 - \epsilon$. Toward that end, we first make the following immediate observation using Lemma \ref{compactness lemma}:
\begin{align}\label{K_epsilon observation}
{^*}\mathfrak{K}_{(\epsilon)} \subseteq \left[ \left(\bigcap_{n \in \mathbb{N}_{\geq n_{\epsilon}}}{^*}\mathfrak{F}_n\right) \cap \ns({^*}\prs) \right] \subseteq {^*}\mathfrak{U}.
\end{align}

Note that $\left(\bigcap_{n \in \mathbb{N}_{\geq n_{\epsilon}}}{^*}\mathfrak{F}_n\right) \in L_{P_N}({^*}{\sigmabprs})$, since ${^*}\mathfrak{F}_n \in {^*}{\sigmabprs}$ for all $n \in \mathbb{N}$---this follows from the fact that each $\mathfrak{F}_n$ is the complement of an element of $\bprs$ by construction. By \eqref{K_epsilon observation} and Theorem \ref{P has everything nearstandard}, we thus obtain:
\begin{align}
   LP_N({^*}\mathfrak{U}) \geq LP_N \left[ \left(\bigcap_{n \in \mathbb{N}_{\geq n_{\epsilon}}}{^*}\mathfrak{F}_n\right) \cap \ns({^*}\prs) \right] = LP_N \left(\bigcap_{n \in \mathbb{N}_{\geq n_{\epsilon}}}{^*}\mathfrak{F}_n\right). \label{pre-union bound}
\end{align}

The proof is now completed by the following elementary calculation which shows (after taking complements) that the right side of \eqref{pre-union bound} is indeed greater than or equal to $(1 - \epsilon)$:
\begin{align}
     LP_N \left({^*}\prs \backslash \bigcap_{n \in \mathbb{N}_{\geq n_{\epsilon}}}{^*}\mathfrak{F}_n\right) &=  LP_N \left(\bigcup_{n \in \mathbb{N}_{\geq n_{\epsilon}}}{^*}\left(\prs \backslash \mathfrak{F}_n\right)\right) \nonumber \\
     &\leq \sum_{n \in \mathbb{N}_{\geq n_{\epsilon}}} LP_N\left( {^*}\left(\prs \backslash \mathfrak{F}_n\right)\right) \nonumber \\
     &= \sum_{n \in \mathbb{N}_{\geq n_{\epsilon}}} L{^*}\mathbb{P}\left[{\mu_{\cdot, N}}^{-1}({^*}\left(\prs \backslash \mathfrak{F}_n\right)) \right] \label{fourth to last step} \\
     &= \sum_{n \in \mathbb{N}_{\geq n_{\epsilon}}} L{^*}\mathbb{P}\left[\left\{\omega \in {^*}\Omega: \mu_{\omega, N}({^*}K_n) < 1 - \frac{1}{n}\right\} \right] \label{third to last step} \\
     &\leq \sum_{n \in \mathbb{N}_{\geq n_\epsilon}} \frac{1}{n^2} \label{second to last step} \\
     &< \epsilon, \label{last step}
\end{align}
where \eqref{fourth to last step} follows from Corollary \ref{corollary}, \eqref{third to last step} follows from the definition \eqref{definition of F_n} of $\mathfrak{F}_n$, while the last two steps \eqref{second to last step} and \eqref{last step} follow respectively from \eqref{alpha_n inequality} and \eqref{definition of n_epsilon}. This completes the proof. 
\end{proof}

Theorems \ref{Prokhorov for P_N} and \ref{Albeverio theorem} now immediately lead to the following result. 

\begin{theorem}\label{second order nearstandard}\footnote{In the previous iterations of this manuscript, it was incorrectly claimed that $P_N$ is nearstandard to $LP_N \circ \st^{-1}$ in ${^*}\mathfrak{P}(\prs)$. This is not true since we cannot, in the general situation, even view $P_N$ as an element of ${^*}\mathfrak{P}(\prs)$---see also Footnote \ref{main footnote}.}
Suppose that $S$ is a Hausdorff space. Let $N \in {^*}\mathbb{N}$ and let $P_N$ be as in \eqref{definition of P}. Let $({^*}\prs, {L_{P_N}}({^*}{\sigmabprs}), LP_N)$ be the associated Loeb space. Then $LP_N \circ \st^{-1}$ is a Radon measure on the Hausdorff space $\prs$. 
\end{theorem}

Note that since $LP_N \circ \st^{-1}$ is a measure on $\mathcal{B}(\prs)$, it is (in particular) the case that $\st^{-1}(\mathfrak{B})$ is $LP_N$-measurable for all $\mathfrak{B} \in \mathcal{B}(\prs)$. This observation is useful enough that we record it as a corollary. 

\begin{corollary}\label{measurability corollary}
Let $S$ be a Hausdorff space and let $P_N$ be as in \eqref{definition of P}. For each $\mathfrak{B} \in \mathcal{B}(\prs)$, we have $\st^{-1}(\mathfrak{B}) \in {L_{P_N}}({^*}{\sigmabprs})$. 
\end{corollary}

\subsection{Almost sure standard parts of hyperfinite empirical measures}\label{3.3}
We now return to studying properties of the measures $L\mu_{\omega, N}$ for $N \in {^*}\mathbb{N}$. Corollary \ref{measurability corollary} immediately leads us to the following. 

\begin{lemma}\label{measurability lemma}
Let $S$ be a Hausdorff space. Let $N \in {^*}\mathbb{N}$ and let $E_N$ be the $L{^*}\mathbb{P}$-almost sure set fixed in \eqref{almost sure set}. Then for each $B \in \mathcal{B}(S)$, the set $\st^{-1}(B)$ is $L\mu_{\omega, N}$-measurable for all $\omega \in E_N$. Furthermore, for each $B \in \mathcal{B}(S)$, the function $\omega \mapsto L\mu_{\omega, N}(\st^{-1}(B))$ thus defines a $[0,1]$-valued random variable almost everywhere on $({^*}\Omega, L({^*}\mathcal{F}), L{^*}\mathbb{P})$. 
\end{lemma}

\begin{proof}
It was proved as part of Lemma \ref{standard part exists for almost all} that for each $B \in \mathcal{B}(S)$, the set $\st^{-1}(B)$ is $L\mu_{\omega, N}$-measurable for all $\omega \in E_N$. Thus, the function $\omega \mapsto L\mu_{\omega, N}(\st^{-1}(B))$ is defined $L^{*}\mathbb{P}$-almost surely on ${^*}\Omega$ for all $B \in \mathcal{B}(S)$. 

Now fix $B \in \mathcal{B}(S)$. Since $L{^*}\mathbb{P}(E_N) = 1$ and $({^*}\Omega, L({^*}\mathcal{F}), L{^*}\mathbb{P})$ is a complete probability space, showing that the map $\omega \mapsto L\mu_{\omega, N}(\st^{-1}(B))$ is Loeb measurable is equivalent to showing that for any $\alpha \in \mathbb{R}$, the set $\{\omega \in E_N: L\mu_{\omega, N}\left[ \st^{-1}(B) \right] > \alpha \}$ is Loeb measurable. Toward that end, fix $\alpha \in \mathbb{R}$. Note that by Lemma \ref{standard part exists for almost all}, we obtain the following:
 \begin{align*}
    \{\omega \in E_N: L\mu_{\omega, N}\left[ \st^{-1}(B) \right] > \alpha \} &= \{\omega \in E_N: \left[\st(\mu_{\omega, N})\right](B) > \alpha \} \\
    &= E_N \cap \left[{\mu_{\cdot, N}}^{-1} \left(\st^{-1}\left(\{\nu \in \prs: \nu(B) > \alpha\} \right) \right) \right].
 \end{align*}

By Lemma \ref{evaluation of Borel sets'} and Corollary \ref{measurability corollary}, we have that $\st^{-1}\left(\{\nu \in \prs: \nu(B) > \alpha\} \right) \in {L_{P_N}}({^*}{\sigmabprs})$, which completes the proof by Corollary \ref{corollary}. 
\end{proof}

The next two lemmas are preparatory for Theorem \ref{standard part commutativity}, which will show that for each Borel set $B \in \mathcal{B}(S)$, the $L\mu_{\omega, N}$ measures of $\st^{-1}(B)$ and ${^*}B$ are almost surely equal to each other.

\begin{lemma}\label{compact set standard inverse}
Let $S$ be a Hausdorff space and let $N \in {^*}\mathbb{N}$. Let $K$ be a compact subset of $S$. Then,
\begin{align*}
    L\mu_{\omega, N}(\st^{-1}(K)) = L\mu_{\omega, N}({^*}K) \text{ for } L{^*}\mathbb{P} \text{-almost all } \omega \in {^*}\Omega.
\end{align*}
\end{lemma}
\begin{proof}
Let $K \subseteq S$ be a compact set. Let $E_N \subseteq {^*}\Omega$ be as in \eqref{almost sure set}. By Lemma \ref{standard part exists for almost all}, we know that $\st^{-1}(K)$ is $L\mu_{\omega, N}$-measurable for all $\omega \in E_N$. Since $K$ is compact, we also have ${^*}K \subseteq \st^{-1}(K)$. It is thus clear from the definition of standard parts that the following holds: 
\begin{align}\label{definition of standard part}
    \st^{-1}(K) \backslash {^*}K \subseteq {^*}O \backslash {^*}K = {^*}(O \backslash K) \text{ for all open sets } O \text{ such that } K \subseteq O.
\end{align}

Using Lemma \ref{measurability lemma} and Corollary \ref{transfer corollary 1} respectively, we know that the maps $\omega \mapsto L\mu_{\omega, N} \left[\st^{-1}(K) \backslash {^*}K \right]$ and $\omega \mapsto L\mu_{\omega, N}({^*}O\backslash {^*}K)$ are $L{^*}\mathbb{P}$-measurable for all open sets $O$ containing $K$. Taking expected values and using \eqref{definition of standard part}, we obtain the following for any open set $O$ containing $K$:
\begin{align}\label{O minus C}
    \mathbb{E}_{L{^*}\mathbb{P}} \left[ L\mu_{\cdot, N} \left(\st^{-1}(K) \backslash {^*}K \right) \right] \leq \mathbb{E}_{L{^*}\mathbb{P}} \left[ L\mu_{\cdot, N} \left({^*}O \backslash {^*}K \right) \right].
\end{align}

But, by $\mathbf{S}$-integrability of the map $\omega \rightarrow \mu_{\omega, N}({^*}O \backslash {^*}K)$, we also obtain the following: 
\begin{align*}
    \mathbb{E}_{L{^*}\mathbb{P}} \left[ L\mu_{\cdot, N} \left({^*}O \backslash {^*}K \right) \right] &\approx {^*}\mathbb{E}(\mu_{\cdot}({^*}O \backslash {^*}K)) \\
    &= \frac{1}{N} \sum_{i \in [N]} \mathbb{P}[X_i \in O \backslash K] \\
    &= \mathbb{P}[X_1 \in O \backslash K].
\end{align*}

Using this in \eqref{O minus C}, taking infimum as $O$ varies over open sets containing $K$, and using the fact that the distribution of $X_1$ is outer regular on compact sets, we obtain the following:
\begin{align}
     \mathbb{E}_{L{^*}\mathbb{P}} \left[ L\mu_{\cdot, N} \left(\st^{-1}(K) \backslash {^*}K \right) \right] = 0. \label{used outer regularity}
\end{align}

As a result, there exists a Loeb measurable set $E_{K, N} \in L({^*}\mathcal{F})$ such that
$$L\mu_{\omega, N} \left(\st^{-1}(K) \backslash {^*}K \right) = 0 \text{ for all } \omega \in E_{K, N},$$
completing the proof.
\end{proof}

\begin{remark}\label{tight and outer regular on compact remark}
So far, we have only used the facts that the common distribution of the random variables $X_1, X_2, \ldots$ is tight and that it is outer regular on compact subsets of $S$. Tightness was used in \eqref{K_n} and all subsequent results that depended on it, while outer regularity on compact subsets was used to obtain \eqref{used outer regularity}, and will be used going forward. When this work was initially posted on the arXiv as the preprint \cite{Alam-deFinetti-Hewitt-Savage} in 2020, it was mentioned that our assumption of Radonness of the common distribution was only for ease of presentation as we only had ``occasion to use the fact that this distribution is tight and outer regular on compact subsets of $S$''. However, as recently observed by Potaptchik, Roy, and Schrittesser \cite{PotaptchikRoySchrittesser}, we are not losing any generality by only assuming Radnonness. Indeed, \cite[Proposition 4.1]{PotaptchikRoySchrittesser} proves that a Borel probability measure on a Hausdorff space is Radon if and only if it is both tight and outer regular on compact subsets.
\end{remark}

We can strengthen Lemma \ref{compact set standard inverse} to now be applicable for all closed sets, as we show next. 

\begin{lemma}\label{closed subset and standard inverse}
Let $S$ be a Hausdorff space and let $N \in {^*}\mathbb{N}$. Let $F$ be a closed subset of $S$. Then we have the following:
\begin{align}\label{closed set lemma}
    L\mu_{\omega, N}(\st^{-1}(F)) = L\mu_{\omega, N}({^*}F) \text{ for } L{^*}\mathbb{P} \text{-almost all } \omega \in {^*}\Omega.
\end{align}
\end{lemma}
\begin{proof}
Let $(K_n)_{n \in \mathbb{N}}$ be the increasing sequence of compact subsets of $S$ fixed in \eqref{K_n}, and let $E_N$ be as in \eqref{almost sure set}. Thus, we have:
\begin{align*}
    L\mu_{\omega, N}\left(\cup_{n \in \mathbb{N}} {^*}K_n \right) = 1 \text{ for all } \omega \in E_N.
\end{align*}
Using the upper monotonicity of $L\mu_{\omega, N}$, we rewrite the above as follows:
\begin{align}\label{limit statement}
   \lim_{n \rightarrow \infty} L\mu_{\omega, N}\left({^*}K_n \right) = 1 \text{ for all } \omega \in E_N.
\end{align}

Let $F \subseteq S$ be closed. Since $F \cap K_n$ is compact for all $n \in \mathbb{N}$, by Lemma \ref{compact set standard inverse}, there exist $L{^*}\mathbb{P}$-almost sure sets $(E^{(n)})_{n \in \mathbb{N}}$ such that the following holds:
\begin{align}\label{K_n statement}
    L\mu_{\omega, N}(\st^{-1}(F \cap K_n)) = L\mu_{\omega, N}({^*}F \cap {^*}K_n) \text{ for all } \omega \in E^{(n)}, \text{ where } n \in \mathbb{N}.
\end{align}

Let $E_F \defeq E_N \cap \left(\cap_{n \in \mathbb{N}}E^{(n)} \right)$. Being a countable intersection of almost sure sets, $E_F$ is also $L{^*}\mathbb{P}$-almost sure. Letting $\omega \in E_F$ and taking limits as $n \rightarrow \infty$ on both sides of \eqref{K_n statement}, we obtain the following in view of \eqref{limit statement}:
\begin{align}\label{limit statement 2}
    \lim_{n \rightarrow \infty}  L\mu_{\omega, N}(\st^{-1}(F \cap K_n)) = L\mu_{\omega, N}({^*}F) \text{ for all } \omega \in E_F.
\end{align}

Using the upper monotonicity of the measure $L\mu_{\omega, N}$ on the left side of \eqref{limit statement 2}, we obtain the following:
\begin{align}\label{limit statement 3}
     L\mu_{\omega, N}\left(\cup_{n \in \mathbb{N}}\st^{-1}(F \cap K_n)\right) = L\mu_{\omega, N}({^*}F) \text{ for all } \omega \in E_F.
\end{align}
But, we also have the following: 
\begin{align*}
    \cup_{n \in \mathbb{N}}\st^{-1}(F \cap K_n) &= \st^{-1} \left(\cup_{n \in \mathbb{N}} (F \cap K_n) \right) \\
    &= \st^{-1}\left(F \cap (\cup_{n \in \mathbb{N}}K_n) \right),
\end{align*}
so that
\begin{align*}
    \st^{-1}(F) \backslash \cup_{n \in \mathbb{N}}\st^{-1}(F \cap K_n) &= \st^{-1}(F) \backslash \st^{-1}\left(F \cap (\cup_{n \in \mathbb{N}}K_n) \right) \\
    &= \st^{-1}\left(F \cap \left(\cap_{n \in \mathbb{N}} S \backslash K_n \right) \right) \\
    &\subseteq \cap_{n \in \mathbb{N}} \st^{-1}(S \backslash K_n) \\
    &= \cap_{n \in \mathbb{N}} \left[\st^{-1}(S) \backslash \st^{-1}(K_n) \right].
\end{align*}
Thus, for any $\omega \in E_F$, the following holds:
\begin{align}
     L\mu_{\omega, N}\left[\st^{-1}(F) \backslash \cup_{n \in \mathbb{N}}\st^{-1}(F \cap K_n)\right] &\leq \lim_{n \rightarrow \infty}  L\mu_{\omega, N}\left[\st^{-1}(S) \backslash \st^{-1}(K_n)) \right] \nonumber \\
     &= \lim_{n \rightarrow \infty} \left[ L\mu_{\omega, N}(\ns({^*}S)) - L\mu_{\omega, N}(\st^{-1}(K_n)) \right] \nonumber \\
     &= \lim_{n \rightarrow \infty} \left[ 1 - L\mu_{\omega, N}({^*}K_n) \right], \label{limit statement 4}
\end{align}
where the last line follows from Lemma \ref{compact set standard inverse} and the fact that $L\mu_{\omega, N}(\ns({^*}S)) = 1$ for all $\omega \in E_F \subseteq E_N$. Using \eqref{limit statement} and \eqref{limit statement 4}, we thus obtain the following:
\begin{align*}
     L\mu_{\omega, N}\left[\st^{-1}(F) \backslash \cup_{n \in \mathbb{N}}\st^{-1}(F \cap K_n)\right] &= 0.
\end{align*}
Since $\cup_{n \in \mathbb{N}}\st^{-1}(F \cap K_n) \subseteq \st^{-1}(F)$, we thus conclude that 
\begin{align}\label{limit statement 5}
    L\mu_{\omega, N}\left[\cup_{n \in \mathbb{N}}\st^{-1}(F \cap K_n) \right] = L\mu_{\omega, N}(\st^{-1}(F)). 
\end{align}
Using \eqref{limit statement 5} in \eqref{limit statement 3} completes the proof. 
\end{proof}

Having proved \eqref{closed set lemma} for closed sets, it is easy to generalize it for all Borel sets using the standard measure theory trick of showing that the collection of sets satisfying equation \eqref{closed set lemma} forms a sigma algebra. This is the next result. 

\begin{theorem}\label{standard part commutativity}
Let $S$ be a Hausdorff space and let $N \in {^*}\mathbb{N}$. Let $B$ be a Borel subset of $S$. Then we have the following:
\begin{align}\label{Borel set lemma}
    L\mu_{\omega, N}(\st^{-1}(B)) = L\mu_{\omega, N}({^*}B) \text{ for } L{^*}\mathbb{P} \text{-almost all } \omega \in {^*}\Omega.
\end{align}
\end{theorem}

\begin{proof}
Let $E_N$ be as in \eqref{almost sure set}. By Lemma \ref{standard part exists for almost all}, we know that $\st^{-1}(B)$ is $L\mu_{\omega, N}$-measurable for all $\omega \in E_N$ and $B \in \mathcal{B}(S)$. Consider the following collection:
\begin{align}
    \mathcal{G} \defeq \{B \in \mathcal{B}(S): &~~\exists E_B \in L({^*}\mathcal{F}) \nonumber \\
    &[\left(L{^*}\mathbb{P}(E_B) = 1\right) \land \left(\forall \omega \in E_B \cap E_N \left( L\mu_{\omega, N}(\st^{-1}(B)) = L\mu_{\omega, N}({^*}B)\right) \right)] \} \label{sigma algebra argument}.
\end{align}

By Lemma \ref{closed subset and standard inverse}, we know that $\mathcal{G}$ contains all closed sets. In order to show that $\mathcal{G}$ contains all Borel sets, by Dynkin's $\pi\text{-}\lambda$ theorem, it thus suffices to show that $\mathcal{G}$ is a Dynkin system. In other words, it suffices to show the following:
\begin{enumerate}[(i)]
    \item\label{D1} $S \in \mathcal{G}$.
    \item\label{D2} If $B \in \mathcal{G}$, then $S \backslash B \in \mathcal{G}$ as well. 
    \item\label{D3} If $(B_n)_{n \in \mathbb{N}}$ is a sequence of mutually disjoint elements of $\mathcal{G}$, then $\cup_{n \in \mathbb{N}} B_n \in \mathcal{G}$.
\end{enumerate}

\ref{D1} is immediate from Lemma \ref{closed subset and standard inverse}, with $E_S \defeq E_N$. To see \ref{D2}, take $B \in \mathcal{G}$ and let $E_B$ be as \eqref{sigma algebra argument}. Note that for any $\omega \in E_B \cap E_N$, we have:
\begin{align*}
    L\mu_{\omega, N}\left({^*}(S \backslash B) \right) &=   L\mu_{\omega, N} \left({^*}S \backslash {^*}B\right)\\
    &=   L\mu_{\omega, N}({^*}S) -   L\mu_{\omega, N}({^*}B)\\
    &= L\mu_{\omega, N}(\st^{-1}(S)) -   L\mu_{\omega, N}(\st^{-1}(B)) \\
    &= L\mu_{\omega, N} \left(\st^{-1}(S) \backslash \st^{-1}(B)\right) \\
    &= L\mu_{\omega, N} \left(\st^{-1}(S \backslash B)\right). 
\end{align*}

In the above argument, the third line used the fact that $S$ and $B$ are in $\mathcal{G}$, the fourth line used the fact that $\st^{-1}(B) \subseteq \st^{-1}(S)$, and the fifth line used the fact that $\st^{-1}(S) \backslash \st^{-1}(B) = \st^{-1}(S \backslash B)$  (which can be seen to follow from Lemma \ref{Hausdorff lemma} since $S$ is Hausdorff). 

We now prove \ref{D3}. Let $(B_n)_{n \in \mathbb{N}}$ be a sequence of mutually disjoint members of $\mathcal{G}$ and let $B \defeq \sqcup_{n \in \mathbb{N}}B_n$. By Lemma \ref{Hausdorff lemma} and the fact that $B_n \in \mathcal{G}$ for all $n \in \mathbb{N}$, we have the following for all $\omega \in {^*}\Omega$:
\begin{align}
    L\mu_{\omega, N} \left(\st^{-1}\left( B \right) \right) &=  L\mu_{\omega, N} \left(\st^{-1}\left( \sqcup_{n \in \mathbb{N}}B_n \right) \right) \nonumber \\
    &= L\mu_{\omega, N} \left( \sqcup_{n \in \mathbb{N}} \st^{-1}(B_n) \right) \nonumber \\
    &= \sum_{n \in \mathbb{N}} L\mu_{\omega, N} \left( \st^{-1}(B_n) \right) \nonumber \\
    &= \sum_{n \in \mathbb{N}} L\mu_{\omega, N} \left({^*}B_n \right). \label{star of countable union}
\end{align}

Let $E_{(B_n)_{n \in \mathbb{N}}}$ be as in Lemma \ref{leftover has zero mass} and define $E_B \defeq E_{(B_n)_{n \in \mathbb{N}}}$. Using \eqref{star of countable union} and \eqref{leftover mass equation}, we thus obtain the following:
\begin{align*}
    L\mu_{\omega, N} \left(\st^{-1}\left( B \right) \right) = L\mu_{\omega, N} \left[{^*} \left(\sqcup_{n \in \mathbb{N}} B_n \right)\right] = L\mu_{\omega, N}({^*}B) \text{ for any } \omega \in E_B \cap E_N,
\end{align*}
completing the proof.
\end{proof}

Recall that, by Lemma \ref{standard part exists for almost all}, if $S$ is Hausdorff then $\mu_{\omega, N} \in \ns({^*}\prs)$, with $\st(\mu_{\omega, N}) =  L\mu_{\omega, N} \circ \st^{-1}$  for all $\omega \in E_N$. Thus Theorem \ref{standard part commutativity} shows the following:

\begin{theorem}\label{commutativity theorem}
Let $S$ be a Hausdorff space. For any Borel set $B \in \mathcal{B}(S)$, we have \begin{align}\label{commutativity}
    \st(\mu_{\omega, N}({^*}B)) = (\st(\mu_{\omega, N}))(B) \text{ for almost all } \omega \in {^*}\Omega.
\end{align}
\end{theorem}

We point out an interesting interpretation of Theorem \ref{commutativity theorem}. For each Borel set $B \in \mathcal{B}(S)$, the Loeb measure $L\mu_{\omega, N}({^*}B)$ can almost surely be computed by either of the following two-step procedures:
\begin{enumerate}[(i)]
    \item First find $\mu_{\omega, N}({^*}B) \in {^*}[0,1]$ and then take the standard part of this finite nonstandard real number, which is the direct way. 
    \item First take the standard part of the internal measure $\mu_{\omega, N} \in {^*}\prs$, and then compute the measure $(\st(\mu_{\omega, N}))(B)$ of $B$ with respect to this standard part. 
\end{enumerate}

Since the intersection of countably many almost sure sets is almost sure, we have thus shown the almost sure commutativity of the following diagram for any countable subset $\mathcal{C} \subseteq \mathcal{B}(S)$:

\begin{center}
\begin{tikzcd}[row sep=0.7in]
& {}^*[0,1] \arrow{dr}{\st} & \\
\mathcal{C} \arrow{ur}{B \mapsto \mu_{\omega, N}({^*}B)} \arrow{rr}{\st(\mu_{\omega, N})} & & {[0,1]}
\end{tikzcd}
\end{center}

It is also interesting to remark that equation \eqref{Borel set lemma} in the conclusion of Theorem \ref{standard part commutativity} is related to the notion of the so-called \textit{standardly distributed} internal measures, first defined in Anderson \cite[Definition 8.1, p. 683]{Anderson-transactions} as a concept motivated by an application to mathematical economics \'a la Anderson \cite{Anderson--Econometrica}. 
\begin{definition}\label{standardly distributed definition}
An internal probability measure $\nu$ on $({^*}S, {^*}\mathcal{B}(S))$ is said to be \textit{standardly distributed} if the following holds:
\begin{align}\label{standardly distributed equation}
    L\nu({^*}B) = L\nu(\st^{-1}(B)) \text{ for all } B \in \mathcal{B}(S).
\end{align}
\end{definition}

Theorem \ref{standard part commutativity} shows that given a particular $B \in \mathcal{B}(S)$ and $N \in {^*}N$, equation \eqref{standardly distributed equation} holds for $\nu$ of the type $\mu_{\omega, N}$ for $L{^*}\mathbb{P}$-almost all $\omega$. Using a more quantitative approach, Anderson \cite[Theorem 8.7(i), p. 685]{Anderson-transactions} shows a stronger version of this result, though with the added hypothesis that the $(X_n)_{n \in \mathbb{N}}$ are independent.

\subsection{Pushing down certain Loeb integrals on the space of all Radon probability measures}\label{3.4}
We finish this appendix by relating certain nonstandard integrals over the space $({^*}\prs, {L_{P_N}}({^*}{\sigmabprs}), P_N)$ to those over the space $(\prs, \mathcal{B}(\prs), LP_N \circ \st^{-1})$.  
\begin{theorem}\label{theorem 1}
Suppose $S$ is a Hausdorff space. Let $N \in {^*}\mathbb{N}$ and let $P_N$ be as in \eqref{definition of P}. Let $({^*}\prs, {L_{P_N}}({^*}\mathcal{B}(\prs)), LP_N)$ be the associated Loeb space. Then for any Borel subset ${B}$ of $S$, we have:
\begin{align}\label{closed subsets of ps}
\starint_{{^*}\prs} \mu({^*}B) dP_N(\mu) \approx \int_{\prs} \mu(B) d\mathscr{P}_N(\mu),    
\end{align}
where $\mathscr{P}_N = LP_N \circ \st^{-1} \in \prs$.
\end{theorem}

\begin{proof}
Fix $B \in \mathcal{B}(S)$. By Corollary \ref{transfer corollary 1} and \eqref{definition of P}, the function $\mu \mapsto \mu({^*}B)$ is internally measurable on ${^*}\prs$. Since it is finitely bounded (by one), it is $\mathbf{S}$-integrable. Using this and Lemma \ref{bounded function transfer result}, we thus obtain the following:
\begin{align*}
     {^*}\mathbb{E}_{P_N}(\mu({^*}B)) &\approx \int_{{^*}\prs} \st(\mu({^*}B)) dLP_N(\mu) \\
     &= \int_{{^*}\Omega} \st(\mu_{\omega, N}({^*}B)) dL{^*}\mathbb{P}(\omega) \\
     &=  \int_{{^*}\Omega} (\st(\mu_{\omega, N}))(B) dL{^*}\mathbb{P}(\omega),    
\end{align*}
where we used Theorem \ref{commutativity theorem} in the last line. Writing the last integral as a Lebesgue integral of tail probabilities, we make the following conclusion:
\begin{align*}
      {^*}\mathbb{E}_{P_N}(\mu({^*}B)) &\approx \int_{[0,1]} L{^*}\mathbb{P}((\st(\mu_{\omega, N}))(B) > y) 
      d\lambda(y) \\
      &= \int_{[0,1]} L{^*}\mathbb{P} \left[ {\mu_{\cdot, N}}^{-1} \left(\st^{-1}\left(\{\nu \in \prs: \nu(B) > y\} \right) \right) \right] d\lambda(y) \\
      &= \int_{[0,1]} LP_N \left(\st^{-1}\left(\{\nu \in \prs: \nu(B) > y\} \right) \right) d\lambda(y),
\end{align*}
where the last line follows from Corollaries \ref{corollary} and \ref{measurability corollary}. (This also uses the fact that the set $\{\nu \in \prs: \nu(B) > y\}$ is Borel measurable, in view of Lemma \ref{evaluation of Borel sets'}.)

Defining $\mathscr{P}_N \defeq LP_N \circ \st^{-1}$ and noting that $\mathscr{P}_N$ is a Radon probability measure on $\prs$ (by Theorem \ref{second order nearstandard}), we obtain the following:
\begin{align*}
     {^*}\mathbb{E}_{P_N}(\mu({^*}B)) &\approx \int_{[0,1]} \mathscr{P}_N \left(\{\nu \in \prs: \nu(B) > y\} \right) d\lambda(y) \\
     &= \int_{\ps} \mu(B) d\mathscr{P}_N(\mu),
\end{align*}
thus completing the proof.
\end{proof}

Note that the same proof idea can be used to prove the version of \eqref{closed subsets of ps} for multiple Borel sets. Indeed, we have the following theorem.

\begin{theorem}\label{theorem 2}
Suppose $S$ is a Hausdorff space. Let $N \in {^*}\mathbb{N}$ and let $P_N$ be as in \eqref{definition of P}. Let $({^*}\prs, {L_{P_N}}({^*}{\sigmabprs}), LP_N)$ be the associated Loeb space. Then for finitely many Borel subsets $B_1, \ldots, B_k$ of $S$, we have:
\begin{align}\label{many closed subsets of ps}
\starint_{{^*}\prs} \mu({^*}B_1)\cdots\mu({^*}B_k) dP_N(\mu) \approx \int_{\prs} \mu(B_1) \cdots \mu(B_k) d\mathscr{P}_N(\mu),    
\end{align}
where $\mathscr{P}_N = LP_N \circ \st^{-1}$.
\end{theorem}

The proof goes exactly the same way as that of Theorem \ref{theorem 1}, once we know that the set $\{\nu \in \prs: \nu(B_
1) \cdots \nu(B_k) > y\}$ is Borel measurable in $\prs$ for all $y \in [0,1]$. But this follows from the fact that a product of measurable functions is measurable (and that for each $i \in \mathbb[k]$, the function $\nu \mapsto \nu(B_i)$ is measurable by Lemma \ref{evaluation of Borel sets'}). 

Combining with Lemma \ref{bounded function transfer result}, we can interject a ${^*}\mathbb{P}$-integral in the approximate equation \eqref{many closed subsets of ps}, which, after passing down to Loeb measures, yields the following corollary which is a crucial tool in our generalization of the de Finetti--Hewitt--Savage theorem in the main body of the paper. 

\begin{corollary}\label{most useful corollary}
Suppose $S$ is a Hausdorff space. Let $N \in {^*}\mathbb{N}$ and let $P_N$ be as in \eqref{definition of P}. Let $({^*}\ps, {L_{P_N}}({^*}{\sigmabprs}), LP_N)$ be the associated Loeb space. Let $\mathscr{P}_N = LP_N \circ \st^{-1}$, which is a Radon measure on $\prs$. Then for finitely many Borel subsets $B_1, \ldots, B_k$ of $S$, we have:
\begin{align}\label{closed subsets of ps 2}
\int_{{^*}\Omega} L\mu_{\omega, N}({^*}B_1)\cdots L\mu_{\omega, N}({^*}B_k) dL{^*}\mathbb{P}(\omega) = \int_{\prs} \mu(B_1) \cdots \mu(B_k) d\mathscr{P}_N(\mu).  
\end{align}
\end{corollary}

\section{An elementary combinatorial proof of Theorem \ref{internal version}}\label{AppC}
Let us recall the set-up of Theorem \ref{internal version}. We are given an internal probability space $(\Gamma, \mathcal{A}, \mathbf{P})$ and an internal measurable space $(\mathbb{S}, \mathfrak{S})$. For some fixed $N > \mathbb{N}$, we are given a hyperfinite collection $\{X_i : i \in [N]\}$ of exchangeable $\mathbb{S}$-valued internal random variables defined on $\Gamma$. 

For each $\gamma \in \Gamma$, we have defined the $N^{\text{th}}$ empirical mean at $\gamma$ to be the internal function $\mu_{\gamma, N} \colon \mathfrak{S} \to {^*}[0,1]$ satisfying the following formula:
   \begin{align}\label{AppgeneralEmpiricalDefinition}
       \mu_{\gamma, N}(\mathscr{S}) \defeq \frac{\# \{i \in [N]: X_i(\gamma) \in \mathscr{S} \}}{N} \text{ for all } \mathscr{S} \in \mathfrak{S}. 
   \end{align}
   
 It is clear that $\mu_{\gamma, N}$ is an internal probability on $(\mathbb{S}, \mathfrak{S})$ for all $\gamma \in \Gamma$. Rewriting \eqref{AppgeneralEmpiricalDefinition} to say $\mu_{\gamma, N}(\mathscr{S}) = \frac{1}{N}\sum_{i \in [N]} \mathbbm{1}_{\mathscr{S}}(X_i(\gamma))$ shows that for each $\mathscr{S}$, the map $\gamma \mapsto \mu_{\gamma, N}(\mathscr{S})$ is a ${^*}\mathbb{R}$-valued internal random variable defined on $\Gamma$.

It only remains to show \eqref{generalLoebDeFinetti} now. To that end, let $(\mathscr{S}_1, \ldots, \mathscr{S}_k) \in {^*}\mathfrak{S}^k$ be a finite sequence of internal events. Consider the following equation obtained by rewriting the internal product of internal sums on the left as an internal sum of internal products by (the transfer of) distributivity: 
\begin{align}\label{internal product}
    \prod_{i \in [k]} \left(\sum_{j \in [N]} \mathbbm{1}_{\mathscr{S}_i}(X_j)\right) = \sum_{\substack{(\ell_1, \ldots, \ell_k) \in [N]^k}} \left(\prod_{i \in [k]} \mathbbm{1}_{\mathscr{S}_{i}}(X_{\ell_i}) \right).
\end{align}

We separate the terms in the sum on the right of \eqref{internal product} according to whether there is any repetition in $(\ell_1, \ldots, \ell_k)$ or not. Let $$\mathcal{R} \defeq \{(\ell_1, \ldots, \ell_k) \in [N]^k : \ell_\alpha = \ell_\beta \text{ for some } \alpha \neq \beta\}.$$ An exact value of $\#(\mathcal{R})$ can be found using the (internal) inclusion-exclusion principle. However, the following immediate combinatorial estimate will be sufficient for our needs (for each of the $N$ numbers in $[N]$, there are at most $\binom{k}{2}N^{k-2}$ elements of $[N]^k$ in which that number is repeated at least twice):
\begin{align}\label{Combinatorial estimate}
    \#(\mathcal{R}) \leq N\binom{k}{2}N^{k-2} = \binom{k}{2}N^{k-1}.
\end{align}

Dividing both sides of \eqref{internal product} by $N^k$ and noting that $\frac{1}{N}\sum_{j \in [N]} \mathbbm{1}_{\mathscr{S}_i}(X_j)$ is the same as $\mu_{\cdot, N}(\mathscr{S}_i)$ for each $i \in [k]$, we obtain the following:
\begin{align}
\prod_{i \in [k]} \mu_{\cdot, N}(\mathscr{S}_i) = \frac{1}{N^k}\sum_{(\ell_1, \ldots, \ell_k) \in \mathcal{R}} \left(\prod_{i \in [k]} \mathbbm{1}_{\mathscr{S}_{i}}(X_{\ell_i}) \right) + \frac{1}{N^k}\sum_{(\ell_1, \ldots, \ell_k) \in [N]^k \backslash \mathcal{R}} \left(\prod_{i \in [k]} \mathbbm{1}_{\mathscr{S}_{i}}(X_{\ell_i}) \right) \nonumber.
\end{align}

Taking internal expectations with respect to the internal probability $\mbp$, and denoting the corresponding internal expectation operator by `$\mathbf{E}$', we thus obtain:
\begin{align}\nonumber
   0 \leq &\mathbf{E} \left(\prod_{i \in [k]} \mu_{\cdot, N}(\mathscr{S}_i) \right) - \mathbf{E}\left[ \frac{1}{N^k}\sum_{(\ell_1, \ldots, \ell_k) \in [N]^k \backslash \mathcal{R}} \left(\prod_{i \in [k]} \mathbbm{1}_{\mathscr{S}_{i}}(X_{\ell_i}) \right)\right] \\
    = &\mathbf{E} \left[\frac{1}{N^k}\sum_{\ell_1, \ldots, \ell_k \in \mathcal{R}} \left(\prod_{i \in [k]} \mathbbm{1}_{\mathscr{S}_{i}}(X_{\ell_i}) \right) \right] \nonumber \\
  \leq &\frac{\#(\mathcal{R})}{N^k} \nonumber\\
  \leq &\frac{\binom{k}{2}N^{k-1}}{N^k} \nonumber = \frac{\binom{k}{2}}{N} \approx 0,
\end{align}
where \eqref{Combinatorial estimate} is used at the beginning of the last line above. As a consequence, we thus obtain the following:
\begin{align}\label{combinatorial step used 2}
      &\mathbf{E} \left(\prod_{i \in [k]} \mu_{\cdot, N}(\mathscr{S}_i) \right) \approx  \frac{1}{N^k}\sum_{(\ell_1, \ldots, \ell_k) \in [N]^k \backslash \mathcal{R}} \mathbf{E}\left( \prod_{i \in [k]} \mathbbm{1}_{\mathscr{S}_{i}}(X_{\ell_i})\right).
\end{align}
By exchangeability, we also have the following:
\begin{align}\nonumber
   \mathbf{E}\left(\prod_{i \in [k]} \mathbbm{1}_{\mathscr{S}_{i}}(X_{\ell_i})\right) &= \mathbf{P}(X_{\ell_1} \in \mathscr{S}_1, \ldots, X_{\ell_k} \in \mathscr{S}_k) \\
   &= \mathbf{P}(X_{1} \in \mathscr{S}_1, \ldots, X_{k} \in \mathscr{S}_k) \text{ for all } (\ell_1, \ldots, \ell_k) \in [N]^k \backslash \mathcal{R} \label{new equation 2},
\end{align}
which allows us to conclude the following from \eqref{combinatorial step used 2}:
\begin{align}\label{combinatorial step used 3}
      &\mathbf{E} \left(\prod_{i \in [k]} \mu_{\cdot, N}(\mathscr{S}_i) \right) \approx  \frac{\#([N]^k \backslash \mathcal{R})}{N^k} \mathbf{P}(X_{1} \in \mathscr{S}_1, \ldots, X_{k} \in \mathscr{S}_k).
\end{align}

From \eqref{Combinatorial estimate}, it is clear that
\begin{align}
  1 >  \frac{\#([N]^k \backslash \mathcal{R})}{N^k} \geq \frac{N^k - \binom{k}{2}N^{k-1}}{N^k} = 1 - \frac{\binom{k}{2}}{N} \approx 1, \label{new equation 3}
\end{align}
so that 
\begin{align}\label{combinatorial estimate 2}
\frac{\#([N]^k \backslash \mathcal{R})}{N^k} \approx 1.
\end{align}

Using \eqref{combinatorial estimate 2} in \eqref{combinatorial step used 3} yields the following:
\begin{align}\nonumber
      &\mathbf{E} \left(\prod_{i \in [k]} \mu_{\cdot, N}(\mathscr{S}_i) \right) \approx   \mathbf{P}(X_{1} \in \mathscr{S}_1, \ldots, X_{k} \in \mathscr{S}_k),
\end{align}
thus completing the proof after taking standard parts and using $\mathbf{S}$-integrability of the map $\gamma \mapsto \prod_{i \in [k]} \mu_{\cdot, N}(\mathscr{S}_i)$. 

\section{A proof of Theorem \ref{internal version} using conditional probabilities}\label{appendix 2}
In this appendix, we will carry out an alternative proof of Theorem \ref{internal version} using the technique of \textit{conditioning}. The proof that we will present here is a refinement of the key idea from \cite{Alam-deFinetti}. We restate Theorem \ref{internal version} for convenience.  

\begin{customthm}{2.1}\label{new 2.1}
Let $\Gamma$ be an internal set, $\mathcal{A}$ be an internal algebra on $\Gamma$, and $\mathbf{P}$ be an internal hyperfinitely-additive map $\mathbf{P} \colon \mathcal{A} \to {^*}[0,1]$ such that $\mathbf{P}(\Gamma) \approx 1$. Let $\mathbb{S}$ be a (possibly different) internal set equipped with an internal algebra $\mathfrak{S}$. 

Let $N > \mathbb{N}$---that is, $N \in {^*}\mathbb{N} \backslash \mathbb{N}$. Let us denote by $[N]$ the initial segment of $N$ in ${^*}\mathbb{N}$, and let $\{\mathfrak{X}_i : i \in [N]\}$ be a hyperfinite collection of exchangeable $\mathbb{S}$-valued internal random variables defined on $\Gamma$---that is, each $\mathfrak{X}_i \colon \Gamma \to \mathbb{S}$ in this internal collection is an internal map such that the pre-image ${\mathfrak{X}_i}^{-1}(\mathscr{S}) \in \mathcal{A}$ for all $\mathscr{S} \in \mathfrak{S}$, while for any finitely many sets $\mathscr{S}_1, \ldots, \mathscr{S}_k \in \mathfrak{S}$ and any permutation $\sigma \in S_N$ (the internal symmetric group on $[N]$) we have $\mathbf{P}(\mathfrak{X}_1 \in \mathscr{S}_1, \ldots, \mathfrak{X}_k \in \mathscr{S}_k) = \mathbf{P}(\mathfrak{X}_{\sigma(1)} \in \mathscr{S}_1, \ldots, \mathfrak{X}_{\sigma(k)} \in \mathscr{S}_k)$.

For each $\gamma \in \Gamma$, we define the $N^{\text{th}}$ empirical mean at $\gamma$ to be the internal function $\mu_{\gamma, N} \colon \mathfrak{S} \to {^*}[0,1]$ satisfying the following formula:
   \begin{align}\label{generalEmpiricalDefinitionA}
       \mu_{\gamma, N}(\mathscr{S}) \defeq \frac{\# \{i \in [N]: \mathfrak{X}_i(\gamma) \in \mathscr{S} \}}{N} \text{ for all } \mathscr{S} \in \mathfrak{S}. 
   \end{align}
   
 Then $\mu_{\gamma, N}$ is an internal finitely additive probability on $(\mathbb{S}, \mathfrak{S})$ for all $\gamma \in \Gamma$. Furthermore, for each $\mathscr{S}$, the map $\gamma \mapsto \mu_{\gamma, N}(\mathscr{S})$ is a ${^*}\mathbb{R}$-valued internal random variable such that the following is true:

 \begin{align}\label{generalLoebDeFinettiA}
     L\mathbf{P}(\mathfrak{X}_1 \in \mathscr{S}_1, \ldots, \mathfrak{X}_k \in \mathscr{S}_k) = \int_{\Gamma} L\mu_{\gamma, N}(\mathscr{S}_1)\cdots L\mu_{\gamma, N}(\mathscr{S}_k) dL\mathbf{P}(\gamma) \nonumber \\
    \text{ for all } k \in \mathbb{N} \text{ and } \mathscr{S}_1, \ldots, \mathscr{S}_k \in \mathfrak{S}, 
 \end{align}
 where $(\Gamma, L(\mathcal{A}), L\mathbf{P})$ and $(\mathbb{S}, L_{\mu_{\gamma, N}}(\mathfrak{S}), L\mu_{\gamma, N})_{\gamma \in \Gamma}$ are the Loeb spaces induced by $(\Gamma, \mathcal{A}, \mathbf{P})$ and $(\mathbb{S}, \mathfrak{S}, \mu_{\gamma, N})_{\gamma \in \Gamma}$ respectively.
\end{customthm}

It turns out that one difficulty in a direct generalization of the method in \cite{Alam-deFinetti} is that the sets $\mathscr{S}_i$ were all either $\{0\}$ or $\{1\}$ in \cite{Alam-deFinetti}, while they may have intersections in the current context. We get around this difficulty by observing that it suffices to prove \eqref{generalLoebDeFinettiA} for tuples $(\mathscr{S}_1, \ldots, \mathscr{S}_k)$ such that $\mathscr{S}_i$ and $\mathscr{S}_j$ are either disjoint or equal for all $i, j \in [k]$.

\begin{definition}
Call a finite tuple $(\mathscr{S}_1, \ldots, \mathscr{S}_k)$ of sets \textit{disjointified} if for all $i , j \in [k]$, we have $\mathscr{S}_i \cap \mathscr{S}_j = \emptyset$ or $\mathscr{S}_i \cap \mathscr{S}_j = \mathscr{S}_i = \mathscr{S}_j$. In the setting of Theorem \ref{new 2.1}, call an event \textit{disjointified} if it is of the type $\{\mathfrak{X}_1 \in \mathscr{S}_1, \ldots, \mathfrak{X}_k \in \mathscr{S}_k\}$ for some disjointified tuple $(\mathscr{S}_1, \ldots, \mathscr{S}_k)$.
\end{definition}

\begin{lemma}\label{lemma for proposition}
Let $N > \mathbb{N}$. In the setting of Theorem \ref{new 2.1}, suppose that
\begin{align}
    L\mathbf{P}(\mathfrak{X}_1 \in \mathscr{D}_1, \ldots, \mathfrak{X}_k \in \mathscr{D}_k) = \int_{\Gamma} L\mu_{\gamma, N}(\mathscr{D}_1)\cdots\mu_{\gamma, N}(\mathscr{D}_k) dL\mathbf{P}(\gamma) \label{what we want2} \\
    \text{ for all } k \in \mathbb{N} \text{ and } \mathscr{D}_1, \ldots, \mathscr{D}_k \in \mathfrak{S} \text{ such that } (\mathscr{D}_1, \ldots, \mathscr{D}_k) \text{ is disjointified}. \nonumber
\end{align}
Then \eqref{generalLoebDeFinettiA} holds. 
\end{lemma}

\begin{proof}
Suppose \eqref{what we want2} holds. Let $\mathscr{S}_1, \ldots, \mathscr{S}_k \in \mathfrak{S}$ be fixed. We can write the event $\{\mathfrak{X}_1 \in \mathscr{S}_1, \ldots, \mathfrak{X}_k \in \mathscr{S}_k\}$ as a disjoint union of disjointified events. Indeed, for $d \in \{0,1\}$ and a set $\mathscr{S} \subseteq \mathbb{S}$, let $\mathscr{S}^d$ be equal to $\mathscr{S}$ if $d = 1$, and let it be equal to the complement $\mathbb{S} \backslash B$ if $d = 0$. For a tuple $a = (a_1, \ldots, a_k) \in \{0, 1\}^k$ of zeros and ones, define the following set:
\begin{align}\label{zeros and ones}
    [\mathscr{S}_1, \ldots, \mathscr{S}_k]^a \defeq \bigcap_{i \in [k]} {\mathscr{S}_i}^{a_i}.
\end{align}

Being a finite intersection of sets in the internal algebra $\mathfrak{S}$, it follows that $[\mathscr{S}_1, \ldots, \mathscr{S}_k]^{a} \in \mathfrak{S}$ for all $a \in \{0,1\}^k$. For $i \in [k]$, define $\mathfrak{D}_i \defeq \{(a_1, \ldots, a_k) \in \{0,1\}^k: a_i = 1\}$. For a tuple $\tilde{a} = (\tilde{a}_1, \ldots, \tilde{a}_k) \in \mathfrak{D}_1 \times \ldots \times \mathfrak{D}_k$ of $k$-tuples, we define \begin{align}\label{zeros and ones 2}
    [\mathscr{S}_1, \ldots, \mathscr{S}_k]^{\tilde{a}} \defeq \{\mathfrak{X}_1 \in  [\mathscr{S}_1, \ldots, \mathscr{S}_k]^{\tilde{a}_1}, \ldots, \mathfrak{X}_k \in  [\mathscr{S}_1, \ldots, \mathscr{S}_k]^{\tilde{a}_k} \}.
\end{align}

It is clear that the event $[\mathscr{S}_1, \ldots, \mathscr{S}_k]^{\tilde{a}}$ is disjointified for each $\tilde{a} \in \mathfrak{D}_1 \times \ldots \times \mathfrak{D}_k$, and that 
$$[\mathscr{S}_1, \ldots, \mathscr{S}_k]^{\tilde{a}} \cap [\mathscr{S}_1, \ldots, \mathscr{S}_k]^{\tilde{b}} = \emptyset \text{ if } \tilde{a}, \tilde{b} \text{ are distinct elements of } \mathfrak{D}_1 \times \ldots \times \mathfrak{D}_k.$$ 

We thus have the following representation as a disjoint union of disjointified events:
\begin{align}\label{disjointifying}
    \{\mathfrak{X}_1 \in \mathscr{S}_1, \ldots, \mathfrak{X}_k \in \mathscr{S}_k\} = \bigsqcup_{\tilde{a} \in \mathfrak{D}_1 \times \ldots \times \mathfrak{D}_k} [\mathscr{S}_1, \ldots, \mathscr{S}_k]^{\tilde{a}}. 
\end{align}

Since $\mathbf{P}$ is (hyper)finitely additive, we obtain the following by assumption \eqref{what we want2} for disjointified events in view of the linearity of integrals:

\begin{align}\label{key idea}
    &L\mathbf{P}(\mathfrak{X}_1 \in \mathscr{S}_1, \ldots, \mathfrak{X}_k \in \mathscr{S}_k) \\
    = &\sum_{\tilde{a} \in \mathfrak{D}_1 \times \ldots \times \mathfrak{D}_k}L\mathbf{P}\left([\mathscr{S}_1, \ldots, \mathscr{S}_k]^{\tilde{a}} \right) \\
     = &\sum_{\tilde{a} = (\tilde{a}_1, \ldots, \tilde{a}_k) \in \mathfrak{D}_1 \times \ldots \times \mathfrak{D}_k} \int_{\Gamma} \left(\prod_{i \in [k]} L\mu_{\gamma, N} \left(  [\mathscr{S}_1, \ldots, \mathscr{S}_k]^{\tilde{a}_i}\right) \right) dL\mathbf{P}(\gamma)\\
     = & \int_{\Gamma}\left(\sum_{\tilde{a} = (\tilde{a}_1, \ldots, \tilde{a}_k) \in \mathfrak{D}_1 \times \ldots \times \mathfrak{D}_k} \prod_{i \in [k]} L\mu_{\gamma, N} \left(  [\mathscr{S}_1, \ldots, \mathscr{S}_k]^{\tilde{a}_i}\right)\right) dL\mathbf{P}(\gamma) \label{key idea 2}\\
     = & \int_{\Gamma}  L\mu_{\gamma, N}(\mathscr{S}_1)\cdots L\mu_{\gamma, N}(\mathscr{S}_k) dL\mathbf{P}(\gamma),
\end{align}
where the last line follows from an elementary application of finite additivity of $L\mu_{\gamma, N}$. 
\end{proof}

For the rest of the paper, fix $N > \mathbb{N}$, $k \in \mathbb{N}$, and $\mathscr{D}_1, \ldots, \mathscr{D}_k \in \mathfrak{S}$ such that the tuple $(\mathscr{D}_1, \ldots, \mathscr{D}_k)$ is disjointified. By Lemma \ref{lemma for proposition}, it thus suffices (for a proof of Theorem \ref{internal version}) to show that \eqref{what we want2} holds for our fixed disjointified tuple $(\mathscr{D}_1, \ldots, \mathscr{D}_k)$. 

For some $n \in \mathbb{N}$, let $\mathscr{C}_1, \ldots, \mathscr{C}_n$ be the distinct (disjoint) sets appearing in the tuple $(\mathscr{D}_1, \ldots, \mathscr{D}_k)$. For each $i \in [n]$, let $\mathscr{C}_i$ appear in $(\mathscr{D}_1, \ldots, \mathscr{D}_k)$ with a frequency $k_i$. Note that this necessarily implies that $k_1 + \ldots + k_n = k$ and $k_i \geq 1$ for all $i \in [n]$. 

For each $i \in [n]$, let $Y_i \co \Gamma \rightarrow [N]$ be defined as follows:
\begin{align}\label{definition of Y_i}
    Y_i(\gamma) \defeq \#\{j \in [N]: \mathfrak{X}_j(\gamma) \in \mathscr{C}_i\} = \sum_{j \in [N]} \mathbbm{1}_{\mathscr{C}_i}(\mathfrak{X}_j(\gamma)) \text{ for all } \gamma \in \Gamma.
\end{align}
Thus $\mu_{\gamma, N}(\mathscr{C}_i) = \frac{Y_i(\gamma)}{N}$ \text{ for all } $\gamma \in \Gamma$. 

Let $\vec{\mathscr{D}}$, $\vec{\mathfrak{X}}$, and $\vec{Y}$ denote the tuples $(\mathscr{D}_1, \ldots, \mathscr{D}_k)$, $(\mathfrak{X}_1, \ldots, \mathfrak{X}_k)$, and $(Y_1, \ldots, Y_n)$ respectively. The following lemma follows from elementary combinatorial arguments.

\begin{lemma}\label{claim 0}
Suppose that $t_i \in {^*}\mathbb{N}$ are such that $t_i \geq k_i$ for all $i \in [n]$, and such that $\mathbf{P}(\vec{Y} = (t_1, \ldots, t_n)) > 0$. Then we have: 
\begin{align}
    \mathbf{P}(\vec{\mathfrak{X}} \in \vec{\mathscr{D}} \vert \vec{Y} = (t_1, \ldots, t_n))  
    &= \frac{1}{N(N-1)\ldots (N - (k-1))} \cdot \frac{t_1! \ldots t_n!}{(t_1 - k_1)! \ldots (t_n - k_n)!}\label{combinatorial step}.
\end{align}
\end{lemma}

\begin{proof}
Let $t_1, \ldots, t_n$ be as in the statement of the lemma. Define the following event:
\begin{align*}
    E_{t_1, \ldots, t_n} \defeq \{&\mathfrak{X}_1, \ldots, \mathfrak{X}_{t_1} \in \mathscr{C}_1; \\
    &\mathfrak{X}_{t_1 + 1}, \ldots, \mathfrak{X}_{t_1 + t_2} \in \mathscr{C}_2;\\
    &\ldots;\\ 
    &\mathfrak{X}_{t_1 + \ldots + t_{n-1} + 1}, \ldots, \mathfrak{X}_{t_1 + \ldots + t_n} \in \mathscr{C}_n;\\
    &\mathfrak{X}_i \in \mathbb{S} \backslash {\mathscr{C}_1 \sqcup \ldots \sqcup \mathscr{C}_n} \text{ for all other } i \in [N]\}.
\end{align*}

By exchangeability and the fact that the $\mathscr{C}_i$ are disjoint, we have the following:
\begin{align}\label{use of exchangeability}
    \mathbf{P}(\vec{Y} = (t_1, \ldots, t_n)) = N_1\mathbf{P}\left(E_{t_1, \ldots, t_n} \right), \\
    \text{and } \label{use of exchangeability 2}
    \mathbf{P}(\vec{\mathfrak{X}} \in \vec{\mathscr{D}} \text{ and } \vec{Y} = (t_1, \ldots, t_n)) = N_2\mathbf{P}\left(E_{t_1, \ldots, t_n} \right),
\end{align}
where 
\begin{align}
   N_1 &= \text{Number of ways to choose } t_i \text{ spots of the } i^{\text{th}} \text{ kind in } [N] \text{ as } i \text{ varies over } [n] \nonumber \\
   &= \binom{N}{t_1} \binom{N - t_1}{t_2}\cdot \ldots \cdot \binom{N - t_1 - \ldots - t_{n-1}}{t_n}, \label{N_1}
\end{align}
and 
\begin{align}
   N_2 &= \text{Number of ways to choose } (t_i - k_i) \text{ spots of the } i^{\text{th}} \text{ kind in } [N]\backslash [k] \text{ as } i \text{ varies over } [n] \nonumber \\
   &= \binom{N - k}{t_1 - k_1} \binom{N - k - (t_1 - k_1)}{t_2 - k_2}\cdot \ldots \cdot \binom{N - k - (t_1 + \ldots + t_{n-1} - k_1 \ldots - k_{n-1})}{t_n - k_n} \label{N_2}.
\end{align}

Since it is given that $\mathbf{P}(\vec{Y} = (t_1, \ldots, t_n)) > 0$, we thus have $\mathbf{P}\left(E_{t_1, \ldots, t_n} \right) > 0$ by \eqref{use of exchangeability}. By \eqref{use of exchangeability}, \eqref{use of exchangeability 2}, \eqref{N_1}, and \eqref{N_2}, we therefore obtain \eqref{combinatorial step} after elementary simplifications.
\end{proof}

\begin{corollary}\label{corollary of claim 0}
Suppose that $t_i \in {^*}\mathbb{N}$ are such that $\mathbf{P}(\vec{Y} = (t_1, \ldots, t_n)) > 0$. Then we have: 
\begin{align}\label{just close}
    \mathbf{P}(\vec{\mathfrak{X}} \in \vec{\mathscr{D}} \vert \vec{Y} = (t_1, \ldots, t_n)) \approx \left( \frac{t_1}{N}\right)^{k_1}\cdot \ldots \cdot \left( \frac{t_n}{N}\right)^{k_n} \text{ for all such } (t_1, \ldots, t_n) \in [N]^n.
\end{align}
\end{corollary}

\begin{proof}
Suppose that the $t_i \in [N]$ are such that $\mathbf{P}(\vec{Y} = (t_1, \ldots, t_n)) > 0$. If $t_i \geq k_i$ for all $i \in [n]$, then by Lemma \ref{claim 0}, we obtain the following: 
\begin{align}
    \frac{\mathbf{P}(\vec{\mathfrak{X}} \in \vec{\mathscr{D}} \vert \vec{Y} = (t_1, \ldots, t_n))}{\left( \frac{t_1}{N}\right)^{k_1}\cdot \ldots \cdot \left( \frac{t_n}{N}\right)^{k_n}}
    &= \frac{1}{1 - \frac{1}{N}}\ldots \frac{1}{1 - \frac{k-1}{N}} \cdot \prod_{i \in [n]}\left( \prod_{j \in [k_i - 1]} \left(1 - \frac{j}{t_i} \right) \right) \label{ratio close 2} \\
    &< \frac{1}{1 - \frac{1}{N}}\ldots \frac{1}{1 - \frac{k-1}{N}} \approx 1 \label{less than}.
\end{align}

Note that if $t_i > \mathbb{N}$ for all $i \in \mathbb{N}$, then both $\frac{1}{1 - \frac{1}{N}}\ldots \frac{1}{1 - \frac{k-1}{N}} \approx 1$ and $ \prod_{i \in [n]}\left( \prod_{j \in [k_i - 1]} \left(1 - \frac{j}{t_i} \right) \right) \approx 1$, so that \eqref{ratio close 2} implies that 
\begin{align}
    \frac{\mathbf{P}(\vec{\mathfrak{X}} \in \vec{\mathscr{D}} \vert \vec{Y} = (t_1, \ldots, t_n))}{\left( \frac{t_1}{N}\right)^{k_1}\cdot \ldots \cdot \left( \frac{t_n}{N}\right)^{k_n}} \approx 1 \text{ if } t_1, \ldots, t_n > \mathbb{N}, \label{ratio close observation}
\end{align}
which, in particular, implies \eqref{just close} in this case.

Now, if $t_j$ is in $\mathbb{N}$ for some $j \in [n]$ but such that $t_i \geq k$ for all $i \in [n]$ and $\mathbf{P}(\vec{Y} = (t_1, \ldots, t_n)) > 0$, then the inequality in \eqref{less than} implies that 
\begin{align*}
    {\mathbf{P}(\vec{\mathfrak{X}} \in \vec{\mathscr{D}} \vert \vec{Y} = (t_1, \ldots, t_n))} < 2{\left( \frac{t_1}{N}\right)^{k_1}\cdot \ldots \cdot \left( \frac{t_n}{N}\right)^{k_n}} < 2 \left(\frac{t_j}{N} \right)^{k_j}\approx 0, 
\end{align*}
so that 
\begin{align*}
     {\mathbf{P}(\vec{\mathfrak{X}} \in \vec{\mathscr{D}} \vert \vec{Y} = (t_1, \ldots, t_n))} \approx 0 \approx {\left( \frac{t_1}{N}\right)^{k_1}\cdot \ldots \cdot \left( \frac{t_n}{N}\right)^{k_n}}, 
\end{align*}
proving \eqref{just close} in that case as well. 

Finally, if $t_i < k_i$ for any $i \in [n]$, then $\mathbf{P}(\vec{\mathfrak{X}} \in \vec{\mathscr{D}} \vert \vec{Y} = (t_1, \ldots, t_n)) = 0$, while $\left( \frac{t_1}{N}\right)^{k_1}\cdot \ldots \cdot \left( \frac{t_n}{N}\right)^{k_n} \approx 0$ in that case as well. This completes the proof. 
\end{proof}

We record \eqref{ratio close observation} in the proof of Corollary \ref{corollary of claim 0} as its own result.
\begin{corollary}\label{its own result}
Suppose that $t_i > \mathbb{N}$ such that $\mathbf{P}(\vec{Y} = (t_1, \ldots, t_n)) > 0$. Then we have the following approximate equality:
\begin{align}
    \frac{\mathbf{P}(\vec{\mathfrak{X}} \in \vec{\mathscr{D}} \vert \vec{Y} = (t_1, \ldots, t_n))}{\left( \frac{t_1}{N}\right)^{k_1}\cdot \ldots \cdot \left( \frac{t_n}{N}\right)^{k_n}} \approx 1 \text{ if } t_1, \ldots, t_n > \mathbb{N}. \nonumber 
\end{align}
\end{corollary}

By \eqref{less than} and underflow applied to Corollary \ref{its own result}, we obtain the following.

\begin{corollary}\label{m_epsilon}
Given $\epsilon \in \mathbb{R}_{>0}$, there is an  $m_{\epsilon} \in \mathbb{N}$ satisfying the following. 
\begin{align}
    1 - \epsilon < \frac{\mathbf{P}(\vec{\mathfrak{X}} \in \vec{\mathscr{D}} \vert \vec{Y} = (t_1, \ldots, t_n))}{\left( \frac{t_1}{N}\right)^{k_1}\cdot \ldots \cdot \left( \frac{t_n}{N}\right)^{k_n}} < 1 + \epsilon \nonumber \\
    \text{ if } t_1, \ldots, t_n > m_{\epsilon} \text{ are such that } \mathbf{P}(\vec{Y} = (t_1, \ldots, t_n)) > 0. \nonumber 
\end{align}
\end{corollary}

The proof of Corollary \ref{corollary of claim 0} also leads to the following observation.
\begin{corollary}\label{claim 1}
For each $m \in {^*}\mathbb{N}$, define the set 
\begin{align}\label{L_M}
    L_m \defeq \{(t_1, \ldots, t_n) \in [N]^n : \text{ there is } j \in [n] \text{ such that } t_j \leq m\}. 
\end{align}
Then, we have the following for all $m \in \mathbb{N}$:
\begin{align*}
   0 &\approx \sum_{(t_1, \ldots, t_n) \in L_m} \hspace{-5pt}\mathbf{P}(\vec{\mathfrak{X}} \in \vec{\mathscr{D}} \vert \vec{Y} = (t_1, \ldots, t_n))  \mathbf{P}\left(\mu_{\cdot, N}(\mathscr{C}_1) = \frac{t_1}{N}, \ldots, \mu_{\cdot, N}(\mathscr{C}_n) = \frac{t_n}{N} \right) \\
    &\approx  \sum_{(t_1, \ldots, t_n) \in L_m} \left( \frac{t_1}{N}\right)^{k_1}\cdot \ldots \cdot \left( \frac{t_n}{N}\right)^{k_n} \mathbf{P}\left(\mu_{\cdot, N}(\mathscr{C}_1) = \frac{t_1}{N}, \ldots, \mu_{\cdot, N}(\mathscr{C}_n) = \frac{t_n}{N} \right).
\end{align*}
\end{corollary}

\begin{proof}
Let $m \in \mathbb{N}$ and $L_m$ be as in the statement of the corollary. Noting that the event $\left\{\mu_{\cdot, N}(\mathscr{C}_1) = \frac{t_1}{N}, \ldots, \mu_{\cdot, N}(\mathscr{C}_n) = \frac{t_n}{N} \right\}$ is the same as the event $\{\vec{Y} = (t_1, \ldots, t_n)\}$, we obtain the following from \eqref{less than} (we also use the fact that if $t_i < k_i$ for any $i \in [n]$, then $\mathbf{P}(\vec{\mathfrak{X}} \in \vec{\mathscr{D}} \vert \vec{Y} = (t_1, \ldots, t_n)) = 0$):

\begin{align*}
     &\sum_{(t_1, \ldots, t_n) \in L_m} \mathbf{P}(\vec{\mathfrak{X}} \in \vec{A} \vert \vec{Y} = (t_1, \ldots, t_n)) \cdot \mathbf{P}\left(\mu_{\cdot, N}(\mathscr{C}_1) = \frac{t_1}{N}, \ldots, \mu_{\cdot, N}(\mathscr{C}_n) = \frac{t_n}{N} \right) \\
    \leq &~2\sum_{(t_1, \ldots, t_n) \in L_m} \left( \frac{t_1}{N}\right)^{k_1}\cdot \ldots \cdot \left( \frac{t_n}{N}\right)^{k_n} \cdot \mathbf{P}\left(\mu_{\cdot, N}(\mathscr{C}_1) = \frac{t_1}{N}, \ldots, \mu_{\cdot, N}(\mathscr{C}_n) = \frac{t_n}{N} \right) \\
    \leq &~2 \sum_{j \in [n]} \left(\sum_{r \in [m]} \left[\sum_{\substack{(t_1, \ldots, t_n) \in [N]^n \\ t_j = r}} \left(\frac{{t_j}}{N}\right)^{k_j} \cdot \mathbf{P}\left(\mu_{\cdot, N}(\mathscr{C}_1) = \frac{t_1}{N}, \ldots, \mu_{\cdot, N}(\mathscr{C}_n) = \frac{t_n}{N} \right)  \right] \right) \\
     \leq &~2 \sum_{j \in [n]} \left(\left[\sum_{\substack{(t_1, \ldots, t_n) \in [N]^n \\ t_j \leq m}} \frac{m}{N} \cdot \mathbf{P}\left(\mu_{\cdot, N}(\mathscr{C}_1) = \frac{t_1}{N}, \ldots, \mu_{\cdot, N}(\mathscr{C}_n) = \frac{t_n}{N} \right)  \right] \right) \\
    = &~\frac{2m}{N} \cdot \sum_{j \in [n]} \mathbf{P}\left(\mu_{\cdot, N}(\mathscr{C}_j) \leq \frac{m}{N}\right) \\
    \leq &~\frac{2mn}{N} \\
    \approx &~0,
\end{align*}
completing the proof.
\end{proof}

We now have all the ingredients for our proof of Theorem \ref{internal version}.

\begin{proof}[Proof of Theorem \ref{internal version}]
Conditioning on the various possible values of $Y_i$ as $i$ varies in $[n]$, and noting that the event $\left\{\mu_{\cdot, N}(\mathscr{C}_1) = \frac{t_1}{N}, \ldots, \mu_{\cdot, N}(\mathscr{C}_n) = \frac{t_n}{N} \right\}$ is the same as the event $\{\vec{Y} = (t_1, \ldots, t_n)\}$, we obtain:
\begin{align}
     &\mathbf{P}((\mathfrak{X}_1, \ldots, \mathfrak{X}_{k}) \in \vec{\mathscr{D}}) \nonumber \\
     &= \sum_{(t_1, \ldots, t_n) \in [N]^n} \mathbf{P}(\vec{\mathfrak{X}} \in \vec{\mathscr{D}} \vert \vec{Y} = (t_1, \ldots, t_n)) \cdot \mathbf{P}\left(\mu_{\cdot, N}(\mathscr{C}_1) = \frac{t_1}{N}, \ldots, \mu_{\cdot, N}(\mathscr{C}_n) = \frac{t_n}{N} \right) \label{Bayes expansion'}
\end{align}

Now, by $S$-integrability and the definition of internal expected values with respect to internal hyperfinitely additive measures, we have the following equality:
\begin{align}
    &\int_{\Gamma} L\mu_{\gamma, N}(\mathscr{D}_1)\cdots\mu_{\gamma, N}(\mathscr{D}_k) dL\mathbf{P}(\gamma) \\
    \nonumber &= \st\left(\int_{\Gamma} \mu_{\gamma, N}(\mathscr{D}_1)\cdots\mu_{\gamma, N}(\mathscr{D}_k) d\mathbf{P}(\gamma)\right) \\
    &= \st\left(\sum_{(t_1, \ldots, t_n) \in [N]^n} \left( \frac{t_1}{N}\right)^{k_1}\cdot \ldots \cdot \left( \frac{t_n}{N}\right)^{k_n} \cdot \mathbf{P}\left(\mu_{\cdot, N}(\mathscr{C}_1) = \frac{t_1}{N}, \ldots, \mu_{\cdot, N}(\mathscr{C}_n) = \frac{t_n}{N} \right)\right). \label{Expectation expansion'}
\end{align}

Let $\epsilon \in \mathbb{R}_{>0}$ and let $m_{\epsilon} \in \mathbb{N}$ be as in Corollary \ref{m_epsilon}. By that corollary, we obtain:
\begin{align*}
    &\sum_{\substack{(t_1, \ldots, t_n) \in [N]^n}} \mathbf{P}(\vec{\mathfrak{X}} \in \vec{\mathscr{D}} \vert \vec{Y} = (t_1, \ldots, t_n)) \cdot \mathbf{P}\left(\mu_{\cdot, N}(\mathscr{C}_1) = \frac{t_1}{N}, \ldots, \mu_{\cdot, N}(\mathscr{C}_n) = \frac{t_n}{N} \right) \\
    > &\sum_{\substack{(t_1, \ldots, t_n) \in L_{m_{\epsilon}}}} \mathbf{P}(\vec{\mathfrak{X}} \in \vec{\mathscr{D}} \vert \vec{Y} = (t_1, \ldots, t_n)) \cdot \mathbf{P}\left(\mu_{\cdot, N}(\mathscr{C}_1) = \frac{t_1}{N}, \ldots, \mu_{\cdot, N}(\mathscr{C}_n) = \frac{t_n}{N} \right) \\
    &+ (1 - \epsilon) \sum_{\substack{(t_1,\ldots, t_n) \in [N]^n \\ t_1, \ldots, t_n > m_{\epsilon}}} \left( \frac{t_1}{N}\right)^{k_1}\cdot \ldots \cdot \left( \frac{t_n}{N}\right)^{k_n} \cdot \mathbf{P}\left(\mu_{\cdot, N}(\mathscr{C}_1) = \frac{t_1}{N}, \ldots, \mu_{\cdot, N}(\mathscr{C}_n) = \frac{t_n}{N} \right). 
\end{align*}

By taking standard parts and using Corollary \ref{claim 1}, the above yields the following inequality:
\begin{align*}
    &\st \left[\sum_{\substack{(t_1, \ldots, t_n) \in [N]^n}} \mathbf{P}(\vec{\mathfrak{X}} \in \vec{\mathscr{D}} \vert \vec{Y} = (t_1, \ldots, t_n)) \cdot \mathbf{P}\left(\mu_{\cdot, N}(\mathscr{C}_1) = \frac{t_1}{N}, \ldots, \mu_{\cdot, N}(\mathscr{D}_n) = \frac{t_n}{N} \right)\right] \\
    \geq & (1 - \epsilon) \st \left[\sum_{\substack{(t_1,\ldots, t_n) \in [N]^n}} \left( \frac{t_1}{N}\right)^{k_1}\cdot \ldots \cdot \left( \frac{t_n}{N}\right)^{k_n} \cdot \mathbf{P}\left(\mu_{\cdot, N}(\mathscr{C}_1) = \frac{t_1}{N}, \ldots, \mu_{\cdot, N}(\mathscr{C}_n) = \frac{t_n}{N} \right) \right]. 
\end{align*}

Since $\epsilon \in \mathbb{R}_{>0}$ is arbitrary, we thus obtain:  
\begin{align}
    &\st \left[\sum_{\substack{(t_1, \ldots, t_n) \in [N]^n}} \mathbf{P}(\vec{\mathfrak{X}} \in \vec{\mathscr{D}} \vert \vec{Y} = (t_1, \ldots, t_n)) \cdot \mathbf{P}\left(\mu_{\cdot, N}(\mathscr{C}_1) = \frac{t_1}{N}, \ldots, \mu_{\cdot, N}(\mathscr{C}_n) = \frac{t_n}{N} \right)\right] \nonumber \\
    \geq & \st \left[\sum_{\substack{(t_1,\ldots, t_n) \in [N]^n}} \left( \frac{t_1}{N}\right)^{k_1}\cdot \ldots \cdot \left( \frac{t_n}{N}\right)^{k_n} \cdot \mathbf{P}\left(\mu_{\cdot, N}(\mathscr{C}_1) = \frac{t_1}{N}, \ldots, \mu_{\cdot, N}(\mathscr{C}_n) = \frac{t_n}{N} \right) \right] \label{forward inequality}. 
\end{align}

But the reverse inequality to \eqref{forward inequality} is also true because of \eqref{less than} and the fact that ${^*}\mathbb{P}(\vec{\mathfrak{X}} \in \vec{A} \vert \vec{Y} = (t_1, \ldots, t_n))  = 0$ if $t_i < k_i$ for any $i \in [n]$. This completes the proof by \eqref{Bayes expansion'} and \eqref{Expectation expansion'}. 
\end{proof}

\section{Concluding the theorem of Hewitt and Savage from the theorem of Ressel}\label{appendix}
In this supplementary appendix, we prove that the theorem of Ressel showing Radon presentability of completely regular Hausdorff spaces (\cite[Theorem 3, p. 906]{Ressel-harmonic}) implies the theorem of Hewitt and Savage on the presentability of the Baire sigma algebra of compact Hausdorff spaces (\cite[Theorem 7.2, p. 483]{Hewitt-Savage-1955}) --- illustrating the sense in which sense Ressel's generalization (which also motivates our generalization) of de Finetti's Theorem directly superseded the generalization due to Hewitt and Savage, despite it having a slightly different appearance in its form. 

Since we will have occasion to talk about the presentability of Baire sigma algebras and Radon presentability in the same context, it is desirable to reduce the risk of confusion by introducing more precise notation for the relevant sigma algebras. 

\begin{notation}\label{Baire and Borel notation}
For a Hausdorff space $S$, let $\bas$ denote its Baire sigma algebra, the smallest sigma algebra with respect to which all continuous functions $f\co S \rightarrow \mathbb{R}$ are measurable). Let $\mathcal{B}(S)$ denote its Borel sigma algebra, the smallest sigma algebra containing all open subsets of $S$ (it is clear that $\bas \subseteq \mathcal{B}(S)$). Let $\prs$ denote the set of all Radon probability measures on $S$, and let $\pbas$ denote the set of all Baire probability measures on $S$. Let $\mathcal{C}(\prs)$ be the smallest sigma algebra on $\prs$ that makes all maps of the form $\mu \mapsto \mu(B)$ measurable, where $B \in \mathcal{B}(S)$. Let $\mathcal{C}(\pbas)$ be the smallest sigma algebra on $\pbas$ that makes all maps of the form $\mu \mapsto \mu(A)$ measurable, where $A \in \bas$. 
\end{notation}

Note that any compact Hausdorff space is normal (see, for example, Kelley \cite[Theorem 9, chapter 5]{Kelley}), and in particular completely regular. The key idea in going from Ressel's result to that of Hewitt--Savage is that on any completely regular Hausdorff space, a tight Baire measure has a unique extension to a Radon measure (see Bogachev \cite[Theorem 7.3.3, p. 81, vol. 2]{Bogachev-measure}). In particular, since every Baire measure on a $\sigma$-compact space is tight, it follows that every Baire measure on a completely regular $\sigma$-compact Hausdorff space admits a unique extension to a Radon measure on that space. See Bogachev \cite[Corollary 7.3.4, p. 81, vol. 2]{Bogachev-measure} for this result. Bogachev also has a formula for this unique extension on \cite[p. 78, vol. 2]{Bogachev-measure}. We record these facts as a lemma. 

\begin{lemma}\label{Bogachev lemma}
Let $S$ be a  completely regular $\sigma$-compact Hausdorff space. For a subset $A \subseteq S$, let $\tau_A(S)$ denote the collection of those open subsets of $S$ that contain $A$. For every $\mu \in \mathfrak{P}_{\text{Ba}}(S)$, there is a unique element $\hat{\mu} \in \prs$ such that  $\hat{\mu}(A) = \mu(A)$ for all $A \in \bas$. Furthermore, $\hat{\mu}$ is precisely given by the following formula:
\begin{align}\label{hat mu formula}
    \hat{\mu}(B) = \inf_{U \in \tau_B(S)} \sup_{\substack{A \in \bas \\ A \subseteq U}} \mu(A) \text{ for all } B \in \mathcal{B}(S).
\end{align}
\end{lemma}

As a consequence, we obtain the following lemma.

\begin{lemma}\label{Bogachev consequence}
Let $S$ be a  completely regular $\sigma$-compact Hausdorff space. Consider the map $~\hat{} \co \pbas \rightarrow \prs$ defined by $\hat{~}(\mu) = \hat{\mu}$ for all $\mu \in \pbas$ (where $\hat\mu$ is as in \eqref{hat mu formula}). Then $\hat{~}$ is a bijection. 

Furthermore, for a set $\mathcal{A} \in \mathcal{C}(\pbas)$, define $\hat{\mathcal{A}}$ to be its image under $\hat{}$ (thus $\hat{\mathcal{A}} \defeq \{\hat{\mu}: \mu \in \mathcal{A}\}$). Then $\hat{\mathcal{A}} \in \mathcal{C}(\prs)$ for all $\mathcal{A} \in \mathcal{C}(\pbas)$. 
\end{lemma}

\begin{proof}
If $\mu$ and $\nu$ are distinct elements of $\pbas$, then there exists an $A \in \bas$ such that $\mu(A) \neq \nu(A)$, which implies $\hat{\mu}(A) \neq \hat{\nu}(A)$, so that $\hat{\mu} \neq \hat{\nu}$. Thus $\hat{~}$ is an injection. That it is also a surjection follows from the fact that for any $\mu \in \prs$, its restriction $\mu \restriction_{\bas}$ to the Baire sigma algebra is a Baire measure that has a unique Radon extension by Lemma \ref{Bogachev lemma}, so that it must be the case that 
\begin{align}\label{surjection of hat}
    \mu = \widehat{{\mu \restriction_{\bas}}} \text{ for all } \mu \in \prs.
\end{align}

Consider the collection $\mathfrak{G}$ of sets $\mathcal{A} \in \mathcal{C}(\pbas)$ for which $\hat{\mathcal{A}}$ is an element of $\mathcal{C}(\prs)$, that is,
\begin{align}\label{Good hat set}
    \mathfrak{G} \defeq \{\mathcal{A} \in \mathcal{C}(\pbas): \hat{\mathcal{A}} \in \mathcal{C}({\prs})\}.
\end{align}

We want to show that $\mathfrak{G}$ equals $\mathcal{C}(\pbas)$. It is not very difficult to see that for any collection $ (A_n)_{n \in \mathbb{N}} \subseteq \mathcal{C}(\pbas)$, we have the following: $$\reallywidehat{\cup_{n \in \mathbb{N}} \mathcal{A}_n} = \cup_{n \in \mathbb{N}} \hat{\mathcal{A}_n}.$$

Hence, by the fact that $\mathcal{C}(\prs)$ is a sigma algebra, it follows that $\mathfrak{G}$ is closed under countable unions. Furthermore, if $\mathcal{A} \in \mathcal{C}(\pbas)$, then we have the following (the inclusion from left to right follows from the injectivity of $\hat{~}$, while the inclusion from right to left follows from the fact that $\hat{~}$ is a bijection): 
\begin{align}\label{hat of complement}
    \reallywidehat{\pbas \backslash \mathcal{A}} = \prs \backslash \hat{\mathcal{A}}.
\end{align}
This shows that $\mathfrak{G}$ is closed under complements as well. Since $\emptyset \in \mathfrak{G}$, it thus follows that $\mathfrak{G}$ is a sigma algebra. Thus by Dynkin's $\pi$-$\lambda$ theorem, it suffices to show that $\mathfrak{G}$ contains a $\pi$-system (that is, a collection of sets that is closed under finite intersections) that generates $\mathcal{C}(\pbas)$. A convenient $\pi$-system of that type is the following (that this is a $\pi$-system is trivial, and the fact that the smallest sigma algebra containing it coincides with $\mathcal{C}(\pbas)$ follows from the fact that any map on $\pbas$ of the type $\mu \mapsto \mu(A)$ for some $A \in \pbas$ is measurable on the former sigma algebra):  
\begin{align}\label{convenient pi system}
    \mathfrak{A} \defeq \{\mathfrak{A}_{C_1, \ldots, C_n}^{A_1, \ldots, A_n}: n \in \mathbb{N}, A_1, \ldots, A_n \in \bas \text{ and } C_1, \ldots, C_n \in \mathcal{B}(\mathbb{R})\},
\end{align}
where for any $n \in \mathbb{N}$, $A_1, \ldots, A_n \in \bas$ and $C_1, \ldots, C_n \in \mathcal{B}(\mathbb{R})$, the set $  \mathfrak{A}_{C_1, \ldots, C_n}^{A_1, \ldots, A_n}$ is defined as follows: 
\begin{align}
    \mathfrak{A}_{C_1, \ldots, C_n}^{A_1, \ldots, A_n} \defeq \{\mu \in \pbas: \mu(A_1) \in C_1, \ldots, \mu(A_n) \in C_n\}.
\end{align}

For $n \in \mathbb{N}$, consider the sets  $A_1, \ldots, A_n \in \mathcal{B}(S)$ and $C_1, \ldots, C_n \in \mathcal{B}(\mathbb{R})$. Define the collection $\mathfrak{B}_{C_1, \ldots, C_n}^{A_1, \ldots, A_n}$ as follows:
\begin{align}
   \mathfrak{B}_{C_1, \ldots, C_n}^{A_1, \ldots, A_n} \defeq \{\mu \in \prs: \mu(A_1) \in C_1, \ldots, \mu(A_n) \in C_n\} \in \mathcal{C}(\prs).
\end{align} 
It thus suffices to show the following claim.
\begin{claim}\label{claim for hat}
We have $\reallywidehat{\mathfrak{A}_{C_1, \ldots, C_n}^{A_1, \ldots, A_n}} = \mathfrak{B}_{C_1, \ldots, C_n}^{A_1, \ldots, A_n}$ for all $A_1, \ldots, A_n \in \bas$ and $C_1, \ldots, C_n \in \mathcal{B}(\mathbb{R})$.
\end{claim}
\begin{proof}[Proof of Claim \ref{claim for hat}]\renewcommand{\qedsymbol}{}
Note that for any $\mathcal{A}, \mathcal{B} \in \mathcal{C}(\pbas)$, we have the following (the inclusion from left to right is trivial, while the inclusion from right to left follows from the injectivity of the map $\hat{~}$):
$$\reallywidehat{\mathcal{A} \cap \mathcal{B}} = \hat{\mathcal{A}} \cap \hat{\mathcal{B}}.$$

Since $\mathfrak{A}_{C_1, \ldots, C_n}^{A_1, \ldots, A_n} = \cap_{i \in [n]} \mathfrak{A}_{C_i}^{A_i}$ and $\mathfrak{B}_{C_1, \ldots, C_n}^{A_1, \ldots, A_n} = \cap_{i \in [n]} \mathfrak{B}_{C_i}^{A_i}$, it suffices to show the following set equality:
\begin{align}\label{required set equality}
    \reallywidehat{{\mathfrak{A}_{C}^{A}}} = \mathfrak{B}_{C}^{A} \text{ for any } C \in \mathcal{B}(\mathbb{R}) \text{ and } A \in \bas. 
\end{align}

Toward that end, let $C \in \mathcal{B}(\mathbb{R}) \text{ and } A \in \bas$. If $\mu \in \mathfrak{A}_{C}^{A}$, then we have $\hat{\mu}(A) = \mu(A) \in C$, so that $\hat{\mu} \in \mathfrak{B}_{C}^{A}$. Thus the left side of \eqref{required set equality} is contained in the right side of \eqref{required set equality}. Conversely, if $\mu \in \mathfrak{B}_{C}^{A}$, then $\mu = \reallywidehat{\mu \restriction_{\bas}}$, where $\mu \restriction_{\bas} \in \mathfrak{A}_{C}^{A}$, completing the proof. 
\end{proof}
\end{proof}

As a corollary, we now have a way to define a natural measure on $\mathcal{C}(\pbas)$ corresponding to any measure on $\mathcal{C}(\prs)$ in the case when $S$ is completely regular, Hausdorff, and $\sigma$-compact. 

\begin{corollary}\label{inverse hat measure}
Let $S$ be a completely regular $\sigma$-compact Hausdorff space. Let $\hat{~} \co \pbas \rightarrow \prs$ be as in Lemma \ref{Bogachev consequence}. Suppose $\mathscr{P}$ is a probability measure on $\mathcal{C}(\prs)$. Define a map $\check{\mathscr{P}} \co \mathcal{C}(\pbas) \rightarrow [0,1]$ as follows:
\begin{align}\label{definition of inverse hat} 
    \check{\mathscr{P}}(\mathcal{A}) \defeq \mathscr{P}(\hat{\mathcal{A}}) \text{ for all } \mathcal{A} \in \mathcal{C}(\pbas).
\end{align}
Then $\check{\mathscr{P}}$ is a probability measure on $\mathcal{C}(\pbas)$. 
\end{corollary}
\begin{proof}
The fact that $\check{\mathscr{P}}$ is well-defined follows from Lemma \ref{Bogachev consequence}. Its countable additivity follows from that of $\mathscr{P}$ and the fact that the map $\hat{~}$ is injective. Finally, the fact that $\check{\mathscr{P}}(\pbas) = 1$ follows from the surjectivity of the map $\hat{~}$ (as we have $\reallywidehat{\pbas} = \prs$, whose measure with respect to $\mathscr{P}$ is one).  
\end{proof}

We are now able to show that the main result in Hewitt--Savage \cite{Hewitt-Savage-1955} is a direct consequence of the theorem of Ressel on the Radon presentability of completely regular Hausdorff spaces. 

\begin{theorem}[Hewitt--Savage {\cite[Theorem 7.2, p. 483]{Hewitt-Savage-1955}}] \label{reprove Hewitt and Savage}
Suppose all completely regular spaces are Radon presentable as in Definition \ref{Radon presentable definition}. Let $S$ be a compact Hausdorff space equipped with its Baire sigma algebra $\bas$. Suppose $(\Omega, \mathcal{F}, \mathbb{P})$ is a probability space and let $(X_n)_{n \in \mathbb{N}}$ be a sequence of exchangeable random variables (with respect to the Baire sigma algebra $\bas$). In other words, suppose the following holds:
\begin{align}
    \mathbb{P}(X_1 \in A_1, \ldots, X_k \in A_k) = \mathbb{P}(X_{\sigma(1)} \in A_1, \ldots, X_{\sigma(k)} \in A_k) \nonumber \\
    \text{ for all } k \in \mathbb{N}, \sigma \in S_k, \text{ and } A_1, \ldots, A_k \in \bas. \label{Baire exchangeability}
\end{align}
Then there is a unique probability measure $\mathscr{Q}$ on $\mathcal{C}(\pbas)$ such that
\begin{align}
\mathbb{P}(X_1 \in A_1, \ldots, X_k \in A_k) = \int\limits_{\pbas} \mu(A_1)\cdot \ldots \cdot \mu(A_k) d\mathscr{Q}(\mu) \nonumber \\
    \text{ for all } A_1, \ldots, A_k \in \mathcal{B}_a(S). \label{generalized de Finetti Baire sets}
\end{align}
\end{theorem}

\begin{proof}
We will only prove the existence of a probability measure $\mathscr{Q}$ on $\mathcal{C}(\bas)$ satisfying \eqref{generalized de Finetti Baire sets}, with uniqueness following more elementarily from Hewitt--Savage \cite[Theorem 9.4, p. 489]{Hewitt-Savage-1955}.

Since $S$ is compact Hausdorff, so is the countable product $S^{\infty}$ under the product topology (this follows from Tychonoff's theorem). Furthermore, Bogachev \cite[Lemma 6.4.2 (iii), p. 14, vol. 2]{Bogachev-measure} implies the following: 
\begin{align}\label{product Baire}
    \mathcal{B}_a(S^{\infty}) = \bigotimes \mathcal{B}_a(S),
\end{align}
where $\bigotimes \mathcal{B}_a(S)$ denotes the product sigma algebra on $S^{\infty}$ induced by the Baire sigma algebra $S$ (thus $\bigotimes \mathcal{B}_a(S)$ is the smallest sigma algebra on $S^{\infty}$ that makes the projection $\pi_i \co S^{\infty} \rightarrow S$ Baire measurable for each $i \in \mathbb{N}$). Let $\nu \in \pbasinf$ be the distribution of the $S^{\infty}$-valued Baire measurable random variable $(X_n)_{n \in \mathbb{N}}$ (the Baire measurability of this random variable follows from the Baire measurability of the $X_i$ together with \eqref{product Baire}).

Let  $\hat{~} \co \pbasinf \rightarrow \prsinf$ be as in Lemma \ref{Bogachev consequence}. Consider $\hat{\nu} \in \prsinf$. We show in the next claim that the Baire exchangeability of the sequence $(X_n)_{n \in \mathbb{N}}$ implies the exchangeability of the measure $\hat{\nu}$. In particular, let $\Omega' \defeq S^{\infty}$, $\mathcal{F}' \defeq \mathcal{B}(S^{\infty})$, and $\mathbb{P}' \defeq \hat{\nu}$. Consider the sequence of Borel measurable $S$-valued random variables $(Y_n)_{n \in \mathbb{N}}$ where, for each $n \in \mathbb{N}$, the map $Y_n \co \Omega' \rightarrow S$ is the projection onto the $n^{\text{th}}$ coordinate. Then we have the following claim:

\begin{claim}\label{extension to radon}
The sequence $(Y_n)_{n \in \mathbb{N}}$ is a jointly Radon distributed sequence of exchangeable random variables taking values in a completely regular Hausdorff space.
\end{claim}
\begin{proof}[Proof of Claim \ref{extension to radon}]\renewcommand{\qedsymbol}{}
The fact that $(Y_n)_{n \in \mathbb{N}}$ is a jointly Radon distributed sequence is immediate from the construction. Thus we only need to check the exchangeability of the $(Y_n)_{n \in \mathbb{N}}$ as Borel measurable random variables. 

To that end, suppose $k \in \mathbb{N}$ and $B \in \mathcal{B}(\mathbb{R}^k)$. Let $\psi \in \prsk$ be the Borel distribution of $(Y_1, \ldots, Y_k)$. That is, $\psi$ is the measure on $(\mathbb{R}^k, \mathcal{B}(\mathbb{R}^k))$ given by the pushforward $\mathbb{P}' \circ (Y_1, \ldots, Y_k)^{-1}$ (which is Radon, being the marginal of a Radon distribution on $S^{\infty}$). Let $\psi'$ be its restriction to the Baire sigma algebra on $S^k$---that is, $\psi' \defeq \psi \restriction_{\bask}$. Let $\sigma \in S_k$, and let $\psi_{\sigma}$ be the pushforward $\mathbb{P}' \circ (Y_{\sigma(1)}, \ldots, Y_{\sigma(k)}) \in \prsk$ induced by the permuted random vector $(Y_{\sigma(1)}, \ldots, Y_{\sigma(k)})$, with $\psi'_{\sigma} \defeq \psi_{\sigma}\restriction_{\bask}$ being its restriction to the Baire sigma algebra on $S^k$. It suffices to show that $\psi = \psi_{\sigma}$.

Note that for any $A \in \bask$, we have the following chain of equalities:
\begin{align}
    \psi'(A) &= \mathbb{P}' ((Y_1, \ldots, Y_k) \in A) \nonumber \\
    &= \hat{\nu}(A) \nonumber \\
    &= \nu(A) \nonumber \\
    &= \mathbb{P}((X_1, \ldots, X_k) \in A)  \nonumber \\
    &= \mathbb{P}((X_{\sigma(1)}, \ldots, X_{\sigma(k)}) \in A) \label{Borel to Baire exchangeability} 
    \\
    &= \mathbb{P}'((Y_{\sigma(1)}, \ldots, Y_{\sigma(k)}) \in A), \nonumber \\
    &= \psi_{\sigma}(A) \nonumber \\
    &= \psi'_{\sigma}(A). \label{permuted Baire}
\end{align}
In the above, equation \eqref{Borel to Baire exchangeability} follows from the Baire-exchangeability of $(X_1, \ldots, X_k)$, while the other lines follow from the fact that $A \in \bask$. 

Note that  by Lemma \ref{Bogachev lemma}, we have $\psi = \hat{\psi'}$ and $\psi_{\sigma} = \hat{\psi'_{\sigma}}$. By \eqref{hat mu formula}, we thus have the following for any $B \in \mathcal{B}(S^k)$ (where we use \eqref{permuted Baire} in the third line):
\begin{align*}
     \psi(B) &= \hat{\psi'}(B) \\
     &= \inf_{U \in \tau_B(S^k)} \sup_{\substack{A \in \mathcal{B}_a(S^k) \\ A \subseteq U}} \psi'(A) \\
    &= \inf_{U \in \tau_B(S^k)} \sup_{\substack{A \in \mathcal{B}_a(S^k) \\ A \subseteq U}} \psi'_{\sigma}(A) \\
    &= \hat{\psi'_{\sigma}}(A) \\
    &= \psi_{\sigma}(B) \text{ for all } B \in \mathbb{R}^k,
\end{align*}
which completes the proof of the claim. 
\end{proof}

Since completely regular Hausdorff spaces are Radon presentable, we obtain a unique Radon measure $\mathscr{P}$ on $(\prs, \mathcal{C}(\prs))$ such that the following holds:
\begin{align}\label{Ressel's conclusion}
    \mathbb{P}'(Y_1 \in B_1, \ldots, Y_k \in B_k) = \int\limits_{\prs} \mu(B_1)\cdot \ldots \cdot \mu(B_k) d\mathscr{P}(\mu) \nonumber \\
    \text{ for all } B_1, \ldots, B_k \in \mathcal{B}(S). 
\end{align}

Define $\mathscr{Q} \defeq \check{\mathscr{P}} \co \mathcal{C}(\pbasinf) \rightarrow [0,1]$ as in Lemma \ref{inverse hat measure}. We claim that $\mathscr{Q}$ satisfies \eqref{generalized de Finetti Baire sets}. Indeed, if $k \in \mathbb{N}$ and $A_1, \ldots, A_k \in \mathcal{B}_a(S)$, then we have:
\begin{align*}
    \mathbb{P}(X_1 \in A_1, \ldots, X_k \in A_k) &= \nu (A_1 \times \ldots \times A_k) \\
    &= \hat{\nu} (A_1 \times \ldots \times A_k) \\
    &=  \mathbb{P}'(Y_1 \in A_1, \ldots, Y_k \in A_k) \\
    &= \int\limits_{\prs} \mu(A_1)\cdot \ldots \cdot \mu(A_k) d\mathscr{P}(\mu) \\
    &= \int_{[0,1]} \mathscr{P}(\{ \mu \in \prs: \mu(A_1)\cdot \ldots \cdot \mu(A_k) > y \}) d\lambda(y) \\
    &= \int_{[0,1]} \mathscr{P}(\reallywidehat{\mathfrak{A}_y}) d\lambda(y), 
\end{align*}
where $$\mathfrak{A}_y \defeq \{ \mu \in \pbas: \mu(A_1)\cdot \ldots \cdot \mu(A_k) > y \}.$$

As a consequence, we have the following:
\begin{align*}
    \mathbb{P}(X_1 \in A_1, \ldots, X_k \in A_k) &= \int_{[0,1]} \check{\mathscr{P}}({\mathfrak{A}_y}) d\lambda(y) \\
    &= \int_{[0,1]} \mathscr{Q}(\{ \mu \in \pbas: \mu(A_1)\cdot \ldots \cdot \mu(A_k) > y \}) d\lambda(y) \\
    &= \int\limits_{\pbas} \mu(A_1)\cdot \ldots \cdot \mu(A_k) d\mathscr{Q}(\mu), 
\end{align*}
which completes the proof.
\end{proof}

\end{appendix}

\bibliography{References}
\bibliographystyle{amsplain}
\end{document}